\documentclass[11pt,reqno,tbtags]{amsart}

\usepackage[normalem]{ulem}
\usepackage{bbm}
\usepackage{amsthm,amssymb,amsfonts,amsmath}
\usepackage{amscd} %Commutative diagrams. 
\usepackage[numbers, sort&compress]{natbib} 
\usepackage{subfigure, graphicx}
\usepackage{vmargin}
\usepackage{setspace}
\usepackage{paralist}
\usepackage[pagebackref=true]{hyperref}
\usepackage{color}
\usepackage[normalem]{ulem}
\usepackage{enumerate}
\usepackage{enumitem}
\usepackage{stmaryrd} %for \bbr
\usepackage[refpage,noprefix,intoc]{nomencl} %for list of notation.
\usepackage[normalem]{ulem}
\providecommand{\T}{}
\providecommand{\R}{}
\providecommand{\Z}{}
\providecommand{\N}{}

\providecommand{\Q}{}

\renewcommand{\T}{\mathbb{T}}
\renewcommand{\R}{\mathbb{R}}
\renewcommand{\Z}{\mathbb{Z}}
\renewcommand{\N}{{\mathbb N}}

\renewcommand{\Q}{\mathbb{Q}}

\newcommand{\E}[1]{{\mathbb E}\left[#1\right]}		
\newcommand{\EN}[1]{{\mathbb E}_N\left[#1\right]}		
	
\newcommand{\Ezero}[1]{{\mathbb E}^{\,{0}}\left[#1\right]}	
\newcommand{\e}{{\mathbf E}}

\newcommand{\p}[1]{{\mathbb P}\left(#1\right)}
\newcommand{\px}[1]{{\mathbb P}^{\, x}\left(#1\right)}

\newcommand{\pN}[1]{{\mathbb P}_{ N}\left(#1\right)}

\newcommand{\pBig}[1]{{\mathbb P}\Big(#1\Big)}
\newcommand{\pbigg}[1]{{\mathbb P}\bigg(#1\bigg)}

\newcommand{\pzeroBig}[1]{{\mathbb P}^{\, {0}}\Big(#1\Big)}

\newcommand{\I}[1]{{\mathbbm 1}_{\{#1\}}}
\newcommand{\Itwo}[1]{{\mathbbm 1}_{#1}}
\newcommand{\Cprob}[2]{\mathbb{P}\left(\left. #1 \; \right| \; #2\right)} 
 
\newcommand{\CprobBig}[2]{\mathbb{P}\Big( #1 \; \Big| \; #2\Big)} 
\newcommand{\CprobBigg}[2]{\mathbb{P}\Bigg( #1 \; \Bigg| \; #2\Bigg)} 
\newcommand{\Cprobx}[2]{\mathbb{P}^{\,x}\left(\left. #1 \; \right| \; #2\right)}

\newcommand{\CprobzeroBig}[2]{\mathbb{P}^{\,{0}}\Big( #1 \; \Big| \; #2\Big)} 

\newcommand{\Cexp}[2]{\mathbb{E}\left[\left. #1 \; \right| \; #2\right]} 
\newcommand{\Cexpx}[2]{\mathbb{E}^{x}\left[\left. #1 \; \right| \; #2\right]} 
\newcommand{\CexpN}[2]{\mathbb{E}_{N}\left[\left. #1 \; \right| \; #2\right]} 
\newcommand{\Cexpzero}[2]{\mathbb{E}^{\,{0}}\left[\left. #1 \; \right| \; #2\right]}

\newcommand\cA{\mathcal A}
\newcommand\cB{\mathcal B}
\newcommand\cC{\mathcal C}

\newcommand\cE{\mathcal E}
\newcommand\cF{\mathcal F}
\newcommand\cG{\mathcal G}
\newcommand\cH{\mathcal H}

\newcommand\cK{\mathcal K}

\newcommand\cP{\mathcal P}

\newcommand\cU{{\mathcal U}}

\newcommand\cW{\mathcal W}

\newcommand\cY{{\mathcal Y}}

\newcommand{\bA}{\mathbf{A}}

\newcommand{\bW}{\mathbf{W}} 
\newcommand{\bX}{\mathbf{X}}

\newcommand{\ba}{\mathbf{a}}
\newcommand{\bx}{\mathbf{x}}
\newcommand{\bw}{\mathbf{w}}
\newcommand{\bh}{\mathbf{h}}
\newcommand{\bb}{\mathbf{b}}
\newcommand{\convdist}{\ensuremath{\stackrel{\!_{\mathrm{d}}}{\rightarrow}}}
\newcommand{\convp}{\ensuremath{\stackrel{\!_{\mathbb{P}}}{\rightarrow}}}

\providecommand{\eps}{}
\renewcommand{\eps}{\varepsilon}
\providecommand{\ora}[1]{}
\renewcommand{\ora}[1]{\overrightarrow{#1}}
\DeclareRobustCommand{\SkipTocEntry}[5]{} %For table of contents when using AMS styles. Change 5 to 4 if not using hyperref. 
\providecommand{\sc}{}
\renewcommand{\sc}{{\sf c}}
\providecommand{\sC}{}
\renewcommand{\sC}{{\sf C}}
\providecommand{\sk}{}
\renewcommand{\sk}{{\sf k}}
\providecommand{\sK}{}
\renewcommand{\sK}{{\sf K}}
\providecommand{\sQ}{}
\renewcommand{\sQ}{{\sf Q}}
\reversemarginpar
\definecolor{ccel}{rgb}{0.2,0.2,0.9}
\definecolor{cnel}{rgb}{0.3,0.7,0}

\newtheorem{thm}{Theorem}
\newtheorem{lem}[thm]{Lemma}
\newtheorem{prop}[thm]{Proposition}

\newtheorem{remark}[thm]{Remark}
\newtheorem{corollary}[thm]{Corollary}

\numberwithin{equation}{section}
\numberwithin{thm}{section}

\newcommand\numberthis{\addtocounter{equation}{1}\tag{\theequation}}
\makenomenclature										
\graphicspath{ {./images/} }
\usepackage{wrapfig}
\usepackage{todonotes}

\DeclareSymbolFont{extraup}{U}{zavm}{m}{n}
\DeclareMathSymbol{\varheart}{\mathalpha}{extraup}{86}
\DeclareMathSymbol{\vardiamond}{\mathalpha}{extraup}{87}
\newcommand{\ensymboldefinition}{$\blacktriangleleft$}

\newcommand{\eqn}[1]{\begin{equation} #1\end{equation}}

\newcommand{\prob}{\mathbb{P}}
\newcommand{\expec}{\mathbb{E}}
\newcommand{\invisible}[1]{}

\usepackage{hyperref}
\renewcommand{\e}{{\mathrm e}}
\newcommand{\sss}{\scriptscriptstyle}
\newcommand{\vep}{\varepsilon}
 
\begin{document}

\title[Condensation in scale-free geometric graphs]{Condensation in scale-free geometric graphs\\
with excess edges} 
%\author[van der Hofstad, van der Hoorn, Kerriou, Maitra, Mörters]{Remco van der Hofstad, Pim van der Hoorn,\\ C\'eline Kerriou, Neeladri Maitra, Peter Mörters}

\author[van der Hofstad]{Remco van der Hofstad}
\address{Department of Mathematics and Computer Science, Eindhoven University of Technology, 5612 AZ Eindhoven, The Netherlands}
\email{r.w.v.d.hofstad@tue.nl}
\author[van der Hoorn]{Pim van der Hoorn}
\address{Department of Mathematics and Computer Science, Eindhoven University of Technology, 5612 AZ Eindhoven, The Netherlands}
\email{w.l.f.v.d.hoorn@tue.nl}
\author[Kerriou]{C\'eline Kerriou}
\address{Universit\"at zu K\"oln, Department of Mathematics and Computer Science, Weyertal 86-90, 50931 Cologne, Germany}
\email{ckerriou@math.uni-koeln.de}
\author[Maitra]{Neeladri Maitra}
\address{Department of Mathematics and Computer Science, Eindhoven University of Technology, 5612 AZ Eindhoven, The Netherlands}
\email{n.maitra@tue.nl}
\author[Mörters]{Peter M\"orters}
\address{Universit\"at zu K\"oln, Department of Mathematics and Computer Science, Weyertal 86-90, 50931 Cologne, Germany}
\email{moerters@math.uni-koeln.de}
\address{}
\email{}

%\date{ (typeset \today{})} %; revised ...
%\urladdrx{}

%\keywords{<keywords>}
%\subjclass[2010]{60C05} 
%{60C05 (68P10,68W40)} %%{Primary: <subject>; Secondary: <subject>}
\begin{abstract} 
    We identify the upper large deviation probability for the number of edges in scale-free geometric random graph models as the space volume goes to infinity. Our result covers the models of scale-free percolation, the Boolean model with heavy-tailed radius distribution, and the age-dependent random connection model. In all  these cases the mechanism behind the large deviation is based~on a condensation effect. Loosely speaking, the mechanism randomly selects a finite number of vertices and increases their power, so that they connect to a macroscopic number of vertices in the graph, while the other vertices retain a degree close to their expectation and thus make no more than the expected contribution to the large  deviation event. We {verify this intuition by means of limit theorems for the empirical distributions of degrees and edge-lengths under the conditioning. We observe that at large finite volumes, the edge-length distribution splits} into a bulk and travelling wave part of asymptotically positive proportions.
\end{abstract}

\maketitle
\vspace{-.5cm}

\tableofcontents
\vspace{-1cm}
%%%%%%%%%%%%%%%
\section{Introduction and main results}\label{sec:intro} 

\subsection{Motivation}\label{sec:motivation}

The physical phenomenon of condensation corresponds to the formation of drops of liquid when a gas is cooled down or put under high pressure. 
The mathematics
behind this phenomenon however turns out not to be constrained to mathematical models of gases, it is ubiquitous throughout probability 
and the theory of stochastic processes, ranging from random walks \cite{BY19}, random trees~\cite{J12}, interacting particle systems~\cite{RCG18} and Gaussian processes~\cite{SM22} to queueing networks~\cite{MY96} and extreme value statistics~\cite{EM08}, to name just a few examples. 
In this paper we show how a condensation effect arises in the large deviation theory of scale-free geometric random graphs.
\pagebreak[3]\smallskip

To compare the condensation effects we look at the simplest probabilistic model capturing a condensation effect, which is the \emph{balls-in-bins} model, see~\cite[Sections~11 and~19]{J12}. In this model
$m$ particles (or balls) are placed into $n$ labelled containers (or bins) and the probability of a configuration $(m_1,\ldots,m_n)$ with $m_i$ particles in the $i$th 
container (and hence $\sum_{i=1}^n m_i=m$)~is
$$\frac1{Z_{m,n}} \prod_{i=1}^n w(m_i),$$
where $w(i)=ci^{-\beta}$ are the weights of a zeta
%(discrete) Pareto 
distribution with $1/c = \sum_{i= 1}^\infty i^{-\beta}$ and $Z_{m,n}$ is a normalisation constant. As $n\to\infty$ and $m/n\to\rho$ for some particle density $0<\rho<\infty$, it is shown in~\cite[Theorem~11.4]{J12}
that the proportion $N_k(n)$ of containers containing 
$k$ particles converges to a number $\pi_k$ with $\sum_{k=1}^\infty \pi_k=1$ and \smallskip
\begin{itemize}
    \item[$\rhd$]
     $\sum_{k=1}^\infty k\pi_k=\rho$ if $\rho\leq\nu:=\sum_{i=1}^\infty i w(i)$;\smallskip
     \item [$\rhd$] $\sum_{k=1}^\infty k\pi_k=\nu$ if $\rho>\nu$.\medskip
\end{itemize}
As $\sum_{k=1}^\infty kN_k(n)=m/n\to\rho$, we observe in the second case that particle mass is escaping. It turns out, see \cite[Theorem 19.34]{J12}, that the excess mass condenses in a \emph{single} container, in other words, the container containing the largest number of particles contains asymptotically $(\rho-\nu)n$ balls. The vastly increased particle density in this container is interpreted as a manifestation of the condensation phenomenon in a gas under high pressure.\medskip

In the balls-in-bins model, if  $(\xi_1,\ldots,\xi_n)$ are independent and zeta distributed, then 
$$Z_{m,n}=\p{\sum_{i=1}^n \xi_i=m}$$
and the probability of a configuration $(m_1,\ldots,m_n)$ is the conditional distribution of $(\xi_1,\ldots,\xi_n)$
given the  event 
$\sum_{i=1}^n \xi_i=m$. If $\rho>\nu$ this event has asymptotically decreasing probability and hence the condensation stems from conditioning on a rare event. In our main result a similar phenomenon arises in the context of large deviations for geometric random graphs. Here, loosely speaking, the bins correspond to vertices and the number of balls in a bin corresponds to the vertex degree. But other than in the balls-in-bins model, in the geometric random graphs we can have more than one vertex of macroscopic degree, depending on the number of excess edges generated in the large deviation event.\smallskip

A further analogy arises in the context of Bose-Einstein condensation. Bose and Einstein predicted in 1925 that above a certain density a macroscopic fraction of particles in a Bose gas aggregate into a single quantum state.   Based on Feynman's representation of the Bose gas as a soup of interacting loops in a container, see for example~\cite{Ue06}, the phenomenon takes the form that as the particle number goes to infinity, loops of macroscopic length occur in the loop soup. This is studied rigorously under simplifying modelling assumptions, for example in~\cite{BU09, AD21, DV22}, but a fully satisfactory mathematical result on a physical model is still missing. The results achieved so far show that in toy models, above a critical particle density, the empirical distribution of the loop lengths per particle in the limit as the volume of the container tends to infinity converges to a distribution of total mass strictly less than one. The missing mass escapes to infinity because a proportion of particles is associated with loops of a macroscopic length scale, see for example~\cite{QT23}. In our model of a random graph embedded in a $d$-dimensional torus of volume~$n$,  we look at the length of edges instead of loops and study the empirical distribution of the Euclidean length of all edges.  Conditioning the random graph on having an edge density above the critical value we prove that a positive fraction of the mass of the empirical distribution converges to a limit distribution of mass strictly smaller than one, while the rest of the mass forms a travelling wave of macroscopic speed and width. In particular, asymptotically,  a positive fraction of edge-lengths is of macroscopic order $n^{1/d}$, loosely analogous to the behaviour of loop lengths in the models above. 
\pagebreak[3]\smallskip

Beyond the relation to condensation phenomena, our results are of interest from a purely random graph perspective, providing insight in the nature of the most important models of scale-free random geometric graphs by giving precise asymptotics for the probability that such a graph has a large number of edges. We also give limit theorems for the 
empirical edge-length distribution and the empirical degree distribution of the graph conditioned on having a large number of edges, which expose the condensation behaviour. Results of this nature have been found for non-geometric graph models like the Chung-Lu graphs~\cite{SZ23}, very simple geometric models~\cite{KM23}, but also for Boolean models which are not scale-free, {see the seminal paper of Chatterjee and Harel~\cite{CH20}}. The behaviour of scale-free Boolean models, and more general scale-free random geometric graphs, is however fundamentally different from the latter and is explored in {the present paper}. Our results are set up in a general framework covering a wide range of scale-free geometric graphs.\medskip

{\textbf{Outline.}} We formulate  our framework of scale-free geometric graphs and give several examples in Section~\ref{sec:main}, before presenting the main results  in Section~\ref{sec:SFP_condensation}. 
In Section~\ref{sec:proofs_overview} we describe the proof strategies behind our theorems and reduce the proof to seven key propositions, which are proved in Sections~\ref{sec:proof_concentration_bulk}
to~\ref{sec:degree_distribution}. \medskip

\textbf{Notation.} We write \smash{$X_n \convp X$}, respectively \smash{$X_n \convdist X$}, to denote that $X_n$ converges in probability, respectively, in distribution, to $X$ as $n\to\infty$. We write $\cB(x,r)$ to denote the open ball of radius $r$ centred at $x$ in a given metric space. 
For real sequences $(x_n)_{n\geq 1}$ and $(y_n)_{n\geq 1}$ with $(y_n)_{n \geq 1}$ being positive, we write $x_n = o(y_n)$ if $\limsup_{n\to \infty} x_n/y_n =0$ and  $x_n = O(y_n)$ if $\limsup_{n\to \infty} x_n/y_n < \infty$.

\pagebreak[3]

\subsection{Geometric random graph models}\label{sec:main} 

In our general setup, the vertex set of the graph $G_n$ is contained in the $d$-dimensional torus $\T_n^d$ of volume $n$, which we equip with the torus metric $d_n$. The vertex set $V_n$ may be of two different types:
\begin{itemize}
 \item[$\rhd$] \emph{Lattice case:} {$n^{\frac1d}$ is an integer and} $V_n$ the integer lattice on the torus, 
 %represented as  $$[-n^{1/d}/2, n^{1/d}/2)^d \cap \Z^d,$$  
 \item[$\rhd$] \emph{Poisson case:} {$n$ is real and} $V_n$ a Poisson point process of intensity one on~$\T_n^d$.
\end{itemize}
Fix $\beta>2$. To define the edge set $E_n$, given the vertex set $V_n$, we equip each vertex $x\in V_n$ with a random weight $W_x$ such that the random variables $(W_x)_{x\in V_n}$ are independent and identically distributed to some $W$ with density 
\begin{align*}
  f_W(t):=  (\beta -1)t^{-\beta}\I{t\geq 1}. \numberthis \label{eq:Pareto_density}
\end{align*}
We fix a \textit{profile function} $\varphi\colon [0,\infty) \rightarrow [0,1]$, which is nonincreasing and integrable,
% and normalized function.
and a \textit{connection kernel}
$\kappa\colon [1,\infty)\times  [1,\infty)\rightarrow (0,\infty)$,
which is symmetric, continuous and nondecreasing in each argument. For every unordered pair 
$\{x,y\} \subseteq \mathbb T_n^d$, we define a function $p_{x,y}\colon [1,\infty) \times [1,\infty) \to [0,1]$ by $p_{x,y}:=0$ if $x=y$, and
\begin{align*}
    p_{x,y}(v, w) := \varphi\left(\frac{d_n(x,y)^d}{\kappa(v,w)}\right) \text{ otherwise.} \numberthis \label{eq:connection_func}
\end{align*}
Given the vertex set and the weights, we include the edge $\{x,y\}$ in $E_n$ with probability $p_{x,y}(W_x, W_y)$, independently for every unordered pair $\{x,y\}$ of distinct vertices in $V_n$ with weights $W_x, W_y$. More precisely, let $(A_{\{x,y\}})_{x,y}$ be conditionally independent Bernoulli-distributed random variables with success parameter $p_{x,y}(W_x, W_y)$. For notational simplicity write $A_{x,y}$ for $A_{\{x,y\}}$. The edge set $E_n$ is given by 
\begin{align}\label{def:edge_set}
    E_n := \{\{x,y\} \subseteq V_n \colon A_{x,y} = 1\}.
\end{align}
The degree of a vertex $x \in V_n$ is the number of neighbours of $x$ in $G_n$  defined as
\begin{align}\label{def:degree}
    D_x := D_x(n) := \sum_{y\in V_n} A_{x,y}.
\end{align}
%Observe that $(D_x(n))_{x\in V_n}$ are identically distributed since the weights are i.i.d.
This model was introduced under the name \emph{weight-dependent random connection model} in~\cite{GHMM, GLM} and under the name \emph{kernel-based spatial random 
graph} in~\cite{JKM}. We now explain how our principal examples fit in this 
framework. In all our examples the empirical degree distributions of the graphs converge to a deterministic limit, 
which is a heavy-tailed distribution with the same tail index $\beta$ as the weight distribution.\medskip
\paragraph{\bf (i) Scale-free percolation}\label{ex:scale_free_percolation}
Take the profile function $\varphi(x)=1-\e^{-x^{-\alpha}}$, for some $\alpha>1$, and connection kernel $\kappa(v,w)=vw$. For the lattice case this model was first introduced and studied in~\cite{DHH13}, the Poisson case has been introduced in~\cite{DW19}. 
%We write $\mu_n = \E{D_{\mathbf{0}}(n)}/2$ for the expected degree of the origin $\mathbf{0} \in V_n$. Let $G_{\infty} = (\Z^d, E(G_{\infty}))$ be the scale-free percolation on $\Z^d$ and denote the expected degree of the origin by $\mu/2 = \E{D_{{\mathbf{0}}}(\infty)}$. Then $\E{|E_n|} = n\mu_n$ and $\lim_{n\rightarrow \infty} \E{|E_n|}/n = \mu$.
\medskip

\paragraph{\bf (ii) Poisson Boolean model with heavy-tailed radii}\label{ex:poisson_boolean}
Consider the Poisson case with profile function 
$\varphi(x)=\Itwo{[0,1]}(x)$ and connection kernel 
\smash{$\kappa(v,w)=(v^{1/d}+w^{1/d})^d$.} Thinking of 
\smash{$r_x:=W_x^{1/d}$} as the radius of a ball $\cB(x,r_x)$ centred in $x$, in this model two vertices $x,y$ are connected by an edge if and only if $\cB(x,r_x) \cap \cB(y,r_y)\not=\emptyset$. Therefore this graph model shares the name with the union of the balls $\cB(x,r_x)$ over all points $x\in V_n$, which is studied in stochastic geometry~\cite{Sto13}.
%The average degree converges to a limiting constant, which is the degree of $0$ in the natural infinite Palm version of the model. 
\medskip%

\paragraph{\bf (iii) Age-based preferential attachment graphs}\label{ex:age_based_pref_att}
We define a dynamical graph model $(G^t \colon t\ge0)$ with vertices embedded in the torus $\T_1^d$ of volume one as in~\cite{GGLM}. Let $\varphi(x)=1\wedge x^{-\alpha}$, for some $\alpha>1$, and fix $0<\gamma<1$. We take a standard Poisson process of arrival times for the vertices. Given the graph $G^{t-}$, at the arrival time~$t$ of a new vertex, we 
give this vertex a random location $x\in \T_1^d$ uniformly on the torus of volume one. Independently for every vertex of  $G^{t-}$, we establish an edge to the new vertex with probability
\smash{$\varphi(t d_1(x,y)^d (s/t)^\gamma),$}
where $y$ is its location and $s<t$ its birth-time.  As $(t/s)^\gamma$ is the order of the expected degree of the vertex at time $t$, the new vertex prefers to attach itself to existing vertices with large expected degree or, equivalently, old age. \medskip

For any fixed time $n$, we rescale the graph $G^n$ to the torus $\T_n^d$ by multiplying all locations with the factor $n^{1/d}$. We denote the rescaled graph by $G_n$ and explain how it fits our framework. If $s\in(0,n)$ is the birth time of a vertex at the new location~$x$, then we let its weight be $W_x:=(s/n)^{-\gamma}$. Then the new vertex locations are a standard Poisson point process on $\T_n^d$, the weights are independent Pareto-distributed with parameter \smash{$\beta=1+\frac1\gamma$}, and the probability of two vertices at locations $x,y$  in $G_n$ being connected is
$$\varphi\Big(n\big(W_x^{-1/\gamma}\vee W_y^{-1/\gamma}\big)  \, d_1\big(xn^{-1/d},yn^{-1/d}\big)^d \big(\tfrac{W_x^{-1/\gamma}\vee W_y^{-1/\gamma}}{W_y^{-1/\gamma}\wedge W_x^{-1/\gamma}}\big)^\gamma\Big)
= \varphi\Big( \frac{d_n(x,y)^d}{\kappa(W_x,W_y)}\Big),$$
for $\kappa(v,w)= (v\wedge w)^{\frac1\gamma-1} (v\vee w)$. Hence results for the number of edges in $G^n$ as \emph{time} goes to infinity can be transferred directly from results for the number of edges for the graph $G_n$ as the \emph{volume} tends to infinity. The latter fits our framework with the given $\kappa$.
%\smallskip

\subsection{Main results}\label{sec:SFP_condensation} 
We now formulate the assumptions on the kernel~$\kappa$
and profile function~$\varphi$ in our framework
for our main results to hold.\smallskip

\noindent
{\bf Assumption A}. \label{assumption_a}
{\rm There exist constants $0<{\sf c}<{\sf C}<\infty$ such that 
\begin{equation}\label{asspt:kernel_bounds_var}
{\sf c}v \leq \kappa(v,w) \leq
{\sf C} v w^{(\beta-2) \vee 1}, \text{ for all } v \geq w\geq1,
\end{equation}
and the limit
\eqn{
\label{asspt:limiting_kernel}
\cK(v,w) := \lim_{x\uparrow \infty} \frac1x \kappa(xv,w)
}
exists. 
The profile function~$\varphi$ is continuous at $0$ and, for some $\alpha>1$,
\begin{equation}\label{asspt:profile_bounds}
{\sf c}  \I{x\leq {\sf c}} \leq \varphi(x) \leq
1 \wedge {\sf C} x^{-\alpha}, \text{ for all }x\geq 0.
\end{equation}
The graph of $\varphi$ has no flat pieces except possibly at its essential supremum or infimum.}  \hfill\ensymboldefinition
\smallskip

Observe that, in contrast to $\kappa$, the derived kernel $\cK$ is not symmetric, but linear in the first argument.
Recall that $|E_n|$ denotes the total number of edges in $G_n$. 
Then we have, as $n \to \infty$, by an application of a law of large numbers for geometric functionals, 
%\cmar{should it be $|E_n|/|V_n|$?}
%\pmar{It is equivalent but to prepare the main statement $n$ is more relevant.}
\begin{align*}
    \frac{|E_n|}{n} = \frac1{2n} \sum_{x\in V_n} D_x(n) 
\stackrel{\mathbb{P}}{\to} \mu, \numberthis \label{eq:LLN_edge}
\end{align*}
for some $0<\mu<\infty$. More details and an explicit formula for $\mu$ are given in Remark~\ref{rem:explicit_mu} below. 
\pagebreak[3]

For $z\in [-\tfrac{1}{2},\tfrac{1}{2}]^d$ and $w> 0$, let
\begin{align}\label{def:lambda_refined}
    {{\Lambda}}(w,z):=\E{\varphi\left(\frac{\|z\|^d}{\cK(w,W)}\right)},
\end{align}
and 
\begin{align}\label{def:lambda}
    {{\Lambda}}(w):=\int_{\big[-\tfrac{1}{2},\tfrac{1}{2}\big]^d}\Lambda(w,z) \, dz .
\end{align}
Further, for $\rho>0$ set $k:=\lceil \rho \rceil$ and
\begin{align}\label{def:f_rho_function}
    F(\rho):=\frac{(\beta-1)^k}{{k!}} \int_{0}^{\infty}\cdots\int_{0}^{\infty}\I{{\Lambda}(y_1)+\cdots+{\Lambda}(y_k)>\rho}\prod_{i=1}^k y_i^{-\beta}dy_i.
\end{align}

Our main result for the large deviations of the number of edges in the random graph models $(G_n)_{n\in\N}$ described in Section~\ref{sec:main} is summarized in the following theorem. 
\begin{thm}[Upper large deviations]\label{thm:main_uldp}
Let $|E_n|$ be the number of edges in the weight-dependent random connection model $G_n$ with kernel $\kappa$
%satisfying Assumption~\eqref{asspt:kernel_bounds_var} 
and profile $\varphi$ satisfying %Assumption~\eqref{asspt:profile_bounds}.
{Assumption~
\hyperref[assumption_a]{A}.}
Let $\rho >0$ be non-integer and $k$ the unique integer such that $k-1< \rho <k$. 
%\pmar{Preferable to $k=\lceil \rho \rceil$ to emphasise $\rho$ is non-integer.}
Then, as $n\rightarrow \infty$, 
    \[\p{|E_n| \geq n(\rho + \mu)} = \big(F(\rho)+o(1)\big) n^{-k(\beta -2)}.\] 
\end{thm}
\smallskip
%\RvdH{Add a remark that $F(\rho)<\infty$, with proof.}

\begin{remark}[Properties of $\rho\mapsto F(\rho)$]
{\rm We always have $F(\rho)<\infty$. If $\sup_{x>0} \varphi(x)=1$, then $F(\rho)>0$  and
%$\lim_{y\uparrow\infty} \lambda(y)=1$, which ensures that  
%the assumption that $\varphi$ has no flat pieces also implies that 
$F$ is a continuous function on all open intervals not containing an integer; see Lemma~\ref{lem:continuity} for a proof. 
Therefore, under this assumption, Theorem~\ref{thm:main_uldp} identifies the precise asymptotic behaviour of the large deviation probability. Otherwise, if $\sup_{x>0} \varphi(x)=p<1$,
the graph can be identified with the graph with profile function $\varphi/p$ after Bernoulli percolation with retention probability $p$. Given the graph prior to percolation and denoting its edge set by $E^p_n$, the number $|E_n|$ of edges in the percolated graph is binomial with parameters $|E^p_n|$ and $p$. By Cram\'er's theorem this random variable has the same upper large deviations as its conditional mean $|E^p_n|p$
and so the upper large deviation result of Theorem~\ref{thm:main_uldp} holds for $p(k-1)<\rho <pk$ verbatim and also identifies the precise asymptotic behaviour.}
%\[\p{|E_n| \geq n(\rho + \mu)} = \big(F(\rho)+o(1)\big) \binom{n}{k}n^{-k(\beta -1)}.\]
\hfill\ensymboldefinition
\end{remark}

\begin{remark} [Our assumptions]
{\rm %Assumptions~\eqref{asspt:kernel_bounds_var} and~\eqref{asspt:profile_bounds} 
{Assumption~
\hyperref[assumption_a]{A} }
holds in our examples \hyperref[ex:scale_free_percolation]{(i)}- \hyperref[ex:age_based_pref_att]{(iii)}. Also, in all our examples we have $\sup_{x>0} \varphi(x)=1$, which ensures that Theorem~\ref{thm:main_uldp} identifies the asymptotic behaviour of the large deviation probability. }
\hfill\ensymboldefinition
\end{remark}

\begin{remark}[Non-integer restriction on $\rho$]
{\rm Observe that the exponent of $n$ in the power-law decay of the large deviation probability in Theorem~\ref{thm:main_uldp} jumps at integer values of~$\rho$. At these values the asymptotic behaviour of $\p{|E_n| \geq n(\rho + \mu)}$ depends on specific model details and a universal result covering a wide range of profiles and kernels in a single limit theorem with explicit limit as above does not hold. A deeper discussion of this fact is given in Section~\ref{sec:conc} below.}
\hfill\ensymboldefinition
\end{remark}

%For our next results 
From now on
we assume, without loss of generality, that $\sup_{x>0} \varphi(x)=1$ and consider the behaviour of the graph conditioned on the large deviation event above.
% We write \smash{$X_n \convp X$}, respectively \smash{$X_n \convdist X$}, to denote that $X_n$ convergs in probability, respectively, in distribution, to $X$ as $n\to\infty$.

\begin{remark}[Conditional limit distribution of the number of edges]\label{rem:cond_lim_num_edges}
{\rm It follows immediately from Theorem~\ref{thm:main_uldp} that, conditionally on 
$|E_n|\geq n(\mu+\rho)$,
    \eqn{
    \label{lim-edges-weak}
    \frac{|E_n|}{n}\convdist \mu +S,
    }
where the random variable $S$ is supported on $[\rho,k]$, and has tail distribution function
    \eqn{
    \label{lim-edges-weak-law}
    \prob(S>s)=\frac{F(s)}{F(\rho)} \quad \mbox{ for all $\rho\leq s<k$.}
    }
This follows as $F$ is continuous on $(k-1,k)$. Observe that the distribution of $S$
has an atom at $k$ if and only if $\lim_{\rho \nearrow k} F(\rho)>0$, which happens, for example, in the case of the Poisson Boolean model, Example~\hyperref[ex:poisson_boolean]{(ii)}, in which $\Lambda$ is constant on the interval \smash{$[\frac{\sqrt{d}}2,\infty)$}.}
\hfill\ensymboldefinition
\end{remark}

% \begin{thm}\label{thm:local_non_int}[{This hasn't been proven yet!}]
%     Let $\rho >0$ be non-integer and pick the unique $k = k(\rho) \in \N$ such that $k-1<\rho<k$. 
%     Then {(define $\rho_1$,$\rho_2$, $I_n(\rho_1(n),\rho_2(n)) := [n(\rho_1(n) + \mu), n(\rho_2(n) + \mu)]$)}
%     \[\p{ 2|E_n| \in  I_n(\rho_1(n),\rho_2(n))} = (1+o(1)){n\choose{k}}(\rho_2(n)- \rho_1(n))n^{-k(\beta-1)}K_{\rho}.\]
% \end{thm}
 Next we describe the empirical distribution of the edge-lengths, which is the (random) measure $\mu_n$ on $(0,\infty)$ given by
    \begin{equation}\label{eq:emp_EL_dist}
    \mu_n:=\frac{1}{n} \sum_{\{x,y\}\in E_n} \delta_{d_n(x,y)},
    \end{equation}
where $\delta_a$ is the unit mass 
%Dirac delta 
at $a\in \R$. Now we introduce the objects appearing in the limit theorem below. For $(y_1,\ldots, y_k)\in(0,\infty)^k$, define a measure
$\mu_{y_1,\ldots, y_k}$ on $\R$ by
\begin{align}\label{eq:lim_meas}
\int f d\mu_{y_1,\ldots, y_k}
=\int_{[-\frac12,\frac12]^d } f(\|z\|) \sum_{i=1}^k {{\Lambda}}(y_i,z) \, dz.
\end{align}
Consider the random measure
$\mu_{Y_1,\ldots, Y_k}$ where the law of $(Y_1,\ldots, Y_k)$
 is given by
\begin{align}\label{eq:prob_y}
    &\p{Y_1\ge x_1,\ldots, Y_k \ge x_k}
    \\
    &\qquad = \frac{(\beta-1)^k}{F(\rho)}  \int_0^\infty\!\!\cdots\!\!\int_0^\infty 
    \I{\sum_{i=1}^k{{\Lambda}}(y_i)>\rho}\prod_{i=1}^k \I{ y_i \ge x_i\vee y_{i+1}}y_i^{-\beta}dy_i,\nonumber
\end{align}
setting $y_{k+1}=0.$ 
In terms of this notation, $S$ in \eqref{lim-edges-weak}--\eqref{lim-edges-weak-law} has distribution       
\eqn{
S=\sum_{i=1}^k \Lambda(Y_i).
    }
For $z\in \R^d$ and $w\geq 1$, let 
\begin{align}\label{def:lambda_w_z}
    \lambda(w,z) & := \E{\varphi\left(\frac{\|z\|^d}{\kappa(w,W)}\right)}.
\end{align}
Let  $f\colon (0,\infty)\to [0,\infty)$ be continuous 
%with compact support 
and define, for $w\geq 1$,
\begin{align}\label{def:lambda_over_rd}
%    \overline{\Lambda}_f(w) & = \begin{cases}
     \lambda_f(w) & = \begin{cases}
           \sum_{z\in \Z^d}  f(\|z\|) \,
    {{\lambda}}(w,z) & \mbox{ in the lattice case},  
     \\[3mm]
      \int_{\R^d}  f(\|z\|) \,
    {{\lambda}}(w,z) \, dz & \mbox{ in the Poisson case}.  
    \end{cases}
\end{align}

\begin{thm}[Convergence of empirical edge-length distribution]
\label{thm:distances} 
Let $\mu_n$ be the empirical edge-length distribution in \eqref{eq:emp_EL_dist} of
the weight-dependent random connection model with kernel $\kappa$
and profile $\varphi$ satisfying Assumption~
\hyperref[assumption_a]{A}.
Let $\rho >0$ be a non-integer and $k$ be the unique integer with $k-1<\rho<k$.
For every $f\colon (0,\infty) \to [0,\infty)$ that is continuous with compact support, conditionally on 
$|E_n|\geq n(\mu+\rho)$,\\[-1mm]
%as~$n\to\infty$,

\begin{enumerate}[label=(\alph*)]
  \item \label{eq:conv_prob_bulk}
  $\displaystyle
        \int f\big(x\big) \, d\mu_n(x) \convp \frac12 \mathbb E \big[{\lambda}_f(W) \big];$\smallskip
    \item \label{eq:conv_distr_condensates} 
   $\displaystyle
    \int f\Big( \frac{x}{n^{1/d}}\Big) \, d\mu_n(x)
    \convdist \int f \, d\mu_{Y_1,\ldots, Y_k}.
    $
\end{enumerate}
\end{thm}
\smallskip

\begin{remark}[Law of large numbers for edge functionals]\label{rem:explicit_mu}
{\rm The convergence in probability in $\ref{eq:conv_prob_bulk}$ also holds when not conditioning on  $|E_n|\geq n(\mu+\rho)$, in this case even for all bounded and continuous functions $f$. 
 A proof based on a local limit theorem, given for a slightly different vertex set, can be found in \cite{LWC_SIRGs_2020} and is easily adapted to our situation. Therefore 
the right-hand side in $\ref{eq:conv_prob_bulk}$, with the choice $f = 1$, provides an explicit formula for the asymptotic edge density~$\mu$ that appears in \eqref{eq:LLN_edge}. }
\hfill\ensymboldefinition
\end{remark}
\pagebreak[3]
\smallskip

%\begin{remark}
%\pmar{This is wrong (your choice of $f$ is not compactly supported). The true fact is already in Remark 1.5.}
%    {\rm In \ref{eq:conv_distr_condensates}, choosing $f \equiv 1$ gives $|E_n| \convdist \int d\mu_{Y_1,\ldots, Y_k} = \sum_{i=1}^k\Lambda(Y_i)$.}\hfill\ensymboldefinition
%\end{remark}
%\smallskip
Theorem~\ref{thm:distances} shows that the empirical edge-length distribution splits into two 
%non-degenerate 
parts, a \emph{bulk} part of total mass $\mu$ which converges to the asymptotic edge-length distribution in the unconditional random graph,
and a \emph{travelling wave} part of mass at least $\rho$ and random shape, moving to infinity at speed $n^{1/d}$, see Figure~\ref{fig:edges} for a schematic picture.\pagebreak[3]

\begin{figure}[h]
    \centering
\includegraphics[width=7cm]{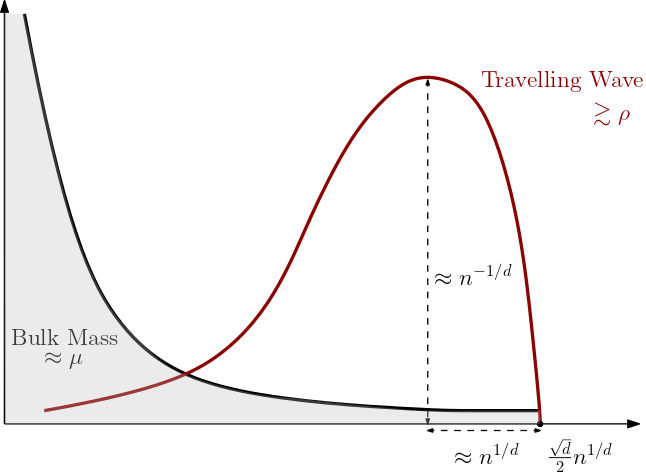}
    \caption{The schematic figure shows the density of the empirical edge-length distribution, which splits into a bulk of total mass $\mu$ and a travelling wave of total mass at least $\rho$, which has speed and width of order $n^{1/d}$. The asymptotic shape of bulk and wave are given explicitly in Theorem~\ref{thm:distances}.}
    \label{fig:edges}
\end{figure}

To state the next theorem, 
denote the order statistics of the weights of  vertices in~$V_n$~by $$W_{\sss (1)}>W_{\sss (2)}>\dots>W_{\sss (|V_n|)}.$$ Let $X_{\sss (i)}$ be the location of the vertex whose weight is $W_{\sss(i)}$, that is \smash{$W_{X_{(i)}}=W_{\sss(i)}$}. %Recall $k = \lceil \rho \rceil$.
For a vertex $x \in V_n$, define
\begin{align*}
    D^{\bb}_x:=\sum_{j > k} A_{x, X_{(j)}}, \qquad \text{ and } \qquad D^{\bh}_x:=\sum_{j=1}^k A_{x,X_{(j)}}. \numberthis \label{eq:def_two_degrees}
\end{align*}
We consider the joint empirical distribution of these two types of degrees. For $a \in \N\cup \{0\}$ and $b \in \{0,1,\dots,k\}$, define
\begin{align*}
    \pi_{a,b}^{\sss(n)}:=\frac{1}{|V_n|}\sum_{x\in V_n}\I{D_x^{\bb}=a,D_x^{\bh}=b}. \numberthis \label{eq:joint_emp_deg}
\end{align*}
In words, \smash{$\pi_{a,b}^{\sss(n)}$} is the proportion of vertices that have exactly $a$ neighbours in the set $V_n \setminus \{X_{(1)},\dots,X_{(k)}\}$, and $b$ neighbours in the set $\{X_{(1)},\dots,X_{(k)}\}$. Let $(Y_1,\dots, Y_k)$ be a sample from the law \eqref{eq:prob_y}.
%, and let $$Y^{\sss(1)}>Y^{\sss(2)}>\dots>Y^{\sss(k)}$$ be the corresponding order statistics. 
We introduce the local limit $G_{\sss\infty}$ of the graph $G_n$, which is based on the same connection probabilities but with an infinite vertex set. For the lattice case, $G_{\infty}$ is obtained by considering as vertex set all the points of the lattice $\Z^d$, 
for the Poisson case $G_{\sss\infty}$ is constructed by taking as vertex set a Poisson point process of intensity one on $\R^d$, with an atom inserted at $0 \in \R^d$. In both cases, the infinite random graph, rooted at $0 \in \R^d$ arises as the local (Benjamini-Schramm) limit in probability of $G_n$, see \cite[Theorem 1.11]{LWC_SIRGs_2020}. Denote the degree of the root in $G_{\sss\infty}$ by $D_{\sss \infty}$ and the degree of the root in $G_{\sss\infty}$ conditioned on having weight $w$ by $D_{\sss \infty}(w)$. 
\medskip

Our third main result describes the degree distribution of the graph $G_n$ conditioned on the event $|E_n|\geq n(\mu+\rho)$. We start with a version describing degree sequences of hubs and other vertices separately.

\begin{thm}[Degree sequences of hubs and normal vertices]\label{thm:EDD}
%Let $|E_n|$ be the number of edges in 
Consider the weight-depen\-dent random connection model with kernel $\kappa$ and profile $\varphi$ satisfying Assumption~
\hyperref[assumption_a]{A}. Let $\rho>0$ be a non integer, and $k$ be the unique integer satisfying $k-1<\rho<k$. Conditionally on the rare event $|E_n|\geq n(\mu+\rho)$, as $n\rightarrow\infty$,
 \begin{align*}
        \left(\tfrac{1}{n}(D_{X_{\sss(i)}})_{i \in [k]}, (\pi_{a,b}^{\sss(n)})_{0\leq a, 0\leq b \leq k} \right) \stackrel{d}{\to} \left((\Lambda(Y_i))_{i \in [k]},(\pi_{a,b})_{0\leq a, 0\leq b \leq k} \right), \numberthis \label{eq:joint_conv_high_degrees_edd}
    \end{align*}
in distribution on sequence space with the product topology, while
\begin{align}\label{eq:scaled_low_weight_vertices}
    \tfrac1{n} {D_{X_{\sss (i)}}} \convp 0 \qquad \mbox{ for each $i>k$,}
\end{align}
where 
\begin{align*}
    \pi_{a,b}=\pi_{a,b}((U_i,Y_i)_{i \in [k]}):=\mathbb E\Big[ D_{\sss \infty}(W)=a, \sum_{i=1}^kB^{\sss{(i)}}_{U,W}=b \, \Big| \, 
    (U_{i}, Y_i)_{i \in [k]}\Big] ,\numberthis \label{eq:def_pi_ab}
\end{align*}
and where {the vector $(Y_1,\dots,Y_{k})$ is chosen according to the law given by (\ref{eq:prob_y})} and $U_{1}, \ldots, U_{k}$ are i.i.d.\ uniformly distributed random variables 
on $\mathbb T_1^d$
independent of $(Y_1,\ldots, Y_k)$. Moreover, the expectation is over $U$ which is uniformly distributed on $\mathbb T_1^d$,
and an independent Pareto-distributed weight $W$ and, for $u \in\mathbb T_1^d$ and $w\ge1$, over conditionally independent Bernoulli random variables $B^{\sss (1)}_{u,w}, \ldots, B^{\sss (k)}_{u,w}$  with parameter 
    $$
    \varphi\left(\frac{d_1(u,U_{i})^d}{\cK(Y_i,w)} \right).
    $$ 
\end{thm}

Combining the results in Theorem~\ref{thm:EDD}, we get convergence of the degree distribution.

\begin{corollary}
[Degree distribution]
    Conditionally on the rare event $|E_n|\geq n(\mu+\rho)$, the empirical degree distribution
    $$\frac{1}{|V_n|}\sum_{x\in V_n}\delta_{D_x}$$
    converges in distribution on the Polish space of probability measures with the weak topology~to the random probability measure
    $$\sum_{a=0}^\infty \sum_{b=0}^k \pi_{a,b} \, \delta_{a+b}\,.$$
\end{corollary}

\begin{remark}[The condensate]\label{rem:condensate}
{\rm  Theorem~\ref{thm:EDD} identifies where the surplus edges, whose limit in distribution equals $S$ by 
Remark~\ref{rem:cond_lim_num_edges}, are located. Indeed, using that the sum of the degrees equals $2|E_n|$, we see that the vertices of high-weight of order $n$ form the condensate attracting roughly $nS$ incident edges. The other ends of the edges are at vertices whose weights have the unconditional weight distribution $W$. Thus, the condensate consists of $k$ vertices, but other vertices can be incident to this condensate. This makes the condensation phenomenon different, and arguably richer, than the one described in Section~\ref{sec:motivation}.
}
\hfill\ensymboldefinition
\end{remark}

\section{Overview and proof strategies}
\label{sec:proofs_overview}

\subsection{Upper large deviations: Strategy and proof of Theorem \ref{thm:main_uldp}}
%We begin with some terminology. 
We call an edge \textit{high-weight} if at least one of its endpoints has high weight, which we make precise below. Further, \textit{light-weight} edges are edges where neither endpoints have high weight. We also divide edges into \textit{long} and \textit{short} edges, where long edges are those that connect vertices far away from each other.  More precisely, let
$(\eps_n)_{n\geq 1}$ be some sequence tending to zero, let $a\in (0,1)$ and split the edge set $E_n$ into the three disjoint sets
\begin{align}
     E_{\rm main} & := \{\{x,y\} \in E_n \colon d_n(x,y) \leq \eps_n n^{1/d},\, W_x \vee W_y \leq n^a\}\label{def:emain},\\
      E_{\rm long} & :=  \{\{x,y\} \in E_n \colon d_n(x,y) > \eps_n n^{1/d},\, W_x \vee W_y \leq n^a\}\label{def:elong},\\[1mm]
      E_{\rm high} & :=  \{\{x,y\} \in E_n \colon W_x \vee W_y > n^a\}\label{def:ehigh},
\end{align}
where $x \vee y$ denotes the maximum of $x$ and $y$.  These sets correspond to the sets containing short and light-weight edges, long and light-weight edges and lastly high-weight edges. 
Theorem~\ref{thm:main_uldp} follows from three key steps estimating the number of edges in these sets.% 
\pagebreak[3]\medskip

%Combining these steps shows
%that the best strategy to achieving the large deviation of $|E_n|$ is by producing exactly $k$ high-weight vertices, more precisely, we show
%\begin{align}
%    \p{|E_n| > n(\rho + \mu)} = (1+o(1))\p{|E_{\rm high}| > n\rho }.
%\end{align}

\paragraph{\bf Step 1. Short and light-weight edges form the bulk.} Consider short, light-weight edges. These constitute the `bulk' edge count, i.e., a typical edge of the graph falls into this category. 
Proposition~\ref{prop:order_bound_emain} shows that the number of these edges are well concentrated about the mean $n \mu$. 

\begin{prop}\label{prop:order_bound_emain}
      Let $\alpha > 1$ be the parameter in~\eqref{asspt:profile_bounds} and let
      %$\rho$ be a positive number which is not an integer and $k = \lceil \rho \rceil$. Further, let 
      $E_{\rm main}$ be as in \eqref{def:emain}  with $\eps_n$ %\searrow0$ 
      and $a$ chosen such that 
 \begin{itemize}
          \item [$\rhd$] $\eps_n^d = n^{-\gamma}$ for some $0 < \gamma < (1-\tfrac{1}{\alpha})\wedge \tfrac{1}{d}$, and
          \item[$\rhd$]  $0<a(2 \vee (\beta-1))<(1-\gamma - \tfrac{1}{\alpha}) \wedge \tfrac{1}{2}$. 
      \end{itemize}
      Then, for all $\delta>0$, we have
   \[\p{||E_{\rm main}|-n\mu|\geq n\delta} =  o({n^{-k(\beta -2)}}).\]
\end{prop}

We prove Proposition~\ref{prop:order_bound_emain} in Section~\ref{sec:proof_concentration_bulk}. In the proof we first fix the vertex weights (and in the Poisson case both the vertex locations and weights) and apply a conditional version of \mbox{McDiarmid's} concentration inequality, using that the edges are included independently. Then we use McDiarmid's inequality again to show concentration of the resulting functionals of the independent weights (and in the Poisson case also the vertex locations).% 
\medskip

\paragraph{\bf Step 2. Long and light-weight edges are negligible.} Proposition~\ref{prop:order_bound_elong} shows that this contribution is asymptotically negligible. 

\begin{prop}\label{prop:order_bound_elong}
  For all $\delta>0$, there exist $\eps_n$ and $a$
   satisfying the conditions of Proposition~\ref{prop:order_bound_emain}
   %$0<a<\frac{1}{2(2 \vee (\beta-1))}$,  such that when $n^{-1/d}\leq \eps^d_n \leq n^{-\gamma}$ for some $\gamma>0$,
   such that $E_{\rm long}$ in \eqref{def:elong} satisfies
    \[\p{|E_{\rm long}| \geq n\delta}  = o({n^{-k(\beta -2)}}).\]
\end{prop}

We prove Proposition~\ref{prop:order_bound_elong}  in Section~\ref{sec:proof_concentration_long_edges}  for which we use that the connection probability of a pair of vertices is increasing with respect to the weights of the vertices, and decreasing with respect to the spatial distance between the vertices. This implies that the total number of long and light-weight edges can be stochastically dominated by a suitable binomial random variable and we can use a standard binomial concentration argument.
\medskip

\paragraph{\bf Step 3. High-weight edges form the condensate.}  Using the first two steps, it remains to understand the probability that the total number of high-weight edges is larger than $n \rho$.  This is accomplished in Proposition~\ref{prop:order_bound_ehigh}.

\begin{prop}\label{prop:order_bound_ehigh}
Let $E_{\rm high}$ be as in \eqref{def:ehigh}. Then, for any $0<a<1$, we have
   \[\p{|E_{\rm high}|\geq n\rho}=(F(\rho)+o(1))n^{-k(\beta -2)},\]
   where $F(\rho)$ is given by (\ref{def:f_rho_function}).
\end{prop}

We prove Proposition~\ref{prop:order_bound_ehigh} in Section~\ref{sec:proof_high_weight_edges} for which we use that the number of high-weight edges is well approximated by the sum of the degrees of the high-weight vertices. The degree of a vertex given its weight has Poisson-like statistics with the mean being an increasing  function (depending on $n$) of the weight, see for example \cite[Proposition 1]{LWC_SIRGs_2020} for a proof of this fact when the locations of the vertices are uniform, as in the Poisson models. As a consequence, again using concentration arguments, when the weight of a vertex diverges, its degree behaves like this function evaluated at its weight. This observation allows us to approximate the total number of high-weight edges by an i.i.d.\ sum of a deterministic (but $n$-dependent) function of the weights.  To understand the probability that this sum is at least $n \rho$, we do a large deviation analysis similar to the one given in~\cite{KM23}, and show that this event occurs precisely when there are $k = \lceil \rho \rceil$ vertices with weight of order~$n$. This shows that $\p{|E_{\rm high}| \geq n\rho}$ behaves as stated in Proposition~\ref{prop:order_bound_ehigh}.
\medskip

{We now prove continuity of $F$, which is the last ingredient in the proof of Theorem~\ref{thm:main_uldp}, which then follows straightforwardly from the above propositions.}

\begin{lem}\label{lem:continuity}
For all $b\ge0$ and $k\in\mathbb  N$,
$$\rho \mapsto F_b(\rho):=\frac{(\beta-1)^k}{ {k!}}\int_{b}^{\infty}\cdots\int_{b}^{\infty}\I{{{\Lambda}(y_1)+\cdots+{\Lambda}(y_k)>\rho}}\prod_{i=1}^k y_i^{-\beta}dy_i$$
defines a continuous function $F_b\colon (k-1,k) \to (0,\infty)$. 
\end{lem}

\begin{proof}
Finiteness for $b>0$ follows from $\beta > 2$. For $b=0$, we note that $\lim_{y\searrow 0} \Lambda(y)=0$. Choosing $b>0$ such that $\Lambda(y)<\rho-(k-1)$ for all $y\in(0,b)$, and using that $\Lambda(y)\leq 1$,
we see that $F_0(\rho)=F_b(\rho)$. 
%\smallskip
To see positivity, we lower bound $F_b(\rho)$ by $$\left(\int_b^{\infty}\I{{\Lambda}(y)>\rho/k}y^{-\beta}dy \right)^{k}.$$ Note that $\lim_{y \uparrow \infty}{\Lambda}(y)=1$ and hence the set $\{{\Lambda}(y)>\rho/k\}$ contains an interval unbounded to the right and hence has positive measure.
Hence $F_b(\rho)>0$. \smallskip

To see continuity, fix $\rho\in(k-1,k)$. Given $\vep>0$, we need to find $\delta>0$ with
$$\int_{0}^{\infty}\cdots\int_{0}^{\infty}\I{\rho+\delta>\Lambda(y_1)+\cdots+{\Lambda}(y_k)>\rho-\delta}\prod_{i=1}^k y_i^{-\beta}dy_i< \vep.$$
For this, we note that if $\delta'>0$ is small enough, then
$\rho+\delta'>\Lambda(y_1)+\cdots+{\Lambda}(y_k)>\rho-\delta'$ implies that there exists $i$ with
$\delta'<\Lambda(y_i)<1-\delta'$. Using that
$\varphi$ has no flat pieces except possibly at zero and one,
%is strictly decreasing on the interval $(\sup\{x>0 \colon \varphi(x)=1\}, \inf\{x>0 \colon \varphi(x)=0\} )$ 
we find $\delta'>\delta'', \delta>0$ such that
$|\Lambda(y)-\Lambda(y')|>2\delta$ if
$\delta'<\Lambda(y)<1-\delta'$ and $|y-y'|>\delta''$. Hence one of the integrations with respect to $dy_1, \ldots, dy_k$ is over an interval of length at most $\delta''$ and choosing $\delta'>0$ small we can upper bound the integral by $\vep$.
\end{proof}

\begin{proof}[Proof of Theorem~\ref{thm:main_uldp}] 
Take $\delta>0$ sufficiently small such that $(\rho-\delta,\rho+\delta)\subseteq (k-1,k)$. Then fix $0<a<1$ and $(\eps_n)_{n\geq 1}$ such that Propositions~\ref{prop:order_bound_emain},~\ref{prop:order_bound_elong} and~\ref{prop:order_bound_ehigh} can be applied. 
   Using the definitions of $E_{\rm main}$, $E_{\rm long}$ and $E_{\rm high}$ in \eqref{def:emain}-\eqref{def:ehigh}, we can write 
    \begin{align}\label{eq:edges_decomposition}
        |E_n| & =|E_{\rm main}|+|E_{\rm long}|+|E_{\rm high}|.
    \end{align}
     The upper bound of Theorem~\ref{thm:main_uldp} follows by noting that
     \begin{align}
     &\p{|E_n|\geq n(\mu+\rho)}\notag\\& \leq \p{|E_{\rm main}|\geq n(\mu+\delta/2)}+\p{|E_{\rm long}|\geq n \delta/2}+\p{|E_{\rm high}|\geq n(\rho-\delta)}\label{eq:intermediate_main_thm_UB}\\&=(F(\rho-\delta)+o(1))n^{-k(\beta -2)}, \label{eq:main_thm_UB}\notag
     \end{align}
     where the last equality holds by 
     %our choice of $a$ and $(\eps_n)$ and 
     applying Propositions~\ref{prop:order_bound_emain},~\ref{prop:order_bound_elong} and~\ref{prop:order_bound_ehigh}. 
     For the lower bound, remark that 
     \begin{align}
         \p{|E_n|\geq n(\mu+\rho)} & \geq \p{|E_{\rm high}| + |E_{\rm main}| \geq  n(\mu+\rho)}\\
         & \geq  \p{|E_{\rm high}| \geq n(\rho +\delta),\,|E_{\rm main}| \geq n(\mu-\delta)} \notag\\
         & \geq  \p{|E_{\rm high}| \geq n(\rho +\delta)} - \p{|E_{\rm main}| < n(\mu-\delta)}\notag.
     \end{align}
     By Propositions~\ref{prop:order_bound_emain} and~\ref{prop:order_bound_ehigh}, the last line equals $(F(\rho+\delta)+o(1))n^{-k(\beta -2)}$.
    Combining the upper and lower bounds and, {recalling from Lemma~\ref{lem:continuity} that} $F$ is continuous on $(k-1,k)$, we complete the proof by 
    %first letting $n \to \infty$ and then 
    letting $\delta$ tend to zero. 
\end{proof}

\subsection{Convergence of the empirical edge-lengths: Strategy and proof of Theorem \ref{thm:distances}}
{Throughout this section $k$ is the unique integer with $k-1<\rho<k$.} The key step in the proof of Theorem \ref{thm:distances} \ref{eq:conv_prob_bulk} is the following law of large numbers.
\begin{prop}\label{prop:lln_local_functionals}
    Let  $f\colon (0,\infty)\to [0,\infty)$ be continuous with compact support and recall~$\lambda_f$ from \eqref{def:lambda_over_rd}.
    %where $w\geq 1$. 
    Then, %for lattice and Poisson models, 
for all $\delta>0$,
   $$\mathbb P\Big(  \Big| \frac1n \sum_{_{\{x,y\}\in E_n}}
   f(d_n(x,y)) -\tfrac{1}{2}\E{\lambda_f(W)}
   \Big|>\delta, |E_n| \geq n(\mu + \rho)\Big) =  o({n^{-k(\beta -2)}}).$$
\end{prop}

In the proof of Proposition~\ref{prop:lln_local_functionals}, given  in Section~\ref{sec:bulk-shape}, we split the sum over all edges into the sum over the edges in $E_{\rm main}$, for which we can argue similarly as in Proposition~\ref{prop:order_bound_emain}, and the sum over the edges in $E_{\rm long}$ and $E_{\rm high}$, which are negligible as most of these edges are too long to be in the support of $f$. 
%\phmar{There is a problem here as the intended tool (Proposition~\ref{prop:general_concentration_edge_function}) does not deal with high weight vertices. Moreover, there currently seems to be no result that ensures convergence to zero of sums of high weight edges that is fast enough to get the needed speed in Proposition~\ref{prop:lln_local_functionals}. We should discuss this at the next meeting.}
% {\color{cyan}For sufficiently large $n$, all edges in $E_{\rm long}$ have a length outside the support of $f$ and therefore do not contribute.}

\begin{proof}[Proof of Theorem~\ref{thm:distances}\ref{eq:conv_prob_bulk}]
% Combine Proposition~\ref{prop:lln_local_functionals} 
% and Theorem~\ref{thm:main_uldp}.
%\cmar{Remco: I added a step in the proof.}
Proposition~\ref{prop:lln_local_functionals} together with Theorem~\ref{thm:main_uldp} imply that for all $\delta>0$, we have
\begin{align*}
    \p{\left|\int f(x) d\mu_n(x) - \tfrac{1}{2}\E{\lambda_f(W)}\right| > \delta,|E_n| \geq n(\mu + \rho)} = o(1) \p{|E_n| \geq n(\rho + \mu)}.
\end{align*}
Theorem~\ref{thm:distances}\ref{eq:conv_prob_bulk} then follows.
\end{proof}

For the proof of Theorem \ref{thm:distances}\ref{eq:conv_distr_condensates}, we first approximate the empirical edge-length distribution at macroscopic scale by a deterministic function of the $k$ largest weights in the graph. Let  $f\colon (0,\infty)\to [0,\infty)$ be continuous with compact support and define 
\begin{align}\label{def:lambdaf}
    \Lambda_f(w) := \int_{[-\frac{1}{2}, \frac{1}{2}]^d} 
f(\|z\|) {{\Lambda}}(w,z)\, dz,
\end{align}
where we recall $\Lambda(w,z)$ from \eqref{def:lambda_refined}. Our next result identifies the limiting distribution of long edges in terms of the largest weights.

\begin{prop}\label{prop:conv_distr_scaled_high_weight_edges}
Conditionally on $|E_n| \geq n(\rho + \mu)$, 
        \begin{align*}
             \frac1n \sum_{_{\{x,y\}\in E_n}}
   f\left(\frac{d_n(x,y)}{n^{1/d}}\right) - \sum_{i=1}^k\Lambda_f\left(\frac{W_{\sss(i)}}{n}\right) \convp 0,
        \end{align*}
        where $W_{\sss(1)}, \ldots, W_{\sss(k)}$ are the $k$ largest weights in the graph ordered by decreasing size.
\end{prop}

We prove Proposition~\ref{prop:conv_distr_scaled_high_weight_edges} in Section~\ref{sec:lenghts} by using a concentration argument to approximate, for all vertices $x$ with weight $W_x$ exceeding $n^a$, the term \smash{$\sum_{y:\,\{x,y\}\in E_n} f(\tfrac{d_n(x,y)}{n^{1/d}})$} by its conditional expectation given $W_x$. For $W_x\approx wn$ (here $\approx$ meaning the ratio of the two sides going to $1$), the conditional expectation converges to
%$\int f(\|z\|) \lambda(w, z)$. 
$\Lambda_f(w)$.
Together with the fact that the contribution of light-weight edges to the 
number of long edges  is negligible, this allows us to approximate $\int f( \frac{x}{n^{1/d}}) \, d\mu_n$
by a deterministic function of the $k$ largest weights.
\medskip

We then show convergence in distribution of the $k$ largest weights in the graph.

\begin{prop}\label{prop:heavy_wt_density} Conditionally on 
$|E_n|\geq n(\mu+\rho)$, 
\begin{align*}
    \frac{1}{n}(W_{\sss (1)},\dots,W_{\sss(k)}) \convdist (Y_{ {1}},\dots,Y_{ {k}}),
\end{align*}
where the vector $(Y_1,\dots,Y_{k})$ is chosen according to the law given by (\ref{eq:prob_y}). %i.e.,
%\begin{align*}
 %   \p{Y_1\ge x_1,\ldots, Y_k \ge x_k}
  %  = \frac{(\beta-1)^k}{F(\rho)}  \int_0^\infty\cdots\int_0^\infty 
  %  \I{\sum_{i=1}^k {{\Lambda}}(y_i)>\rho}\prod_{i=1}^k \I{ y_i \ge x_i\vee y_{i+1}}y_i^{-\beta}dy_i,
%\end{align*}
%and where $F(\rho)$ is given by (\ref{def:f_rho_function}).
\end{prop}

We prove Proposition~\ref{prop:heavy_wt_density} in Section~\ref{sec:heavy_wt_density} by refining the analysis of the high-weight edges in Section~\ref{sec:proof_high_weight_edges}. We show that
the event $|E_n|\geq n(\mu+\rho)$ is asymptotically equivalent to the event $\sum_{i=1}^k {{\Lambda}}(W_{\sss(i)}/n)>\rho$. The result follows because, given the vertex set $V_n$, the random variables $(W_x)_{x\in V_n}$ are independent Pareto distributed.
%\medskip

\begin{proof}[Proof of Theorem~\ref{thm:distances}\ref{eq:conv_distr_condensates}] Note that $\Lambda_{f}$ is bounded from above and continuous by means of the linearity of~$\cK$ in the first component. Thus the function \smash{$(x_1, \dots, x_k) \mapsto \sum_{i=1}^k\Lambda_{f}(x_i)$} is continuous and bounded. 
Combining Propositions~\ref{prop:conv_distr_scaled_high_weight_edges}  and~\ref{prop:heavy_wt_density}, we get that 
$$\frac1n \sum_{_{\{x,y\}\in E_n}}
   f\left(\frac{d_n(x,y)}{n^{1/d}}\right)
  \convdist  \sum_{i=1}^k\Lambda_f\left(Y_i\right),
   $$
   as claimed in the theorem.
\end{proof}
\begin{remark}[A hub clique] \label{rem:hub_clique}
    \rm As a consequence of Proposition \ref{prop:heavy_wt_density},  conditionally on the event $|E_n|\geq (\mu+\rho)n$, there are $k$ %\lceil \rho \rceil$ 
    vertices with weight at least $\eps n$, with probability at least $1-f(\eps)$, where $f(\eps) \to 0$ as $\eps \searrow 0$. We call these high-weight vertices~{\emph{hubs}}. As the length of any edge is at most \smash{$\frac{\sqrt{d}}2 n^{1/d}$,} and the connection probability is \smash{$\varphi({d_n(X_i,X_j)^d}/{\kappa(W_i,W_j)})$}, we use a union bound to see that the probability of the event that there is a pair of hubs which does not share an edge is at most $\binom{k}{2}(1-\varphi(2^{-d}/({\sf c} \eps n^2n)))$. Hence if $\lim_{r\to0} \varphi(r)=1$ 
    %as in our examples (i), (ii), 
    the $k$ hubs form a clique with high probability. 
    \hfill\ensymboldefinition
\end{remark}

\subsection{Convergence of degree distribution: Strategy and proof of Theorem \ref{thm:EDD}}
In this section, we reduce the proof of Theorem~\ref{thm:EDD} to the following proposition. Recall $D_{\infty}(w)$ denotes the degree of the root in the local limit $G_{\infty}$ of the graph $G_n$ conditioned on having weight $w$.

\begin{prop}\label{prop:degree_distr} 
Let $y_1>y_2>\cdots>y_k>0$ and $u_1,\ldots, u_k\in\mathbb T_1^d$. {Let
    \begin{align}\label{def:pi_ab_cond}
    \pi_{a,b}\big( (y_{i},u_{i})_{i\in[k]})=\mathbb P\Big( D_{\sss \infty}(W)=a, \sum_{i=1}^kB^{\sss(i)}_{U,W}=b \Big),
    \end{align}
where $(B^{\sss(i)}_{U,W})_{i\in [k]}$ are, conditionally on $U,W$, independent Bernoulli random variables with parameter $$\varphi\bigg(\frac{d_1(U,u_{i})^d}{\cK(y_{i},W)}\bigg).$$ }
Conditionally on $W_{\sss(i)}=ny_{i}$ and
\smash{$X_{\sss(i)}=n^{1/d}u_{i}$} (in the lattice case projected to the nearest lattice point of~$\mathbb T_n^d$) for all $i\in[k]$,
\smallskip
\begin{enumerate}[label = (\roman*)]
\item $\frac{1}{n}(D_{X_{\sss(1)}},\dots,D_{X_{\sss(k)}}) \convp (\Lambda(y_{1}),\dots,\Lambda(y_{k})),$\label{eq:degree_distr_large_degrees}
%\label{eq:k_high_deg_convp}$ 
\smallskip
\item $
     \pi_{a,b}^{\sss(n)}\convp \pi_{a,b}\big( (y_{i},u_{i})_{i\in[k]}\big),$\label{eq:degree_distr_lim_distr}
     %\label{eq:pin_convp}$ 
     \smallskip
\item$
   \frac1n D_{X_{\sss (j)}} \convp 0, \mbox{ for each $j>k$.}$\label{eq:degree_distr_small_degrees}
   %\label{eq:cond_scaled_low_weight_vertices}
\end{enumerate}
\end{prop}

We prove Proposition~\ref{prop:degree_distr} in Section~\ref{sec:degree_distribution}, and now prove Theorem \ref{thm:EDD} subject to it.

\begin{proof}[Proof of Theorem~\ref{thm:EDD}]
We present the technicalities in the language of the Poisson case, the lattice case is analogous. To prepare this proof, take $g\colon (0,\infty)^k \to [0,\infty)$ an arbitrary, measurable function. Write
$\Pi_n=\sum_{x\in V_n} \delta_{W_x}$ for the point process of weights and $\cW_n=\{W_x \colon x\in V_n\}$ for the corresponding point set. Then
$$g(W_{\sss(1)}, \ldots, W_{\sss(k)})=
\sum_{{w_1,\ldots,w_k \in \cW_n}\atop{w_1> \cdots> w_k}} \I{\Pi_n(w_1,\infty)=0, \Pi_n(w_2,w_1)=0, \ldots, \Pi_n(w_k, w_{k-1})=0} g(w_1,\ldots, w_k).$$
Applying the multivariate Mecke formula~\cite[Theorem 4.4]{Last_Penrose_LPP} and denoting the intensity measure of~$\Pi_n$ by $\pi_n$, we get
\begin{align*}
\mathbb E[ & g(W_{\sss(1)}, \ldots, W_{\sss(k)})] \\
& = \int \!\!\pi_n(dw_1) \cdots\! \int \!\!\pi_n(dw_k)
\I{w_1> \cdots> w_k} \mathbb P( \Pi_n(w_k,\infty)=0) g(w_1, \ldots, w_k).
\end{align*}
Because
$\pi_n(dt)=n (\beta -1)t^{-\beta}\I{t\geq 1} dt$ and hence
$\pi_n(w_k,\infty)=n w_k^{1-\beta}$,
we get 
\begin{align}
    \mathbb E & [g(W_{\sss(1)}, \ldots, W_{\sss(k)})] \notag\\
& = n^k (\beta -1)^k \int_1^\infty dt_1  \cdots \int_1^\infty dt_k 
(t_1\cdots t_k)^{-\beta} \e^{-n t_k^{1-\beta}}
\I{t_1> \cdots> t_k} g(t_1, \ldots, t_k).
\label{mecke}
\end{align}
Now fix an integer $\ell\geq 1$ and a bounded, continuous function $h\colon \mathbb R^{k+(k+1)(\ell+1)} \to \mathbb R$.  To simplify notation we write $(\pi_{a,b}^{\sss(n)})$ for $(\pi_{a,b}^{\sss(n)})_{0\leq a\leq \ell, 0\leq b \leq k}$ and { $(y_i,u_i)$ for $(y_i,u_i)_{i\in[k]}$}, dropping the index range when there is no ambiguity. To prove \eqref{eq:joint_conv_high_degrees_edd}, it suffices to show, as $n\to \infty$,
\begin{align}\label{eq:edd_equivalent_statement}
\expec\Big[h\big(\tfrac{1}{n}(D_{X_{\sss(i)}})_{i \in [k]}, & (\pi_{a,b}^{\sss(n)})\big)~\Big|~ |E_n|\geq(\mu+\rho)n\Big]
     \rightarrow \expec\Big[h\big((\Lambda({Y_{i}})),\big(\pi_{a,b}({(Y_{i}, U_{i})})\big) \big)\Big].
\end{align}
Fix $\eps>0$ small enough so that $\Lambda(\eps) < \rho - (k-1)$. Abbreviate $B=\{|E_n|\geq(\mu+\rho)n\}$ and $$A =\Big\{ (t_1, \ldots, t_k) \in (\eps,\infty)^k \colon \sum_{i\in [k]}\Lambda(t_i)\geq \rho\Big\}.$$
The left-hand side of \eqref{eq:edd_equivalent_statement} can be written as
\begin{align}
    & \frac{1}{\p{B}}\E{\Itwo{A}(({W_{\sss(i)}}/{n}))h\big(\tfrac{1}{n}(D_{X_{\sss(i)}}), (\pi_{a,b}^{\sss(n)})\big)}\label{eq:edd_equivalent_event_a1}\\
    & - \frac{1}{\p{B}}\E{\Itwo{B^c}\Itwo{A}(({W_{\sss(i)}}/{n}))h\big(\tfrac{1}{n}(D_{X_{\sss(i)}}), (\pi_{a,b}^{\sss(n)})\big)}\label{eq:edd_equivalent_event_a2}\\
    & + \frac{1}{\p{B}}\E{\Itwo{B}\Itwo{A^c}(({W_{\sss(i)}}/{n}))h\big(\tfrac{1}{n}(D_{X_{\sss(i)}}), (\pi_{a,b}^{\sss(n)})\big)}.\label{eq:edd_equivalent_event_a3}
\end{align} 
We will show that \eqref{eq:edd_equivalent_event_a1} converges to the right-hand side of \eqref{eq:edd_equivalent_statement} while \eqref{eq:edd_equivalent_event_a2} and \eqref{eq:edd_equivalent_event_a3} tend to zero as $n\to \infty$. \pagebreak[3]

Since $h$ is bounded, \eqref{eq:edd_equivalent_event_a3} can be bounded from above by a constant multiple of 
\begin{align}\label{eq:edd_equivalent_statement_1}   \mathbb P \Big( & (W_{\sss(i)}/n)_{i \in [k]}\in A^c
\mid {B} \Big) \notag\\  \leq & \,\mathbb P \Big( \sum_{i\in [k]}\Lambda(W_{\sss(i)}/n)< \rho \mid B \Big) +  \Cprob{ W_{\sss(i)}/n\leq \eps \mbox{ for some } i \in [k]}{B}.
\end{align}
By Proposition~\ref{prop:heavy_wt_density} and the Portmanteau theorem, 
$$
\prob\Big(\sum_{i\in[k]} \Lambda\big(W_{\sss(i)}/n\big)<\rho\mid B\Big) \to \prob\Big(\sum_{i\in[k]} \Lambda\big(Y_i\big)<\rho\Big)
=0.$$ For the second summand of \eqref{eq:edd_equivalent_statement_1}, observe that by Proposition~\ref{prop:heavy_wt_density}, 
\begin{align*}
   \Cprob{W_{\sss(i)}/n \leq \eps \mbox{ for some } i \in [k] }{B} \to \p{Y_i \leq \eps \mbox{ for some } i \in [k]}.
\end{align*}
Recall that $\eps$ is chosen such that  $\Lambda(\eps) < \rho - (k-1)$ and so 
\smash{$\sum_{i\in [k]}\Lambda(Y_i)\geq \rho$} cannot hold if there is
$i\in [k]$ with $Y_i<\epsilon$. Hence the law of $(Y_1,\ldots, Y_k)$ given in \eqref{eq:prob_y} is such that $Y_i > \eps$ for all $i \in [k]$, almost surely. It follows that \eqref{eq:edd_equivalent_statement_1}, and hence \eqref{eq:edd_equivalent_event_a3} tend to~zero.% 
\smallskip

Next, \eqref{eq:edd_equivalent_event_a2} can be bounded %from above 
by a constant multiple of 
%\begin{align*}
%    \left|\frac{1}{\p{B}}\E{\Itwo{B^c}\Itwo{A}h\big(\tfrac{1}{n}(D_{X_{\sss(i)}})_{i \in [k]}, (\pi_{a,b}^{\sss(n)})\big)}\right|  \leq c\left|\frac{\p{A}}{\p{B}}- \Cprob{A}{B}\right|,
%\end{align*}
$$|{\p{(W_{\sss(i)}/n)_{i \in [k]}\in A}}/{\p{B}}- \mathbb P((W_{\sss(i)}/n)_{i \in [k]}\in A\mid B)|.$$
By the above, the subtrahend goes to one. The minuend  does the same % $\p{A}/\p{B} \to 1$, 
as, by \eqref{mecke}, 
\begin{align}
   & \p{(W_{\sss(i)}/n)_{i \in [k]}\in A}\notag \\ & =  
    n^k (\beta -1)^k \int_1^\infty dt_1  \cdots \int_1^\infty dt_k 
(t_1\cdots t_k)^{-\beta} \e^{-n t_k^{1-\beta}}
\I{t_1> \cdots> t_k, (t_1/n, \ldots, t_k/n) \in A} \notag\\
& \sim n^{k(2-\beta)}
(\beta -1)^k \int_\eps^\infty dy_1  \cdots \int_\eps^\infty dy_k 
(y_1\cdots y_k)^{-\beta} 
\I{y_1> \cdots> y_k, 
\sum_{i\in [k]}\Lambda(y_i)\geq \rho}, \label{eq:rhs}
%(t_1, \ldots, t_k) \in A},
\end{align}
using the substitution $t_i = y_i n$ and using that $2-\beta<0$. {Again by choice of $\eps$ if there is
$i\in [k]$ with $y_i<\epsilon$ then
$\sum_{i\in [k]}\Lambda(y_i)< \rho$. Together with
Theorem~\ref{thm:main_uldp} this shows that \eqref{eq:rhs} is asymptotically equivalent to~$\mathbb P(B)$.}\smallskip

It remains to show that \eqref{eq:edd_equivalent_event_a1} converges to the right-hand side of \eqref{eq:edd_equivalent_statement}. Note that
\begin{align*}
     & \frac{1}{\p{B}} \E{\Itwo{A}((W_{\sss(i)}/n)_{i \in [k]})h\big(\tfrac{1}{n}(D_{X_{\sss(i)}})_{i \in [k]}, (\pi_{a,b}^{\sss(n)})\big)}\\
     & = \frac{1}{\p{B}}\E{\Itwo{A}((W_{\sss(i)}/n)_{i \in [k]})\Cexp{h\big(\tfrac{1}{n}(D_{X_{\sss(i)}})_{i \in [k]}, (\pi_{a,b}^{\sss(n)})\big)}{(W_{\sss(i)}, X_{\sss(i)})}} .
\end{align*}
Using \eqref{mecke} and the substitution $t_i=y_i n$ again, this equals
\begin{align*}
     (\beta -1)^k & \int_\eps^\infty dy_1  \cdots \int_\eps^\infty dy_k 
\int_{\mathbb T_1^d} du_1  \cdots \int_{\mathbb T_1^d} du_k (y_1\cdots y_k)^{-\beta} 
\I{y_1> \cdots> y_k, \sum_{i\in [k]}\Lambda(y_i)\geq \rho}  \\ &
\cdot \frac{1}{\p{B}} n^{k(2-\beta)} \e^{-n^{2-\beta} y_k^{1-\beta}} \Cexp{h\big(\tfrac{1}{n}(D_{X_{\sss(i)}})_{i \in [k]}, (\pi_{a,b}^{\sss(n)})\big)}{(W_{\sss(i)}, X_{\sss(i)})=(n y_i, n^{1/d}u_i)}.\end{align*}
 By Theorem~\ref{thm:main_uldp} the integrand (in the second line) is bounded, and by Proposition~\ref{prop:degree_distr}\ref{eq:degree_distr_large_degrees} and~\ref{eq:degree_distr_lim_distr}, converges to
 $$
  \frac1{F(\rho)}
h\big((\Lambda({y_{i}})),\big(\pi_{a,b}({(y_{i}, u_{i})})\big).
  $$
As the integrating  measure (in the first line) is finite, dominated convergence implies convergence to
$$\expec\Big[h\big((\Lambda({Y_{i}})),\big(\pi_{a,b}({(Y_{i}, U_{i})})\big) \big)\Big],$$
as required to show~\eqref{eq:edd_equivalent_statement} and hence~\eqref{eq:joint_conv_high_degrees_edd}. \smallskip

Finally, to prove \eqref{eq:scaled_low_weight_vertices}, it suffices to show that, for $j>k$ fixed, \smash{$\mathbb E[\frac1n D_{X_{\sss(j)}} |B] \to0$} . Similar to above, this can be reduced to the statement of Proposition~\ref{prop:degree_distr}\ref{eq:degree_distr_small_degrees}.
\end{proof}

\subsection{Discussion: Why  we assume $\rho$ to be non-integer.}\label{sec:conc}
We now briefly discuss the nature of the problem when $\rho>0$ 
is an integer in Theorem~\ref{thm:main_uldp}. 
The proof can be reduced to studying the deviations of a truncated i.i.d.\ sum 
of the form $\sum_{i=1}^n\lambda_n(W_i)\I{W_i > n^a}$ for a suitable $a>0$, 
where we can think of $\lambda_n(w) = \Cexp{D_x}{W_x = w}$. It turns out that, for any $b>0$, we have $\lambda_n(bn)/n \to {\Lambda}(b)$ for some function ${\Lambda}(b)$, which converges to $1$ as $b \to \infty$. 
This leads us to study the large deviations of $S_n:=\sum_{i=1}^n n{\Lambda}(W_i/n)$, with $(W_i)_{i \geq 1}$ an i.i.d.\ sequence with law 
\eqref{eq:Pareto_density} and ${\Lambda}(w) \leq 1$ with ${\Lambda}(w) \to 1$ as $w\to \infty$. We claim that these large deviations are not universal at the integers.\smallskip
 
 %For a heuristic of the non-universality in a simpler setting, let $(W_i)_{i \geq 1}$ be an i.i.d.\ sequence with law \eqref{eq:Pareto_density}, and 
 We discuss the probability of the event $S_n\geq(\mu+k)n$, where $S_n/n \convp \mu = \mathbb{E}[W_1]$,    
 for three illustrative examples of ${\Lambda}$, which differ in the approach to $1$ in the large $x$ limit:
\begin{align*}
    \text{(a)}\;{\Lambda}(x)=1 \wedge x;\;\;\;\text{(b)}\;{\Lambda}(x)=1-\e^{-x};\;\;\;\text{(c)}\;{\Lambda}(x)=\frac{x}{1+x}.
\end{align*}
 
In case (a), we note that if $W_i>n$, the corresponding summand contributes $n$ to the sum~$S_n$. Thus, we get $S_n\geq(\mu+k)n$ if there are $k$ values of $i$ for which $W_i>n$. The remaining contribution comes from the `small' weights remaining tightly concentrated around~$\mu n$.
%, and these $k$ `large' weights help to realise the event $\{S_n>n(\mu + k)\}$. Here note that having only $k$ $W_i$'s at least $n$ is not enough, since again the small $W_i$'s will contribute about $\mu n$, while the large ones will contribute exactly $k n$, and we will not achieve the inequality $S_n>n(\rho+\mu)$. 
This strategy turns out to be optimal and has a  probability of order $n^{-k(\beta-2)}$, similar in flavour to Theorem \ref{thm:main_uldp}.  \smallskip
\pagebreak[3]

In case (b), we can consider a strategy where there are exactly $k$ weights being at least $cn\log{n}$ for some large constant $c>0$. The contribution coming from these large weights is at least $nk(1-n^{-c})$. The contribution coming from the remaining small weights is again tightly concentrated around $\mu n$, so they contribute with good probability at least $\mu n+n^{1-c}$ to the sum. This strategy has probability of  order $(cn\log{n})^{-k(\beta-2)}$, which vanishes faster than \smash{$n^{-k(\beta-2)}$}, but is still much cheaper than having $k+1$ vertices with weight of order $n$, and we believe it to be optimal.\smallskip

In case (c), as before, we can always choose $k+1$ vertices of weight of order $n$. However, another strategy similar to case (b) can be obtained by choosing $k$ weights to be at least~$n^b$ for some appropriate $b>1$, and arguing as before that the contribution from the small weights is highly concentrated. %Since the large weights contribute only at least $kn(1-n^{-b})$, we again obtain $S_n>n(\mu+k)$.
Depending on $k$ and $\beta$, either of these strategies can be optimal, and we expect to see a phase transition depending on $k$ and $\beta$ in this case.

\section{Concentration of the bulk: Proof of Proposition \ref{prop:order_bound_emain}}\label{sec:proof_concentration_bulk}

In this section, we show that {the short edges are essentially those in $E_{\rm main}$ and  the measure $\mu_n$ is asymptotically equivalent to} the empirical measure of the edge  lengths
in the unconditioned model.  We define $\Q_n$ to be the cube of volume $n$ in $\R^d$ centred at zero.

\subsection{Setup and main techniques}

Recall the definition of $E_{\rm main}$ given by \eqref{def:emain},
\[
	E_{\rm main}(G_n) = \{\{x,y\} \in E_n \colon d_n(x,y) \leq \eps_n n^{1/d},\, W_x \vee W_y \leq n^a\},
\]
where $\eps_n^d = n^{-\gamma}$ for some 
\smash{$0 < \gamma < (1-\tfrac{1}{\alpha})\wedge \tfrac{1}{d}$} and $a>0$ such that $a(2\vee (\beta-1)) < (1-\gamma - \tfrac{1}{\alpha}) \wedge \tfrac{1}{2}$.
%\CK{In the proofs of this section its not always clear for which models (poisson or lattice) things are proven.. We should  make sure that this is more clear to the reader. And also make sure that if we just prove things for one model, mention that the proof is similar for the other model (or add more details if needed).}
%\PvdH{This is not needed as we always condition on the number of points first. That is the only step that is model-dependent. After that the arguments are the same.}
Our main goal is to prove that $|E_{\rm main}|$ is well-concentrated around $n \mu$. The number of vertices in our graph, denoted by $N_n$, is equal to $n$ in the lattice case, and a $\mathrm{Poi}(n)$-distributed random variable in the Poisson case. {In this case it is helpful to condition on the number of vertices.} Let $\mathbb{P}_{N_n}$ denote the probability of our graph models conditioned on the number of vertices $N_n$ and write $\mathbb{E}_{N_n}$ for the associated expectation. \medskip

For $M>0$, define the interval
\begin{align}\label{def:interval_n_sqrt}
    I_M(n) := [n - M\sqrt{n\log n}, n+ M\sqrt{n\log n}].    
\end{align}
Standard Chernoff concentration bounds \cite[Remark 2.6]{JanLucRuc00} for Poisson random variables give us, for Poisson models,
\begin{align}
    \p{N_n \notin I_M(n)} & \leq 2 \exp\left(- \frac{M^2n\log n}{2(n+ M\sqrt{n\log n})}\right) = n^{-\frac{M^2}{2+ M\sqrt{\log n/n}}},
\end{align}
for all $M>0$. In particular, choosing $M$ sufficiently large, for example such that $M^2>3k(\beta -2)$, we get that 
\begin{align}\label{eq:conc_bound_nb_ppp}
    \p{N_n \notin I_M(n)} = o(n^{-k(\beta -2)}),
\end{align}
for Poisson models. 
%For lattice models, it holds that $\p{N_n \notin I_M(n)} =0$ for sufficiently large $n$. 
Once we condition on the number of vertices $N_n$, the graph $G_n$ is completely determined by a vector ${\bW}_{N_n} = (W_i)_{1 \le i \le N_n}$ of i.i.d.\ weights with law \eqref{eq:Pareto_density}, the vector ${\bX}_{N_n} = (X_i)_{1 \le i \le N_n}$ of conditionally i.i.d.\ uniform positions in $\T_n^d$ in the Poisson case, resp.\ locations of all the lattice points in $\T_n^d$ in the lattice case, and
the edge indicators 
$${\bA}_{N_n} = (A_{i,j})_{1\leq i< j \leq N_n} \in \{0,1\}^{\binom{N_n}{2}}.$$ 
We denote the degree of the $i$th  vertex by $D_i = D_{X_i}$. This setup allows us to provide a single proof that covers both cases simultaneously. For notational simplicity we will abbreviate $N_n$ to $N$ for the rest of the paper.\smallskip

In this section we actually prove a more general result than just for $E_{\rm main}$, which involves a continuous bounded function $f \colon [0,\infty) \to [0,\infty)$ evaluated on the edge-lengths. The first step is to define a function $h_n$ that depends on the edge indicators as well as the positions and weights of vertices and the value of $f$ on the distances between the positions. We let 
$${\ba}_N = {(a_{i,j})_{1\leq i<j\leq N}} \in \{0,1\}^{\binom{N}{2}}, \phantom{X} {\bf w}_{N}\in [1,\infty)^{N}, \phantom{X} \text{  and } {\bx}_{N} \in \T_n^{d}.$$
We define the function $h_n : \{0,1\}^{\binom{N}{2}} \times [1,\infty)^{N} \times (\T_n^d)^{N} \to \mathbb{R}_+$, which depends on $f$ as
\begin{equation}\label{eq:def_edge_function_h}
	h_n({\ba}_{N}, {\bw}_{N}, {\bx}_{N}) = \sum_{1 \le i < j \le N} a_{i,j} f(d_n(x_i,x_j)) \I{d_n(x_i,x_j)\le \eps_n n^{1/d}}\I{w_i \vee w_j \le n^a}.
\end{equation}
%{Recall $\lambda(z,w)$ from \eqref{def:lambda} and  $\lambda_f(w)$ from \eqref{def:lambda_over_rd}, for $z\in \R^d$ and $w\geq 1$.}
The main result of this section is a concentration result for $h_n({\bA}_{N}, {\bW}_{N}, {\bX}_{N})$ around the value~$n\E{\lambda_f(W)}/2$.

\begin{prop}\label{prop:general_concentration_edge_function}
Let $W$ be a random variable as in~\eqref{eq:Pareto_density} and $f \colon [0,\infty) \to [0,\infty)$ a continuous bounded function. Then, for any $k \in \mathbb{N}$ and $\delta > 0$, as $n \to \infty$,
\[
	\p{\left|h_n({\bA}_{N}, {\bW}_{N}, {\bX}_{N}) - \frac{n}{2}\E{\lambda_f(W)}\right| > \delta n} = o( n^{-k(\beta-2)}).
\]
\end{prop}

Note that if $f= 1$, then $h_n({\bA}_{N}, {\bW}_{N}, {\bX}_{N}) = |E_{\rm main}|$ while $\frac{n}{2}\E{\lambda_f(W)} = n\mu$. Hence Proposition~\ref{prop:order_bound_emain} is a direct corollary of Proposition~\ref{prop:general_concentration_edge_function}.

\subsection{Proof of Proposition~\ref{prop:general_concentration_edge_function}}

The key tool behind proving the concentration result is {\em two} applications of McDiarmid's inequality on $h_n$. The first application conditions on the weights and positions, so that the edge indicators are independent Bernoulli random variables, and the second application is for the conditional expectation of $h_n$ given the weights and positions.\medskip

\pagebreak[3]
It turns out to be convenient to include the indicator of a \emph{good event} in the function~$h_n$. For this, let {$\eps_n = n^{-\gamma}$} for some $0< \gamma < 1/d$, and define $\cP_n(\eps_n)$ as the partition of $\mathbb{T}_n^d$ of equally-sized cubes, each of side length $\eps_n n^{1/d}$. Note that $\cP_n(\eps_n)$ contains at most $\eps_n^{-d}$ many such cubes. For $M>0$, consider the event
    \begin{align}\label{eq:event_P_points_inside_poly_tiles_typical}
        \cE_{\eps_n,M} := \{\forall\, \cC\in \cP_n(\eps_n):\ |\cC\cap \bX_N| \in I_{M}(N\eps_n^d)\},
    \end{align}
where we recall the interval $I_{M}(N\eps_n^d)$ from \eqref{def:interval_n_sqrt}. In words, $\cE_{\eps_n,M}$ corresponds to the event that the number of points of $\bX_N$ contained in each cube in the partition $\cP_n(\eps_n)$ lies in the given interval $I_{M}(N\eps_n^d)$. Note that for a cube $\cC \in \cP_n(\eps_n)$, {and conditionally on $N$,}{} the number of points in $\cC \cap \bX_{N}$ is ${\rm Bin}(N, \eps_n^d)$-distributed. 
Using standard Chernoff bounds, and a union bound for the number of cubes, it is easy to see that
{
\begin{align}\label{eq:conditional_chernoff_bound_cubes}
    1 - \pN{\cE_{\eps_n,M}} & \leq \eps_n ^d \max_{\cC\in \cP_n(\eps_n)}\pN{|\cC\cap \bX_{N}| \notin I_M(N\eps_n^d)}\notag\\
    & \leq 2\eps_n ^d \exp\left(-M^2 \log(\eps_n^d N)/3\right)  \leq cN^{-M^2/3}n^{\gamma d(1+M^2/3)},
\end{align} 
for some constant $c>0$. Hence, by taking an expectation over $N$,
\begin{align*}
    1- \p{\cE_{\eps_n,M}} & \leq\E{\I{N \in I_M(n)}(1-\pN{\cE_{\eps_n,M}})} + \p{N \notin I_M(n)}\\
    & \leq c\,2^{M^2/3}n^{-M^2/3} n^{\gamma d(1+M^2/3)}+ \p{N \notin I_M(n)}.
\end{align*}
Using \eqref{eq:conc_bound_nb_ppp} and the fact that $0<\gamma<1/d$, we can choose $M$ sufficiently large such that 
\begin{align}\label{eq:chernoff_bound_cubes}
     1- \p{\cE_{\eps_n,M}} & = o(n^{-k(\beta-2)}).
\end{align}
Therefore it suffices to work conditionally on the event $\cE_{\eps_n,M}$. Define the function \begin{align}\label{def:gn_function}
g_n\colon  \{0,1\}^{\binom{N}{2}}\times [1,\infty)^{N} \times (\T_n^d)^{N} & \to [0,\infty),\\
    g_n(\ba_{N}, \bw_{N}, \bx_{N}) & = \Itwo{\cE_{\eps_n,M}} h_n(\ba_{N}, \bw_{N}, \bx_{N}).
  \notag
\end{align}
%\phmar{Why is this on one line?}
The main technical result needed to prove Proposition~\ref{prop:general_concentration_edge_function} is the following concentration bound on $g_n$.

\begin{lem}\label{lem:concentration_gn_function}
Let $\delta > 0$. For both the lattice and Poisson case, there exist constants $\sK_1, \sK_2 > 0$ such that for sufficiently large $n$, conditionally on the event $N \in I_M(n)$,
\begin{equation}
	\begin{aligned}
		&\hspace{-20pt}\pN{\left|g_n({\bA}_N, {\bW}_N, {\bX}_N)
			- \mathbb{E}_N\left[g_n({\bA}_N, {\bW}_N, {\bX}_N)\right]\right| \ge \delta n} \\
		&\le 2(1-\pN{\cE_{\eps_n,M}}) + 2\exp\left(-{\sf K}_1 \delta^2 \eps_n^{-d} \right)
			+ 2\exp\left(-{\sf K}_2 \delta^2 n^{1 - 2a(2\vee (\beta-1))}\right).
	\end{aligned}
	\label{eq:concentration_gn_function}
\end{equation}
%\begin{equation}
%	\begin{aligned}
%		&\hspace{-30pt}\pN{\left|g_n({\bA}_N, {\bW}_N, {\bX}_N)
%			- \mathbb{E}_N\left[g_n({\bA}_N, {\bW}_N, {\bX}_N)\right]\right| \ge \delta n} \notag\\
%		&\le 2(1-\pN{\cE_{\eps_n,\sQ}}) + 2\exp\left(-\frac{\delta^2 n^2}{{\sf K}_1 \eps_n^{d} N^2}\right) \\
%		&\hspace{10pt}+ 2\exp\left(-\frac{C_N^2}{\sK_2 n^{2a(2 \vee (\beta-1))}N(1 + n^{-1}N)^2}\right),
%	\end{aligned}
%	\label{eq:concentration_gn_function}
%\end{equation}
%where
%\[
%	C_N = \delta n/2 - (1-\pN{\cE_{\eps_n,\sQ}})2\sK_0 N n^{2a((\beta-1)\vee 1)}(1-n^{-1}N
%	%(0 \vee \delta n/2 - (1-\pN{\cE_{\eps_n,\sQ}})2\sK_0 N n^{2a((\beta-1)\vee 1)}(1-n^{-1}N).
%\]
\end{lem}

Another technical lemma we need provides bounds on the conditional expectation of the degree of a vertex, given the number of vertices $N$ and the weight of the vertex.

\begin{lem}\label{lem:bounds_pi_conditional}
For both the lattice and Poisson case, there exist positive constants $\sk, \sK > 0$ such that, for any $1 \le i \le N$ {and $w\geq 1$},
\[
	\sk \left(\tfrac{N-1}{n} \wedge 1\right) w \le \CexpN{D_i}{ W_i = w} \le \sK \left(\tfrac{N-1}{n} \vee 1\right) w.
\]
%\phmar{The fraction of $N/n$ might be replaced later by something different.}
\end{lem}
We postpone the proofs of both lemmas till the end of this section and proceed with the proof of Proposition~\ref{prop:general_concentration_edge_function} based on these technical results.\pagebreak[3]

\begin{proof}[Proof of Proposition~\ref{prop:general_concentration_edge_function}]

Throughout this proof we let $\sf F$ denote the upper bound for the function $f$. 
%Also, abusing notation a bit we define the function \cmar{Do we need to define $h_n({\ba}_N, {\bx}_N | f)$ or can we just use $\sum_{1 \le i < j \le N} a_{ij} f(d_n(x_i,x_j))$ since it doesn't appear that often? I would prefer not to use this. if we do use $h_n({\ba}_N, {\bx}_N | f)$, I would change the order of the input in $h_n({\bA}_N, {\bW}_N, {\bX}_N | f) $ to $h_n({\bA}_N, {\bX}_N, {\bW}_N| f) $ }
%\begin{equation}
%	h_n({\ba}_N, {\bx}_N | f) := \sum_{1 \le i < j \le N} a_{ij} f(d_n(x_i,x_j)),
%\end{equation}
%which simply removes the condition on the weights and distances from $h_n({\ba}_N, {\bw}_N, {\bx}_N | f)$.
We start by splitting the difference between $h_n$ and the expectation into four parts as follows:
\begin{align}
	&\hspace{-30pt}\left|h_n({\bA}_N, {\bW}_N, {\bX}_N  ) 
		- \frac{n}{2}\E{\lambda_f(W)}\right| \nonumber \\
	&\le \left|h_n({\bA}_N, {\bW}_N, {\bX}_N  ) 
		- g_n({\bA}_N, {\bW}_N, {\bX}_N)\right| \label{eq:order_bound_emain_term1} \\
	&\hspace{10pt}+ \left|g_n({\bA}_N, {\bW}_N, {\bX}_N)
		- \mathbb{E}_N\left[g_n({\bA}_N, {\bW}_N, {\bX}_N)\right]\right| \label{eq:order_bound_emain_term2}\\
	&\hspace{10pt}+ \Big|\mathbb{E}_N\left[g_n({\bA}_N, {\bW}_N, {\bX}_N)\right] 
		- \mathbb{E}_N[\sum_{^{1 \le i < j \le N}} A_{i,j} f(d_n(X_i,X_j))]\Big| \label{eq:order_bound_emain_term3} \\
	&\hspace{10pt}+ \Big|\mathbb{E}_N[\sum_{^{1 \le i < j \le N}} A_{i,j} f(d_n(X_i,X_j))] - \frac{n}{2}\E{\lambda_f(W)}\Big|.
		\label{eq:order_bound_emain_term4}
\end{align}
We aim to show that the probability of exceeding $\delta n/4$ is $o(1)n^{-k(\beta-2)}$ for each of these four terms. Proposition~\ref{prop:general_concentration_edge_function} directly follows from this.\smallskip

\noindent\textbf{Bound on \eqref{eq:order_bound_emain_term1}.} For this term we have
\begin{align*}
	\left|h_n({\bA}_N, {\bW}_N, {\bX}_N ) - g_n({\bA}_N, {\bW}_N, {\bX}_N)\right|
	&= h_n({\bA}_N, {\bW}_N, {\bX}_N) (1-\Itwo{\cE_{\eps_n,M}}).
\end{align*}
Moreover, because $h_n$ is uniformly bounded by ${\sf F} N^2$ we get, using~\eqref{eq:conditional_chernoff_bound_cubes} and Markov's inequality,
\begin{align*}
\pN{h_n({\bA}_N, {\bW}_N, {\bX}_N) (1-\Itwo{\cE_{\eps_n,M}}) \ge \delta n/{4}}
	&\le \frac{{4}{\sf F}N^2}{\delta n} \left(1 - \pN{\cE_{\eps_n,M}}\right)\\
	&\le \frac{4c{\sf F}}{\delta}N^{2 - M^2/3} n^{\gamma d(1 + M^2/3) - 1}.
\end{align*}
For Poisson models, picking $M$ sufficiently large gives us 
\begin{align*}
    & \p{\left|h_n({\bA}_N, {\bW}_N, {\bX}_N  ) 
		- g_n({\bA}_N, {\bW}_N, {\bX}_N  )\right| \ge \delta n/4}\numberthis\label{eq:order_bound_emain_t1}\\
  & \leq \E{\pN{h_n({\bA}_N, {\bW}_N, {\bX}_N  ) (1-\Itwo{\cE_{\eps_n,M}}) > \delta n/4}\I{N\in I_M(n)}} 
  	+ \p{N\notin I_M(n)}\\
  & \leq \frac{4cF}{\delta} n^{\gamma d ( 1 +M^2/3)-1}\E{N^{2-M^2/3}\I{N\in I_M(n)}}+ \p{N\notin I_M(n)}.
\end{align*}
As $N \in I_M(n)$ implies $N = n(1+o(1))$, the {first term}{} of \eqref{eq:order_bound_emain_t1} can further be bounded from above by a constant multiple of $n^{-M^2(1-\gamma d)/3}$. Using that $0<\gamma<1/d$ and \eqref{eq:conc_bound_nb_ppp}, we conclude that \eqref{eq:order_bound_emain_t1} equals $o(n^{-k(\beta - 2)})$. The same conclusion holds for lattice models, as here $N=n$.
\smallskip

\noindent \textbf{Bound on \eqref{eq:order_bound_emain_term2}.} For this term we use Lemma~\ref{lem:concentration_gn_function} together with \eqref{eq:conc_bound_nb_ppp} to get, for sufficiently large $M$,
\begin{align*}
&\hspace{-20pt}\p{\left|g_n({\bA}_N, {\bW}_N, {\bX}_N  )
		- \mathbb{E}_N\left[g_n({\bA}_N, {\bW}_N, {\bX}_N  )\right]\right| \ge \delta n/4}\\
    &\le \E{2\exp\left(-\frac{{\sf K}_1 \delta^2}{16}\eps_n^{-d}\right) \I{{N\in I_M(n)}}} \\
	&\hspace{10pt} + \E{2\exp\left(-\frac{\sK_2 \delta^2}{16} n^{1 - 2a(2 \vee (\beta-1))}\right)
		\I{{N\in I_M(n)}}} \\
	&\hspace{10pt} + \p{N \notin I_M(n)} + 2(1-\p{\cE_{\eps_n,M}}).
\end{align*}
Observe that $\eps_n^d =n^{-\gamma}$ where $0<\gamma <1/d$, so that the first term is $o(n^{-k(\beta - 2)})$. Moreover, since by assumption $a(2 \vee (\beta-1)) < 1/2$, the same holds for the second term. Finally, {$\p{N \notin I_M(n)} = o(n^{-k(\beta - 2)})$ when we pick}{} $M$ large enough, while this also holds for the last term due to~\eqref{eq:chernoff_bound_cubes}. We thus conclude that all four terms in the upper bound  go to zero sufficiently fast.\smallskip

\noindent \textbf{Bound on \eqref{eq:order_bound_emain_term3}.} We will prove that, for all $\delta>0$,
\begin{align}\label{eq:prob_order_bound_t3}
\mathbb P\Big( \Big|\mathbb E_N\big[ \!\!\!\!\sum_{_{1 \le i < j \le N}}\!\! A_{i,j} f(d_n(X_i,X_j))\Big] - \EN{g_n({\bA}_N, {\bW}_N, {\bX}_N)}\Big| > \delta n\Big) = o(n^{-k(\beta -2)}).
\end{align}
Observe that $g_n({\bA}_N, {\bW}_N, {\bX}_N) \le \sum_{1 \le i < j \le N} A_{i,j} f(d_n(X_i,X_j)).$ Therefore,
\begin{align*}
	&\Big|\mathbb E_N\Big[\sum_{_{1 \le i < j \le N}} A_{i,j} f(d_n(X_i,X_j))\Big]
		- \EN{g_n({\bA}_N, {\bW}_N, {\bX}_N)} \Big| \numberthis \label{eq:order_bound_emain_term3_part1} \\
	&= \mathbb E_N\Big[\sum_{_{1 \le i < j \le N}} A_{i,j} f(d_n(X_i,X_j)) - g_n({\bA}_N, {\bW}_N, {\bX}_N  )\Big]\\
	&\le \mathbb E_N\Big[ (1-\Itwo{\cE_{\eps_n, M}}) {\sf F}\sum_{_{1 \le i < j \le N}} A_{i,j}\Big] + {\sf F}\sum_{_{1 \le i < j \le N}} \EN{A_{i,j} \I{d_n(X_i,X_j)> \eps_n n^{1/d}}
		\I{W_i \vee W_j \le n^a}}\\
	&\hspace{10pt}+ {\sf F}\sum_{_{1 \le i < j \le N}} \EN{A_{i,j} \I{W_i \vee W_j > n^a}}.
\end{align*}
%\cmar{Do we need the factor $2$? We could use $\Itwo{A \cup B} = \Itwo{A\cap B^c}+\Itwo{B}$}
%\sout{Here the factor $2$ comes from the fact that we chose to bound $\I{d_n(X_i,X_j) \le \eps_n n^{1/d}} \I{W_i \vee W_j > n^a}$ by $\I{W_i \vee W_j > n^a}$.} \sout{Again, we will bound each term separately, starting with the first term. Here we bound $\sum_{1 \le i < j \le N} A_{ij} \le \binom{N}{2}$ to get (...)} 
The first term on the right-hand side can be bounded by ${\sf F}N^2[1-\p{\cE_{\eps_n, M}}]$, using that $\sum_{1 \le i < j \le N} A_{i,j} \le N^2$.
For the second term on the right-hand side, we note that by~\eqref{asspt:kernel_bounds_var} and \eqref{asspt:profile_bounds}, 
when $W_1 \vee W_2 \leq n^a$ and $d_n(X_1,X_2)> \eps_n n^{1/d}$,
    \begin{align*}
        p_{X_1,X_2}(W_1, W_2) \leq {\sf C} \left(\frac{\kappa(W_1,W_2)}{\eps_n^d n}\right)^{\alpha} \leq {\sf C}^{1+\alpha} \left(\frac{n^{\xi - 1}}{\eps_n^d}\right)^{\alpha}
        \le {\sf C}^{1 + \alpha} n^{-\alpha(1 - \xi - \gamma)},
    \end{align*}
    {where $\xi:=a (2\vee (\beta-1))$}.
It follows that
\begin{equation}\label{eq:order_bound_emain_term3_part2}
	 \mathbb{E}_N\Big[ {\sf F}\sum_{_{1 \le i < j \le N}} A_{i,j} \I{d_n(X_i, X_j)> \eps_n n^{1/d}}\I{W_i \vee W_j \le n^a}\Big] \le {\sf F}{\sf C}^{1 + \alpha} N^2 n^{-\alpha(1 - \xi - \gamma)}.
\end{equation}
%Next we note that by assumption $\xi < (1-\gamma -1/\alpha)$ and hence $\alpha(1 - \xi - \gamma) > 1$. 
Finally, we have
\begin{align*}
	\mathbb E_N\Big[ \sum_{_{1 \le i < j \le N}}A_{i,j} \I{W_i \vee W_j > n^a}\Big]
	\le \mathbb E_N\Big[ \sum_{i = 1}^N D_i \I{W_i > n^a}\Big]
	= N \EN{D_1 \I{W_1 > n^a}}.
\end{align*} 
{By} Lemma~\ref{lem:bounds_pi_conditional}, it follows that $\mathbb{E}_N[D_1 \I{W_1 > n^a} | W_1] \le {\sf K}_1 \frac{N}{n} W_1 \I{W_1 > n^a}$ and hence 
\[
	\EN{D_1 \I{W_1 > n^a}} \le {\sf K}_1 \frac{N}{n} \int_{n^a}^{\infty} w f_W(w) dw = \frac{{\sf K}_1(\beta -1)}{\beta-2} \frac{N}{n} n^{-a(\beta-2)}.
\]
Therefore,
\begin{equation}\label{eq:order_bound_emain_term3_part3}
	\mathbb{E}_N\Big[2 {\sf F} \sum_{_{1 \le i < j \le N}} A_{i,j} \I{W_i \vee W_j > n^a}\Big] 
	\le \frac{2{\sf F}{\sf K}_1(\beta -1)}{\beta-2} N^2 n^{-1-a(\beta-2)}.
\end{equation}
%\cmar{maybe add labels to the results we use?}
Putting~\eqref{eq:order_bound_emain_term3_part2} and~\eqref{eq:order_bound_emain_term3_part3} back into the upper bound for~\eqref{eq:order_bound_emain_term3_part1} we get
\begin{align*}
	&\hspace{-30pt}\Big|\mathbb{E}_N\Big[\sum_{_{1 \le i < j \le N}} A_{i,j} f(d_n(X_i,X_j))\Big] 
		- \mathbb{E}_N\left[g_n({\bA}_N, {\bW}_N, {\bX}_N  )\right] \Big|\\
	&\le \frac{{\sf F}}{2}n N^{2-M/4}
		+ \frac{{\sf F}{\sf C}^{1 + \alpha}}{2} N^2 n^{-\alpha(1 - \xi - \gamma)}
		+ \frac{2{\sf F}{\sf K}_1(\beta -1)}{\beta-2} N^2 n^{-1-a(\beta-2)}.
\end{align*}
Now recall that on the event $N\in I_{M}(n)$, it holds that $N = (1+o(1))n$. 
Therefore, on this event, the right-hand side can be bounded from above by a constant multiple of $n \left( n^{2-M/4} + n^{1-\alpha(1 - \xi - \gamma)} + n^{-a(\beta-2)}\right)$. By selecting $M$ large enough, this is~$o(n)$, using our assumptions that $\beta - 2>0$ and $\xi < (1-\gamma - 1/\alpha)$ and therefore $\alpha(1 - \xi - \gamma)>1$.
Hence, by picking a sufficiently large $M$, we use~{\eqref{eq:conc_bound_nb_ppp}} to conclude that
\begin{align*}
	&\hspace{-30pt}\mathbb P\Big( \Big|\mathbb{E}_N\Big[g_n({\bA}_N, {\bW}_N, {\bX}_N )\Big] 
		- \mathbb{E}_N\Big[\sum_{_{1 \le i < j \le N}} A_{i,j} f(d_n(X_i,X_j))\Big]\Big| \ge \delta n/4\Big)\\
	&\le {\p{N\notin I_{M}(n)}} = o(n^{-k(\beta -2)}).
\end{align*}

\noindent \textbf{Bound on \eqref{eq:order_bound_emain_term4}.} For the final term, 
%$f:(0,\infty) \to [0,\infty)$ a continuous function with compact support, 
recall $\lambda_f(w)$ from \eqref{def:lambda_over_rd} and $\lambda(w,z)$ from \eqref{def:lambda_w_z}, for $w\geq 1$ and $z\in \R^d$.
Let us write
\[
	\lambda_f^{n}(w) := \begin{cases}
		\sum_{z \in \Q_n \cap \mathbb{Z}^d} f(\|z\|) \lambda (z,w) &\text{in the lattice case,} \\
		\int_{\Q_n} f(\|z\|) \lambda (z,w) \,  dz &\text{in the Poisson case.}
	\end{cases}
\]
Recalling that $\Q_n$ is the centred cube of volume $n$, we observe 
\[
	\mathbb{E}_N\Bigg[\sum_{_{1 \le i < j \le N}} A_{i,j} f(d_n(X_i,X_j))\Bigg]
	= \begin{cases}
		\frac{n-1}{2} \E{\lambda_f^n(W)} &\text{for lattice models,} \\
		\frac{N(N-1)}{2n} \E{\lambda_f^n(W)} &\text{for Poisson models,}
	\end{cases}
\]
where in the lattice case we use that $N = n$. Note that \smash{$\lambda_f^n(w)$} converges pointwise to $\lambda_f(w)$ as $n$ tends to infinity. Together with the fact that $\mathbb E[\lambda_f(W)] < \infty$, the dominated convergence theorem implies that $\lim_{n\rightarrow\infty}\mathbb E[\lambda_f^n(W)] = \E{\lambda_f(W)}$.
It follows that for the Poisson case, on the event $N \in I_M(n)$, 
\begin{align*}
    \frac{1}{n}\mathbb E_N\Big[\sum_{_{1 \le i < j \le N}} A_{i,j} f(d_n(X_i,X_j))\Big] = (1+o(1))\frac{1}{2}\E{\lambda_f(W)},
\end{align*}
since $N = (1+o(1))n$. The same equality holds for the lattice case.
% , since $|\Q_n \cap \mathbb{Z}^d| = n$,
% \pmar{If this was always true, why not write it straight away?} the same equality holds.
%that $$\frac{1}{n}\EN{\sum_{1 \le i < j \le N} A_{ij} f(d_n(X_i,X_j))} = (1+o(1))\frac{1}{2}\E{\lambda_f(W)}.$$
For Poisson models, this implies that for all $\delta>0$ and sufficiently large $n$,
\begin{align*}
    \mathbb P\Big( \Big|\frac{1}{n}\mathbb E_N\Big[ \sum_{_{1 \le i < j \le N}} A_{i,j} f(d_n(X_i,X_j))\Big] - \frac{1}{2}\E{\lambda_f(W)}\Big|> \delta/4,\, N\in I_M(n)\Big)= 0.
\end{align*}
Using \eqref{eq:conc_bound_nb_ppp} we can conclude that
\begin{align*}
      & \mathbb P\Big( \Big|\frac{1}{n}\mathbb E_N\Big[ \sum_{_{1 \le i < j \le N}} A_{i,j} f(d_n(X_i,X_j))\Big] - \frac{1}{2}\E{\lambda_f(W)}\Big|> \delta/4\Big) \\ & \leq \p{N \notin I_M(n)} = o(n^{-k(\beta-2)}).
\end{align*}
by choosing $M$ sufficiently large. The bound above also holds for lattice models. This deals with the last term~\eqref{eq:order_bound_emain_term4} and therefore finishes the proof of Proposition~\ref{prop:general_concentration_edge_function} subject to Lemmas~\ref{lem:concentration_gn_function} and ~\ref{lem:bounds_pi_conditional}.
\end{proof}

\subsection{Proofs of technical lemmas}

The rest of this section is dedicated to the proof of Lemmas~\ref{lem:concentration_gn_function} {and~\ref{lem:bounds_pi_conditional}}. Since we condition on the number of vertices $N$, the graph $G_n$ will be based on a vector ${\bW}_N = (W_i)_{1 \le i \le N}$ of i.i.d.\,Pareto weights and ${\bX}_N = (X_i)_{1 \le i \le N}$ of uniform positions in $\T_n^d$ for the Poisson case and $(X_i)_{1 \le i \le N}$ being just the locations of the lattice points of $\T_n^d$} for the lattice case. {The key tool in proving Lemma~\ref{lem:concentration_gn_function} lies in applying McDiarmid's inequality~\cite{mcdiarmid_1989} twice.
\smallskip

We begin by recalling the version of McDiarmid's inequality we shall use, given in~\cite{combes2015extension}. Consider a function $g: \mathcal{Y}_1 \times \dots \times \mathcal{Y}_m \to \mathbb{R}$, where each $\mathcal{Y}_i$ is a probability space. Moreover, let $\mathcal{G} \subseteq \mathcal{Y}_1 \times \dots \times \mathcal{Y}_m$ be in the domain of $g$. We call this the \emph{good event} and say that $g$ satisfies the \emph{bounded difference inequality on $\mathcal{G}$} with non-negative constants $(c_1,\dots,c_m)$ if for every $i \in \{1, \ldots, m\}$ and for all $(y_1,\ldots, y_m), (y_1^\prime, \ldots, y_m^\prime)\in \mathcal{G}$ with $y_j = y_j^\prime$ for $j \ne i$,
\begin{align}\label{eq:bound_diff}
    |g(y_1, \dots, y_m) - g(y_1^\prime, \dots, y_m^\prime)| \le c_i.
\end{align}
For such a function $g$, for any independent random variables $Y_1, \dots, Y_m$ on the spaces $\mathcal{Y}_1, \dots, \mathcal{Y}_m$ and for all $\delta>\p{\mathcal{G}^c}\sum_{i=1}^m c_i$, it holds that 
\begin{equation}\label{eq:def_McDiarmid}
	\begin{aligned}
		&\p{|g(Y_1, \dots, Y_m)- \E{g(Y_1, \dots, Y_m)\mid (Y_1, \dots, Y_m) \in \mathcal{G}}| > \delta} \\
		&\le 2\p{\mathcal{G}^c} 
		+ 2\exp\Big(-\tfrac{2 \left(\delta - \p{\mathcal{G}^c}\sum_{i=1}^m c_i\right)^2}{\sum_{i = 1}^m c_i^2}\Big).
	\end{aligned}
\end{equation}

For the proof we apply McDiarmid's inequality to $|g_n - \CexpN{g_n}{\bW_N, \bX_N}|$ and also to $|\CexpN{g_n}{\bW_N, \bX_N} - \EN{g_n}|$, where we write $g_n$ for $g_n({\bA}_N, {\bW}_N, {\bX}_N)$ for notational simplicity.} The first term considers the probability conditioned on the weights and edges. Under this conditioning the edge indicators are independent so that McDiarmid's inequality can be applied without specifying a good event. The good event is however needed to apply McDiarmid's inequality to the second term, where the good event will be $\cE_{\eps_n,M}$. \smallskip

Both these steps will yield an upper bound in terms of $N$ that is $o(n^{-k(\beta - 2)})$ when $N = (1 + o(1))n$, which holds for the lattice as well as the Poisson case with a probability that converges to $1$, by~\eqref{eq:conc_bound_nb_ppp} with an error of order $o(n^{-k(\beta -2)})$.

\begin{proof}[Proof of Lemma~\ref{lem:concentration_gn_function}]
Recall $g_n(\cdot)$ from \eqref{def:gn_function}. Following the discussion above, we bound the probability on the left hand side {of \eqref{eq:concentration_gn_function}}, for $\delta>0$, by
\begin{align}
	&\hspace{-10pt}\pN{\left|g_n(\bA_N, \bW_N, \bX_N) - \mathbb{E}_N[g_n(\bA_N, \bW_N, \bX_N)]\right| \ge n\delta} 
	\nonumber\\
		&\le \pN{\left|g_n(\bA_N, \bW_N, \bX_n) - \mathbb{E}_N[g_n(\bA_N, \bW_N, \bX_N) \, | \, {\bW}_N, {\bX}_N]\right| 
			\ge n\delta/2} \label{eq:conc_gn_bound_one}\\
		&\hspace{10pt}+ \pN{\left|\mathbb{E}_N[g_n(\bA_N, \bW_N, \bX_N)\, |\,  {\bW}_N, {\bX}_N] 
			- \mathbb{E}_N[g_n(\bA_N, \bW_N, \bX_N)]\right| \ge n\delta/2}.\label{eq:conc_gn_bound_two}
\end{align}

\paragraph{\bf Concentration bulk edges given the weights and positions.}
We begin with a bound on the term \eqref{eq:conc_gn_bound_one}. Choose $m=\binom{N}{2}$ and let $\cY_{(i,j)}$ be the canonical probability space $\{0,1\}$ for $1 \leq i<j\leq N$. Conditionally on the weights $\bW_N$ and the positions $\bX_N$, we consider $g_n$ as a function on \smash{$\{0,1\}^{m}$}.\smallskip

For the application of McDiarmid's inequality we take the good event to be the entire space, i.e., \smash{$\mathcal{G} = \{0,1\}^{m}$} so that in particular $\p{\cG^c} = 0$. Denote by ${\bA}_N^{_{(i,j)}}$ the vector obtained from ${\bA}_N$ where we replace $A_{i,j}$ by $A'_{i,j} = 1-A_{i,j}$ {and leave all other coordinates unchanged.}
Then, conditionally on $\bW_N$ and $\bX_N$, the function $g_n$ satisfies, for all $1\leq i<j\leq N$, the bounded difference inequality
\begin{align*}
    |g_n(\bA_N,\bW_N, \bX_N) -  & g_n(\bA_N^{(i,j)},\bW_N, \bX_N)| \\  & \leq {\sf F}\Itwo{\cE_{\eps_n,M}}|A_{i,j} - A_{i,j}'|\I{d_n(X_i, X_j) \le  \eps_n n^{1/d}}\I{W_i \vee W_j \leq n^a}\\
    & \leq {\sf F}\Itwo{\cE_{\eps_n,M}}\I{d_n(X_i, X_j) \le  \eps_n n^{1/d}} =: c_{ij}.
\end{align*}
The sum of the square of differences is given by 
\[
	\sum_{_{1 \le i < j \le N}} c_{ij}^2 = {\sf F}^2 \sum_{_{1 \le i < j \le N}} \Itwo{\cE_{\eps_n,M}}\I{d_n(X_i, X_j) \le \eps_n n^{1/d}}.
\] 
On the event $\cE_{\eps_n,M}$, the number of points in every cube of {side} length $\eps_n n^{1/d}$ is at most $2 N \eps_n^d$, and any ball of radius $R$ is covered by at most $(2 R n^{-1/d} \eps_n^{-1})^d$ such cubes. This implies that, for any $y \in \T_n^{{d}}$,
\begin{equation}\label{eq:number_points_ball_on_good_event}
	\Itwo{\cE_{\eps_n,M}} \sum_{j = 1}^N \I{d_n(X_j, y) \le R} \le (2 R n^{-1/d} \eps_n^{-1})^d \cdot 2 N \eps_n^d =  2^{d+1} R^d n^{-1} N.
\end{equation}
Taking $R = \eps_n n^{1/d}$ it follows that $\sum_{1 \le i < j \le N} c_{ij}^2 \le {{\sf F}^2}2^{d+1} N^2 \eps_n^d$ and hence McDiarmid's inequality~\eqref{eq:def_McDiarmid} implies that 
\begin{equation}\label{eq:order_bound_A_main}
	\begin{aligned}
		&\mathbb{P}_N(\left|g_n({\bA}_N,\bW_N, \bX_N) - \mathbb{E}_N[g_n({\bA}_N,\bW_N, \bX_N) \, | \, {\bW}_N, {\bX}_N]\right| \ge n\delta/2 \, | \, {\bW}_N, {\bX}_N)\\
		&\le \exp \left(-\tfrac{\delta^2 n^2}{2^{d+2} {\sf F}^2 \eps_n^d N^2}\right).
	\end{aligned}
\end{equation}
%\cmar{shouldn't it be $2^{d+2}$ instead of $2^{d+3}$?}
\medskip

\paragraph{\bf Concentration of bulk for weights and positions.} We next establish an upper bound for \eqref{eq:conc_gn_bound_two} by again an application of McDiarmid's inequality. For the setup, let $m=N$,   $\mathcal{Y}_i = [1,\infty) \times \R_+$ for all $i\in \{1,\ldots, N\}$, and take $\cE_{\eps_n,M}$ as the good event. We define the function $\hat{g}_n \colon [1,\infty)^N \times \R_+^N \to \mathbb{R}_+$, for $\bw_N \in [1,\infty)^N$ and $\bx_N \in \R_+^N$, as
\begin{align}
   \hat{g}_n({\bw}_N, {\bx}_N) & := \CexpN{g_n(\bA_N, \bw_N, \bx_N)}{\bW_N = \bw_N, \bX_N = \bx_N}\\
   & = \Itwo{\cE_{\eps_n,M}}\sum_{_{1 \le i < j \le N}} f(d_n(x_i,x_j))p_{x_i, x_j}(w_i,w_j) \I{d_n(x_i,x_j)\le \eps_n n^{1/d}} \I{w_i \vee w_j \le n^a}.\notag
\end{align}
and observe that $\CexpN{g_n(\bA_N, \bW_N, \bX_n)}{{\bW}_N, {\bX}_N} = \hat{g}_n({\bW}_N, {\bX}_N).$
%In addition, since $g_n(\bA_N, \bW_N, \bX_N) = 0 = \Itwo{\cE_{\eps_n,M}}\hat{g}_n({\bW}_N, {\bX}_N)$ on the complement of the event $\cE_{\eps_n,M}$,
%\[
%	\EN{g_n(\bA_N, \bW_N, \bX_N)} = \CexpN{\hat{g}_n({\bW}_N, {\bX}_N)}{ \cE_{\eps_n,M}}.
%\]
In addition, we note that
\begin{align*}
	\EN{g_n(\bA_N, \bW_N, \bX_N)} = \EN{\hat{g}_n(\bW_N, \bX_N)}
 & = \EN{\hat{g}_n(\bW_N, \bX_N)\Itwo{\cE_{\eps_n,M}}} \\&= \EN{\hat{g}_n(\bW_N, \bX_N) \mid \cE_{\eps_n, M}}\pN{\cE_{\eps_n, M}}.
\end{align*}

We claim that $\hat{g}_n$ satisfies the bounded difference inequality on the event $\cE_{\eps_n,M}$. To see this, for any two vectors ${\bw}_N$ and ${\bx}_N$ of weights and positions, denote by ${\bw}_N^{_{(i)}}$ and ${\bx}_N^{_{(i)}}$ the vectors where the $i^{\rm th}$ entry is replaced {by}{} any value $w_i^\prime \in [1,\infty)$ and $x_i^\prime \in \R_+$, {and all other coordinates are unchanged.}{} Then, 
\begin{align*}
	&\hspace{-30pt}\left|\hat{g}_n({\bw}_n, {\bx}_n) - \hat{g}_n({\bw}_n^{(i)}, {\bx}_n^{(i)})\right|\\
	&\le \Itwo{\cE_{\eps_n,M}} {\sf F}
		\sum_{{{1\leq j\leq N}\atop{j\neq i}}} p_{x_i, x_j}(w_i,w_j) \I{d_n(x_i,x_j)\le \eps_n n^{1/d}} \I{w_i \vee w_j \le n^a}\\
	&\hspace{10pt}+ \Itwo{\cE_{\eps_n,M}} {\sf F} \sum_{{1\leq j\leq N}\atop{j\neq i}}
		p_{x_i', x_j}(w_i',w_j) \I{d_n(x_i',x_j)\le \eps_n n^{1/d}} \I{w_i' \vee w_j \le n^a}.
\end{align*}
We will show that there is a constant $\sK_0 > 0$ that does not depend on $w_i$ and $x_i$ such that
\begin{align}
	\Itwo{\cE_{\eps_n,M}} {\sf F}  \sum_{j \ne i}  p_{x_i, x_j}(w_i,w_j) &
	\I{d_n(x_i,x_j) \le \eps_n n^{1/d}} \I{w_i \vee w_j \le n^a} \notag\\ 
 & \le \sK_0 n^{2a ((\beta -2)\vee 1)} (1+ n^{-1} N).
 \label{eq:order_bound_W_X_main_bound}
\end{align}
The other sum satisfies the same upper bound, which implies that $\hat{g}_n(\cdot)$ satisfies a bounded difference inequality on the event $\cE_{\eps_n,M}$ with bounds 
$
	c_i = 2 \sK_0 n^{a ((\beta -2)\vee 1)} (1+ n^{-1} N) 
$
and thus 
\[
	\sum_{i=1}^N c_i^2 \le 4 \sK_0^2 n^{2a (2 \vee (\beta -1))} N (1+ n^{-1} N)^2.
\]
We will first finish the application of McDiarmid's inequality~\eqref{eq:def_McDiarmid} using~\eqref{eq:order_bound_W_X_main_bound}. For this, we note that on the event $N \in I_M(n)$, it holds for $n$ sufficiently large that
\[
	\frac{n\delta}{2}  - (1-\p{\cE_{\eps_n,M}}) 2 \sK_0 n^{a (2 \vee (\beta -1))} (1+ n^{-1} N) > \frac{\delta}{4}n,
\]
which implies $n\delta/2 > (1-\p{\cE_{\eps_n,M}}) \sum_{i=1}^N c_i$.
Moreover, there exists a constant $Q$ such that
\[
	4 \sK_0^2 n^{2a (2 \vee (\beta -1))} N (1+ n^{-1} N)^2 \le Q n^{1+2a(2 \vee (\beta - 2))}.
\]
Therefore, using that $\hat{g} \le {\sf F} N^2$, we get
\begin{align*}
	&\mathbb{P}_N(\left|\mathbb{E}_N[g_n(\bA_N, \bW_N, \bX_n)\, |\,  {\bW}_N, {\bX}_N] 
		- \mathbb{E}_N[g_n(\bA_N, \bW_N, \bX_n)]\right| \ge n\delta/2)\\
	&\le \pN{\left|\hat{g}_n(\bW_N,\bX_N) - \EN{\hat{g}_n(\bW_N, \bX_N) \mid \cE_{\eps_n, M}}\right| \ge n\delta/4} \\
	&\hspace{10pt}+ \pN{{\sf F}N^2(1-\pN{\cE_{\eps_n, M}})\ge n\delta/4}.
\end{align*}
Noting that on the event $I_M(N)$ we have $N^{-2} \ge n^{-2}/2$ it follows from~\eqref{eq:conditional_chernoff_bound_cubes} that for $n$ sufficiently large the last term is zero. Thus, by applying \eqref{eq:def_McDiarmid} to the first part we get
\begin{align*}
	&\mathbb{P}_N(\left|\mathbb{E}_N[g_n(\bA_N, \bW_N, \bX_n)\, |\,  {\bW}_N, {\bX}_N] 
		- \mathbb{E}_N[g_n(\bA_N, \bW_N, \bX_n)]\right| \ge n\delta/2)\\
	&\hspace{10pt}\le 2(1-\pN{\cE_{\eps_n,M}}) 
		+ 2\exp\left(-\frac{\frac{\delta^2}{8} n^2}{Q n^{1 + 2a(2\vee (\beta -2))}}\right).
\end{align*}
Combining this with~\eqref{eq:order_bound_A_main}, the main result~\eqref{eq:concentration_gn_function} follows by picking $\sK_1 = 2^{d+2} {\sf F}$ and $\sK_2 = (8Q)^{-1}$. 
It only remains to establish the  bound~\eqref{eq:order_bound_W_X_main_bound}. We first use~\eqref{asspt:profile_bounds} to get
\begin{align*}
	&p_{X_i, X_j}(W_i,W_j) \le \left(1 \wedge \sC \left(\tfrac{\kappa(W_i,W_j)}{d_n(X_i,X_j)^d}\right)^\alpha \right)\\
	&= \I{d_n(X_i,X_j) \le \sC^{{1/(\alpha d)}} \kappa(W_i,W_j)^{1/d}} +  \I{d_n(X_i,X_j) > \sC^{{1/(\alpha d)}} \kappa(W_i,W_j)^{1/d}}\sC \left(\tfrac{\kappa(W_i,W_j)}{d_n(X_i,X_j)^d}\right)^\alpha.
\end{align*}

Thus, the left-hand side of \eqref{eq:order_bound_W_X_main_bound} can be bounded from above by a constant multiple of 
\begin{align}
    & \Itwo{\cE_{\eps_n,M}} \sum_{{1\leq j\leq N}\atop{j \ne i}} 
		\I{d_n(X_i,X_j) \le \sC^{1/(\alpha d)} \kappa(W_i,W_j)^{1/d} \wedge \eps_n n^{1/d}}
		\I{W_i \vee W_j \le n^a} \label{eq:order_W_x_one}\\
	&+ \Itwo{\cE_{\eps_n,M}}  \sum_{{1\leq j\leq N,}\atop{j \ne i}} 
		\I{\sC^{1/(\alpha d)} \kappa(W_i,W_j)^{1/d} < d_n(X_i,X_j) \le \eps_n n^{1/d}}
		\sC \left(\tfrac{\kappa(W_i,W_j)}{d_n(X_i,X_j)^d}\right)^\alpha
		\I{W_i \vee W_j \le n^a}.\label{eq:order_W_x_two}
\end{align}
We begin by studying \eqref{eq:order_W_x_one}. For any $v,w\geq 1$ satisfying $w\vee v \leq n^a$, \eqref{asspt:kernel_bounds_var} gives us that $\kappa(v,w) \leq \sC n^{a + a(1 \vee (\beta -2))} = \sC n^{a(2\vee (\beta -1))}$. Recall that we chose $\eps_n$ and $a$ to satisfy the conditions of Proposition~\ref{prop:order_bound_emain}, that is $\eps_n^d = n^{-\gamma}$ and $a(2\vee (\beta -1)) < (1-\gamma - 1/\alpha)\wedge \tfrac{1}{2}$. Thus, if  $w\vee v \leq n^a$, then for $n$ sufficiently large $\sC^{1/\alpha d} \kappa(v,w)^{1/d} < \eps_n n^{1/d}$. Therefore, for large $n$, \eqref{eq:order_W_x_one} can be bounded from above by 
\begin{align*}
    & \Itwo{\cE_{\eps_n,M}} \sum_{{1\leq j\leq N}\atop{j \ne i}} 
		\I{d_n(X_i,X_j) \le \sC^{1/(\alpha d)} \kappa(W_i,W_j)^{1/d}}
		\I{W_i \vee W_j \le n^a} \\
  & \leq \Itwo{\cE_{\eps_n,M}} \sum_{{1\leq j\leq N}\atop{j \ne i}} 
		\I{d_n(X_i,X_j) \le \sC^{\frac{1}{d}(1+ \frac{1}{\alpha})} n^{\frac{1}{d}a(2\vee (\beta -1))}}.
\end{align*}
Applying \eqref{eq:number_points_ball_on_good_event} with $R = \sC^{\frac{1}{d}(1+ \frac{1}{\alpha})} n^{\frac{1}{d}a(2\vee (\beta -1))}$, the above can be %further upper 
bounded by \smash{$2^{d+1} \sC^{1+ \frac{1}{\alpha}} $} \smash{$n^{a(2\vee (\beta -1))}n^{-1}N$}. To conclude  \eqref{eq:order_bound_W_X_main_bound}, it therefore remains to show that the second summand \eqref{eq:order_W_x_two} is bounded from above by a constant multiple of $n^{a(1\vee(\beta-2))}$.
\smallskip

For \eqref{eq:order_W_x_two} observe that, since $d_n$ is translation invariant,
\begin{align*}
	&\hspace{-30pt}\sum_{j\ne i}  \I{\sC^{1/(\alpha d)} \kappa(W_i,W_j)^{1/d} < d_n(X_i,X_j) \le \eps_n n^{1/d}}	d_n(X_i,X_j)^{-d\alpha}\\
	&= \sum_{j \ne i} \I{\sC^{1/(\alpha d)} \kappa(W_i,W_{j})^{1/d} < \|Z_{j}\| \le \eps_n n^{1/d}}\|Z_{j}\|^{-d\alpha},
\end{align*}
where $Z_{j}=\Phi(X_j)$ 
for $\Phi$ the canonical mapping of $\mathbb{T}_n^d$ to
$\Q_n$ with $\Phi(X_i)=0$ recalling that $\Q_n$ is the $d$-dimensional cube of volume $n$ centred at the origin. %The latter sum can be interpreted as a $d$-dimensional Riemann sum. 
In particular, there exist constants $\sQ_1, \sQ_2, \sQ_3 > 0$ such that for sufficiently large $n$
\begin{align*}
	&\sum_{j\ne i}  \I{\sC^{1/(\alpha d)} \kappa(W_i,W_j)^{1/d} < d_n(X_i,X_j) \le \eps_n n^{1/d}}	d_n(X_i,X_j)^{-d\alpha}\\
	&\le \sQ_1 \int_{\Q_n} \I{\sC^{1/(\alpha d)} \kappa(W_i,W_j)^{1/d} < \|z\| \le \eps_n n^{1/d}} \|z\|^{-d\alpha} dz
	= \sQ_2 \int_{\sC^{1/(\alpha d)} \kappa(W_i,W_j)^{1/d}}^{\eps_n n^{1/d}} r^{-d(\alpha-1) -1} dr\\
	&\le \sQ_3 \kappa(W_i,W_j)^{-(\alpha-1)}.
\end{align*}
Plugging this back into the sum \eqref{eq:order_W_x_two}, and using~\ref{asspt:kernel_bounds_var}, we get
\begin{align*}
	&\hspace{-20pt}\Itwo{\cE_{\eps_n,Q}} {\sf F} \sum_{j \ne i} 
		\I{\sC^{1/(\alpha d)} \kappa(W_i,W_j)^{1/d} < d_n(X_i,X_j) \le \eps_n n^{1/d}}
		\sC \left(\frac{\kappa(W_i,W_j)}{d_n(X_i,X_j)^d}\right)^\alpha
		\I{W_i \vee W_j \le n^a}\\
	&\le {\sf F} \sQ_2 \sC \kappa(W_i,W_j) \I{W_i \vee W_j \le n^a}	\le {\sf F} \sQ_2 \sC n^{a(2 \vee (\beta-1))}.
\end{align*}
Taking $\sK_0 = {\sf F} \sQ_2 \sC \vee {\sf F} 2^{d+1} \sC^{\alpha+1}$ now yields~\eqref{eq:order_bound_W_X_main_bound}.
\end{proof}

The proof of the last technical lemma rests on the following easy bound.

\begin{lem}\label{lem:calc}
 There exist constants $\sK \ge {\sf k} > 0$ such that
\begin{align}\label{eq:bounds_pi_key}
    {\sf k} w \le (\beta -1) \int_1^{\infty} \int_{\Q_n} \varphi\left(\frac{\|z\|^d}{\kappa(w,t)}\right) dz  t^{-\beta}dt \le \sK w.
\end{align}
\end{lem}

\begin{proof}
Using~\eqref{asspt:profile_bounds}, for $z\in \R^d$, we bound
\begin{align*}
\sc \I{\|z\|<\sc(\kappa(w,t))^{1/d}} \leq \varphi & \Big(\tfrac{\|z\|^d}{\kappa(w,t)}\Big) \leq \I{\|z\|<\kappa(w,t)^{1/d}}+\sC \left(\tfrac{\kappa(w,t)}{\|z\|^d} \right)^{\alpha}\I{\|z\|\geq (\kappa(w,t))^{1/d}}.
\end{align*}
The integrals  \smash{$\int_{\Q_n} \I{\|z\|<(\kappa(w,t))^{1/d}} dz$}, 
and \smash{$\int_{\Q_n} ({\kappa(w,t)}/{\|z\|^d})^{\alpha}\I{\|z\|\geq (\kappa(w,t))^{1/d}} dz$} for $\alpha > 1$, can both be bounded from above and below by constant multiples of $\kappa(w,t)$, where the constants only depend on the dimension $d$.
Next, by \eqref{asspt:kernel_bounds_var}, 
\[
	\int_1^\infty \kappa(w,t) t^{-\beta}\, dt
	\leq \sC  w \int_1^w t^{(\beta-2)\vee 1} t^{-\beta}\, dt
	+ \sC  w^{(\beta-2)\vee 1} \int_w^\infty   t^{1-\beta}\, dt \leq \sK^\prime  w,
\]
where the last inequality follows from basic computations and holds for some constant $\sK^\prime>1$. For the lower bound, $\kappa(w,t)\geq \sc  (w \vee t)$ by~\eqref{asspt:kernel_bounds_var} and so 
\begin{align*}
    \int_1^\infty \kappa(w,t) t^{-\beta}\, dt \geq \sc w\int_1^{w}t^{-\beta} dt + \sc \int_{w}^{\infty}t^{1-\beta}dt = \sk^\prime w,
\end{align*}
for some constant $\sk^\prime>0$. Combining the above implies \eqref{eq:bounds_pi_key}.
\end{proof}

\begin{proof}[Proof of Lemma~\ref{lem:bounds_pi_conditional}]
We first consider the lattice models, where the number of vertices equals $N=n$.
%the number of vertices $N$ is deterministic and equals $(1+o(1))n$.} 
For $w\geq 1$, by embedding the torus into the integer lattice and using the translation invariance of $d_n$, we get
\begin{equation}\label{riemann}
	\Cexp{D_i }{ W_i = w} = (\beta - 1)\int_1^\infty \sum_{z \in \Q_n \cap \mathbb{Z}^d\setminus\{0\}} \varphi\left(\frac{\|z\|^d}{\kappa(w,t)}\right) t^{-\beta} dt.
\end{equation}
When conditioning on $N$ in the Poisson case, the positions of the vertices are sampled uniformly and independently on the torus. Thus, by embedding the torus into $\Q_n$ and using translation invariance of $d_n$,
\[
	\CexpN{D_i}{W_i = w} = \frac{N-1}{n}(\beta - 1)\int_1^\infty \int_{\Q_n} \varphi\left(\frac{\|z\|^d}{\kappa(w,t)}\right) dz \, t^{-\beta} dt.
\]
Note that for the lattice case, {the sum in \eqref{riemann}} can be interpreted as a $d$-dimensional Riemann sum. Hence, there exists constants ${\sf q}, \sQ > 0$ such that
\[
	{\sf q} \int_{\Q_n} \varphi\left(\frac{\|z\|^d}{\kappa(w,t)}\right) dz
	\le \sum_{z \in \Q_n \cap \Z^d} \varphi\left(\frac{\|z\|^d}{\kappa(w,t)}\right)
	\le {\sf Q} \int_{\Q_n} \varphi\left(\frac{\|z\|^d}{\kappa(w,t)}\right) dz.
\]
Lemma~\ref{lem:bounds_pi_conditional} for both lattice and Poisson cases follows from Lemma~\ref{lem:calc}. 
\end{proof}

\section{Concentration of the long edges: Proof of Proposition \ref{prop:order_bound_elong}}\label{sec:proof_concentration_long_edges}

{Recall that we denote the number of vertices in our graph by $N$. 
%In the lattice case, $N = | \Z^d\cap \T_n^d| = (1+o(1))n$ and in the Poisson case, $N$ is a $\mathrm{Poi}(n)$-distributed random variable. 
Recall $I_M(n)$ from \eqref{def:interval_n_sqrt} for $M>0$.}
     For both lattice and Poisson models we remark that, for all $x,y \in V_n$ with $x\neq y$, conditionally on $W_x \vee W_y \leq n^a$ and $d_n(x,y)> \eps_n n^{1/d}$, 
    \begin{align*}
        p_{x,y}(W_x, W_y) \leq {\sf C} \left(\frac{\kappa(W_x,W_y)}{\eps_n^d n}\right)^{\alpha} \leq {\sf C}^{1+\alpha} \left(\frac{n^{\xi - 1}}{\eps_n^d}\right)^{\alpha},
    \end{align*}
    where $\xi:=a (2\vee (\beta-1))$, using \eqref{asspt:kernel_bounds_var} and \eqref{asspt:profile_bounds}. {Recall that we suppose that $\eps_n$ and $a$ satisfy the conditions of Proposition~\ref{prop:order_bound_emain}. More precisely, $\eps_n = n^{-\gamma}$ for some $0 < \gamma < (1-\tfrac{1}{\alpha})\wedge \tfrac{1}{d}$ and $\alpha(1-\xi-\gamma)>1$.} Define, for $n\geq 1$,
    \begin{align*}
        p_n:=(n^{\xi-1}/\eps_n^d)^{\alpha}=\delta^{-\alpha}n^{\alpha(\xi-1 + \gamma)},
    \end{align*}
    where $\delta>0$. {We begin by proving the proposition for lattice models.} Let $B_n$ be a $\mathrm{Bin}(N^2,p_n)$-distributed random variable with mean $\delta^{-\alpha}n^{2+\alpha(\xi-1 + \gamma)}$ since $N= n$. Then $|E_{\rm long}|$ is stochastically dominated by $B_n$. Observe that since $\alpha(1-\xi-\gamma)>1$, there exists some $\delta_1 > 0$ such that 
    $$n^2 p_n / n = (\delta^{-\alpha}n^{2+\alpha(\xi-1 + \gamma)})/n \le \delta^{-\alpha} n^{-\delta_1},$$ which goes to zero polynomially fast. Consequently, $\p{|E_{\rm long}| \geq n\delta} \leq \p{B_n \geq n\delta}$ and the upper bound tends to zero stretched-exponentially fast again using a standard Chernoff bound \cite[Theorem 2.1]{JanLucRuc00} for binomial random variables, and hence faster than $n^{-k(\beta -2)}$ for any $k\in \mathbb{N}$. This concludes the proof of Proposition~\ref{prop:order_bound_elong} for lattice models.\smallskip
    
    For Poisson models, conditionally on the number of vertices {$N$}, let $B_n$ be a random variable, which is $\mathrm{Bin}({N^2}, p_n)$-distributed and note again that  $|E_{\rm long}|$ is stochastically dominated by $B_n$. Further
    \begin{align}\label{eq:elong_decomp}
          \p{B_n>\delta n} & \leq \p{B_n>{N^2} p_n + n\delta/2} + \p{{N^2} p_n  \geq n\delta/2}.
    \end{align}
    The conditional mean of $B_n$ is ${N^2} p_n$, thus standard large deviation bounds for binomial random variables (see, e.g., \cite[Theorem~2.1]{JanLucRuc00}) give
    \begin{align*}
        \pN{B_n> {N^2}p_n + n\delta/2} \leq \exp\left({-\frac{\delta^2 n^2}{8({N^2}p_n + \delta n/6)}}\right).
    \end{align*}
    By taking expectations we can bound the first summand of the right-hand side of~\eqref{eq:elong_decomp} from above by
    \begin{align*}
        % \E{ \exp\left({-\frac{\delta_0^2 n^2}{2({N^2}p_n + \delta_0 n/3)}}\right)} & \leq 
        \E{ \exp\left({-\frac{\delta^2 n^2}{8({N^2}p_n + \delta n/6)}}\right)\I{N\in I_M(n)}}
        & + \p{N\notin I_M(n)}.
    \end{align*}
    By choosing $M$ sufficiently large, \eqref{eq:conc_bound_nb_ppp} gives us that the second summand above equals $o(n^{-k(\beta -2)})$. The first summand of the above inequality can be bounded from above by 
    {
    \begin{align*}
        \exp{\left(-\frac{\delta^2 n^2}{8{n^2 p_n}+4\delta n/3)}\right)} = \exp{\left(-\frac{\delta^2}{8\delta^{-\alpha}n^{-\alpha(1-\xi-\gamma)} + 4\delta/(3n)}\right)}.
    \end{align*}
    This can be further upper bounded by a constant multiple of $\exp(-n^{\alpha(1-\xi - \gamma)})$ for large~$n$. Since $\alpha(1-\xi - \gamma)>1$, the above equals $o(n^{-k(\beta-2)})$.} We are left to {bound the second summand of \eqref{eq:elong_decomp}.} Note that, for the same $\delta_1>0$ as before,
    \begin{align*}
         \p{{N^2} p_n \geq n\delta/2} & = \p{{N} \geq \sqrt{n \delta/(2p_n)}} \leq \p{{N} \geq n^{1+\delta_1/2} },
    \end{align*}
using the fact that \smash{$\sqrt{n/p_n} = \delta^{\alpha/2} n^{\tfrac{1}{2} + \alpha/2(1-\xi -\gamma)}> n^{1+\delta_1/2}$}. Since $N$ is a Poi$(n)$-distributed random variable, standard concentration bounds \cite[Remark 2.6]{JanLucRuc00} imply $\mathbb P(N \geq n^{1+\delta_1/2})$ is $o(n^{-k(\beta-2)})$, which concludes the proof. \qed

\section{Large deviations of the high-weight edges: Proof of Proposition \ref{prop:order_bound_ehigh}}
\label{sec:proof_high_weight_edges}

\subsection{Overview of the proof of Proposition \ref{prop:order_bound_ehigh}}
In this section, we reduce the proof of Proposition \ref{prop:order_bound_ehigh} to a key technical result, which we state as Proposition \ref{prop:limit_iid_sum}. The proof of Proposition \ref{prop:limit_iid_sum} is deferred to Section~\ref{sec:few_jumps_phen}. 
We begin by noting that 
\begin{align*}
    |E_{\rm high}|& = \sum_{\{x,y\}\subseteq V_n}A_{x,y}\I{W_x \vee W_y \geq n^a}  =  \sum_{x\in V_n}D_x\I{W_x \geq n^a} - \frac{1}{2}\sum_{x\in V_n}\sum_{y\in V_n}A_{x,y}\I{W_x\wedge W_y \geq n^a}\\
 %   & = \sum_{x\in V_n}\I{W_x \geq n^a} \left(D_x - \frac{1}{2}\sum_{y\in V_n}A_{x,y}\I{ W_y \geq n^a}\right)\\
    & =  \sum_{x\in V_n}\I{W_x \geq n^a} \Tilde{D}_x, \numberthis \label{eq:e_high_decomp2}
\end{align*}
where $D_x$ is the degree of $x\in V_n$ defined in \eqref{def:degree}, and 
$$\Tilde{D}_x := D_x - \frac{1}{2}\sum_{y\in V_n}A_{x,y}\I{ W_y \geq n^a}.$$
Recall that for Poisson models, the vertex set $V_n$ is a Poisson point process $\cP$ of intensity one on~$\T_n^d$. In this case, we need to state results for the Palm version %$\cP^x_{n}$ 
obtained by adding the point $x\in\T_n^d$ with independent Pareto-distributed weight $W_x$ to $\cP$.
We denote by $\mathbb{P}^{x}$ the probability and by $\mathbb{E}^{x}$ the expectation with respect to this new vertex set.\smallskip

For $w>1$ and $x\in \T_n^d$ define, for both lattice and Poisson models,
\begin{align}\label{def:lambda_tilde}
    \lambda^x_n(w):=\Cexpx{\Tilde{D}_x}{W_x = w}.
\end{align}
Observe that $\lambda^x_n(w) = \lambda^y_n(w)$ for $x,y\in \T_n^d$, by the translation invariance of the torus $\T_n^d$, we therefore write $\lambda_n(w)$ for $\lambda^x_n(w)$. Further, for $C>0$, define the events
\begin{align}
    \cF_{n,C}:=  {\bigcap_{x\in V_n}
    \big\{W_x\geq n^a, |\Tilde{D}_x-\lambda_n(W_x)|<C\sqrt{\lambda_n(W_x)\log{\lambda_n(W_x)}}\big\}\cup\{W_x<n^a\}.}\label{def:event_concentration_deg_tilde}
\end{align}
We now state a few results that are necessary to prove Proposition~\ref{prop:order_bound_ehigh} and postpone {some of their} proofs to the end of the section:
\begin{lem}[Concentration of $\Tilde{D}_x$]\label{lem:conc_tilde_d}
    For any $\xi>0$, there is $C=C(\xi,{a})>0$ such that 
     \[\p{\cF_{n,C}^{\,c}}=o(n^{-\xi}).\]
    %\cmar{the constant $C$ also depends on $a$, should we write $C(\xi, a)$ instead of $C(\xi)$?}
\end{lem}

\noindent
For $w>1$ and $x\in \T_n^d$, define 
\begin{align}\label{def:pi_n_x}
    \pi^x_n(w) := \Cexpx{D_x}{W_x = w}.
\end{align}
As before, for $x,y \in \T_n^d$, we have $\pi_n^x(w) = \pi_n^y(x)$ and we therefore write $\pi_n(w)$ for $\pi_n^x(w)$. The following lemma bounds $\pi_n(w)$ and $\lambda_n(w)$:
\begin{lem}\label{lem:bounds_pi}
    Let $w>1$. There exist positive constants $\sk, \sK$ such that
    \[\sk w \leq \pi_n(w) \leq \sK w,\]
for sufficiently large $n$.
    Since $\frac{1}{2}D_x \leq \Tilde{D}_x \leq D_x$, the above implies $\frac{1}{2}\sk w \leq \lambda_n(w) \leq \sK w.$
\end{lem}

\noindent
The proof of Lemma~\ref{lem:bounds_pi} follows directly from Lemma~\ref{lem:bounds_pi_conditional} by noting that for lattice models $N=n$, while for Poisson models $\E{N} = n$.
\smallskip

The following proposition states a limit law for the sum $\sum_{x} \lambda_n(W_x)\I{W_x \geq n^a}$.
\begin{prop}\label{prop:limit_iid_sum}
    Let $\rho>0$ be non-integer and $k-1<\rho<k$. {With $F=F_0$ as in (\ref{def:f_rho_function})}
\begin{align}\label{prop:limit_iid_sum_lattice}
                \p{\sum_{x \in V_n}\lambda_n(W_x)\I{W_x \geq n^a}>n\rho}=(F(\rho)+o(1))n^{-k(\beta -2)}.
    \end{align}
\end{prop}
The proof of Proposition \ref{prop:limit_iid_sum} will be given in Section~\ref{sec:few_jumps_phen}. We are now ready to prove Proposition~\ref{prop:order_bound_ehigh} {using Lemmas~\ref{lem:conc_tilde_d} and \ref{lem:continuity}, and Proposition~\ref{prop:limit_iid_sum}.}

\begin{proof}[Proof of Proposition~\ref{prop:order_bound_ehigh}]
 The idea of the proof is to approximate $|E_{\rm high}|$ by the i.i.d.\ sum of Proposition \ref{prop:limit_iid_sum}, and argue that the errors we encounter in this approximation tend to zero faster than the right-hand side of Proposition \ref{prop:limit_iid_sum}. The proof holds for both lattice and Poisson models. To this end, we fix $C>0$ to be sufficiently large such that the probability of the event $\cF_{n,C}^c$ defined in (\ref{def:event_concentration_deg_tilde}), tends to zero faster than the right-hand side of \eqref{prop:limit_iid_sum_lattice}. {The existence of such $C$ is {ensured} by Lemma~\ref{lem:conc_tilde_d}}. Write
\begin{align}
    \p{|E_{\rm high}|>n \rho}&=\p{\sum_{x\in V_n} \Tilde{D}_x \I{W_x\geq n^a}>n\rho}\notag \\&=\p{\sum_{x\in V_n} \Tilde{D}_x \I{W_x\geq n^a}>n\rho,\, \cF_{n,C}}+O(\p{\cF_{n,C}^{\, c}}).  \label{eq:e_high_decomp}
\end{align}
By choice of $C$, the second term on the right-hand side tends to zero faster than $n^{-k(\beta -2)}$. Using the definition of $\cF_{n,C}$, the first summand of (\ref{eq:e_high_decomp}) can be bounded from below by
\begin{align*}
    &\p{\sum_{x\in V_n} \Tilde{D}_x \I{W_x\geq n^a}>n\rho,\,\cF_{n,C}} \\
    & \geq \p{\sum_{x\in V_n}\big(\lambda_n(W_x)-C\sqrt{\lambda_n(W_x)\log{\lambda_n(W_x)}}\big)\I{W_x \geq n^a}>n\rho,\, \cF_{n,C}}\\
    &\geq\p{\sum_{x\in V_n}\big(\lambda_n(W_x)-C\sqrt{\lambda_n(W_x)\log{\lambda_n(W_x)}}\big)\I{W_x \geq n^a}>n\rho}-O(\p{\cF_{n,C}^{\,c}}), \numberthis \label{eq:e_high_int_LB}
\end{align*}
 where as before, the term $O(\mathbb P(\cF_{n,C}^{\,c}))$ tends to zero faster than $n^{-k(\beta -2)}$. An upper bound follows similarly, i.e.,
\begin{align*}
&\p{\sum_{x\in V_n} \Tilde{D}_x \I{W_x\geq n^a}>n\rho,\, \cF_{n,C}}\\
& \leq \p{\sum_{x\in V_n}\left(\lambda_n(W_x)+C\sqrt{\lambda_n(W_x)\log{\lambda_n(W_x)}}\right)\I{W_x \geq n^a}>n\rho,\,\cF_{n,C}}\\
&\leq\p{\sum_{x\in V_n}\left(\lambda_n(W_x)+C\sqrt{\lambda_n(W_x)\log{\lambda_n(W_x)}}\right)\I{W_x \geq n^a}>n\rho}. \numberthis \label{eq:e_high_int_UB}
\end{align*}
%\PvdH{Maybe we can write these bounds more explicitly. Now it not completely clear what the bounds are.} 
%\CK{I added the term we are bounding - hopefully it is more clear now.}
Fix $\eps>0$ such that \smash{$\frac{\rho}{1-\eps},\frac{\rho}{1+\eps} \in (k-1,k)$}. There exists $L>0$ such that  $h(x)=\eps x-C\sqrt{x \log{x}}$ is increasing on the interval $(L,\infty)$. In particular, when $W_x\geq n^a$ for $x\in V_n$, we have $\lambda_n(W_x)\geq \lambda_n(n^a)\geq L$ for all $n$ sufficiently large, where the first inequality holds by monotonicity of $\lambda_n$, and the second inequality follows from the fact that $\lambda_n(n^a)$ diverges by Lemma~\ref{lem:bounds_pi}. Using the monotonicity of the function $h$ on $(L,\infty)$, we can then conclude that when $W_x \geq n^a$ for $x\in V_n$ we have  $h(\lambda_n(W_x))\geq h(\lambda_n(n^a))>0$ for all sufficiently large $n$. This implies that, for sufficiently large $n$ and all $x\in V_n$, 
\begin{align}
    (1+\eps)\lambda_n(W_x) & \geq \lambda_n(W_x) + C\sqrt{\lambda_n(W_x)\log{\lambda_n(W_x)}}, \quad\text{ and }\label{eq:lambda_upper_bound}\\
    (1-\eps)\lambda_n(W_x) & \leq \lambda_n(W_x) - C\sqrt{\lambda_n(W_x)\log{\lambda_n(W_x)}}.\label{eq:lambda_lower_bound}
\end{align}
Combining the above two inequalities together with (\ref{eq:e_high_decomp}), (\ref{eq:e_high_int_LB}) and (\ref{eq:e_high_int_UB}), it follows that for sufficiently large $n$,
\begin{align*}
   \pBig{ \sum_{x\in V_n} & \lambda_n(W_x)  \I{W_x \geq n^a}>\tfrac{\rho}{1-\eps}n } +O(\p{\cF_{n,C}^{\,c}})\\& \leq \p{|E_{\rm high}|>n \rho}\leq \pBig{\sum_{x\in V_n}\lambda_n(W_x)\I{W_x \geq n^a}>\tfrac{\rho}{1+\eps}n}+O(\p{\cF_{n,C}^{\,c}}). \numberthis\label{eq:e_high_ineq_series}
\end{align*}
Proposition \ref{prop:order_bound_ehigh} follows by multiplying this series of inequalities by $n^{k(\beta -2)}$, then letting $n \to \infty$, and finally $\eps \searrow 0$, using \eqref{prop:limit_iid_sum_lattice} and the continuity of $F=F_0$ from Lemma~\ref{lem:continuity}.\qedhere
\end{proof}

We now give the proofs of Lemma~\ref{lem:conc_tilde_d} and Proposition \ref{prop:limit_iid_sum} used in the proof of Proposition~\ref{prop:order_bound_ehigh}. 

\subsection{Concentration of degrees: Proof of Lemma~\ref{lem:conc_tilde_d}}
We begin by introducing a lemma that states that the degree $\Tilde{D}_x$ is a Poisson random variable for the Poisson case:
\begin{lem}\label{lem:dx_poisson}
    In the Poisson case, given a vertex $x\in \cP_n$ and its weight $W_x$, the locations and the weights of the vertices contributing to $\Tilde{D}_x$ are distributed as a Poisson point process on $\T_n^d \times (0,\infty)$ with intensity 
    $$h(y,w) = f_W(w)\frac{1}{2}(1 + \I{w < n^a})\varphi\left(\frac{d_n(x,y)^d}{\kappa(w, W_x)}\right).$$ Consequently, given the weight $W_x$, the degree $\Tilde{D}_x$ is a Poisson random variable with parameter $\int_1^{\infty}\int_{\T_n^d} h(y,w)dydw$. 
\end{lem}
We omit the proof of this lemma, as it is a minor modification of \cite[Lemma 2.1]{scaling_clust_23}.
%\CK{Neeladri to do: add references and state that we omit the proof of this lemma.} 
Next we state and prove a concentration result on $\Tilde{D}_x$ given its weight, for any fixed $x\in V_n$.
\begin{lem}\label{lem:conc_tilde_d_fixed_vertex}
    For $x \in \T_n^d$, any $w>1$ and $C>0$,
    \begin{align}\label{eq:conc_d_fixed_vertices}
           & \Cprobx{|\Tilde{D}_x-\lambda_n(w)|>C\sqrt{\lambda_n(w)\log{\lambda_n(w)}}}{ W_x=w} \\
          & \leq 2 \exp\left(- (\tfrac{C}2)^2\tfrac{\log{\lambda_n(w)}}{4 + \frac{2}{3}\sqrt{\log{\lambda_n(w)}/\lambda_n(w)}}\right)\notag
           % = 2 \lambda_n(w)^{-\frac{C^2}{4}(4 + 2/3\sqrt{\log(\lambda_n(w))/\lambda_n(w)})^{-1}}.\notag
        \end{align}
\end{lem}

\begin{proof}[Proof of Lemma~\ref{lem:conc_tilde_d_fixed_vertex}]
We can write $\Tilde{D}_x=D_x - D'_x$, where $D'_x = \frac{1}{2}\sum_{y\in V_n}A_{x,y}\I{W_y \geq n^a}$ and $D_x$ given by \eqref{def:degree}.  Recall that $\pi_n(w) = \Cexpx{D_x}{W_x=w}$ for $w>0$, and define $\pi'_n(w) = \Cexpx{D'_x}{W_x=w}$. For all $w>1$,
\begin{align}\label{eq:pi_bound_lambda}
    \pi_n(w) \leq 2 \lambda_n(w) & \phantom{XX} \text{ and }  \phantom{XX} \pi'_n(w) \leq \lambda_n(w),
\end{align}
since we can write $\Tilde{D}_x = D_x' + \sum_{y\in V_n}A_{x,y}\I{W_y< n^a}$ and $2\Tilde{D}_x = D_x + \sum_{y\in V_n}A_{x,y}\I{W_y< n^a}$.
 We begin by proving the statement concerning lattice models. Note that, conditionally on $W_x=w$, both $D_x$ and $D_x'$ are sums of independent Bernoulli random variables.  Hence, standard concentration arguments (see e.g.\,{\cite[Theorem~2.8]{JanLucRuc00}}) imply that, for all $t>0$,
\begin{align}
    \Cprobx{|D_x - \pi_n(w)|>t}{W_x = w} & \leq \exp{\left(-\tfrac{t^2}{2\pi_n(w)+2t/3}\right)},\label{eq:dx_conc_result}\\
    \Cprobx{|D'_x - \pi'_n(w)|>t}{W_x = w} & \leq \exp{\left(-\tfrac{t^2}{2\pi'_n(w)+2t/3}\right)}.\label{eq:dx_prime_conc_result}
\end{align}
Now, we write 
\begin{align*}
    \Cprobx{|\Tilde{D}_x- \lambda_n(w)|>t }{W_x = w} & \leq  \Cprobx{|D_x- \pi_n(w)|>t/2 }{W_x = w} \\
    & \phantom{X} +  \Cprobx{|D'_x- \pi'_n(w)|>t/2 }{W_x = w}.
\end{align*}
By (\ref{eq:dx_conc_result}) and (\ref{eq:dx_prime_conc_result}) together with \eqref{eq:pi_bound_lambda}, and letting $t = C\sqrt{\lambda_n(w)\log \lambda_n(w)}$, this in turn is bounded from above by 
\begin{align}\label{eq:conc_tilde_upp_bound}
    2 \exp\left(- (\tfrac{C}2)^2\tfrac{\lambda_n(w)\log{\lambda_n(w)}}{4\lambda_n(w) + \frac{2}{3}C\sqrt{\lambda_n(w)\log{\lambda_n(w)}}}\right) = 2 \exp\left(- (\tfrac{C}2)^2\tfrac{\log{\lambda_n(w)}}{4 + \frac{2}{3}C\sqrt{\log{\lambda_n(w)}/\lambda_n(w)}}\right), 
\end{align}
which concludes the lemma for lattice models. \medskip
\pagebreak[3]

For Poisson models note that Lemma~\ref{lem:dx_poisson} states that $\Tilde{D}_x$ is a Poisson random variable.  Standard Poisson concentration bounds {\cite[Remark~2.6]{JanLucRuc00}} imply that, for $w>1$ and $t>0$,
\begin{align*}
    \Cprobx{|\Tilde{D}_x- \lambda_n(W_x)|> t}{W_x = w} & \leq 2 \exp\left(-\frac{t^2}{2(\lambda_n(w) + t)}\right).
\end{align*}
Plugging in $t= C\sqrt{\lambda_n(w)\log \lambda_n(w)}$ and noting that the upper bound for lattice models given by \eqref{eq:conc_tilde_upp_bound} is larger, we conclude the proof.
 \end{proof}
 
\begin{proof}[Proof of Lemma \ref{lem:conc_tilde_d}]
    We begin by proving the lemma for lattice models. Note that a standard union bound gives us that 
     \begin{align*}
        \p{\cF_{n,C}^{\, c}} & = \p{\exists\, x \in V_n\;\;\text{such that}\;\; |\Tilde{D}_x - \lambda_n(W_x)|\I{W_x \geq n^a} >C \sqrt{\lambda_n(W_x)\log({\lambda}_n(W_x))} }\\
        & \leq n \px{|\Tilde{D}_{x} - {\lambda}_n(W_{x})|\I{W_{x} \geq n^a} > C\sqrt{{\lambda}_n(W_{x})\log({\lambda}_n(W_{x}))}}\\
        & = n (\beta -1) \int_{n^a}^{\infty}w^{-\beta}\Cprobx{|\Tilde{D}_{x} - {\lambda}_n(w)|> C\sqrt{{\lambda}_n(w)\log({\lambda}_n(w))}}{W_{x} = w} dw.
    \end{align*}
    The upper bound \eqref{eq:conc_d_fixed_vertices} from Lemma \ref{lem:conc_tilde_d_fixed_vertex} gives us that
    \begin{align*}
         \p{\cF_{n,C}^{\, c}} & \leq 2(\beta -1) n \int_{n^a}^{\infty} w^{-\beta}\exp{\left(-\tfrac{(C/2)^2\log{{\lambda}_n(w)}}{4+\frac{2}{3}\sqrt{\log{{\lambda}_n(w)}/{\lambda}_n(w)}} \right)} dw.
    \end{align*}
     Remark that  by Lemma~\ref{lem:bounds_pi}, $\lambda_n(w)$ inside the integral above can be bounded from below by $\sk n^a$, which is at least $\e$ for all large $n$. %$\lambda_n(\cdot)$ is monotone increasing, $\log(\lambda_n(n^a))/\lambda_n(n^a)\leq 1$ and $\lambda_n(n^a)>{\rm e}$ for large $n$. 
     Together with the fact that the function $h(x) = (\log x)/x$ is decreasing in $({\rm e},\infty)$ we get that  $h(\lambda_n(w)) \leq h(\sk n^a) \leq 1$ for large $n$ and $w>n^a$. Hence, for large $n$, the probability of $\cF_{n,C}^{\, c}$ can be bounded from above by
    \begin{align*}
        2(\beta -1) n & \int_{n^a}^{\infty}  w^{-\beta}\exp{\left(-\tfrac{(C/2)^2\log{{\lambda}_n(w)}}{5} \right)} dw\\&\leq 2(\beta-1) n \int_{n^a}^{\infty}   w^{-\beta}\exp{\left(-\tfrac{(C/2)^2\log{{\lambda}_n(n^a)}}{5} \right)} dw
        =2n^{1 - a(\beta -1)}{\lambda}_n(n^a)^{-C^2/20}.
    \end{align*} 
    By Lemma \ref{lem:bounds_pi}, we conclude that for any $\xi>0$ we can choose a constant $C(\xi)$ large enough such that the last term tends to zero faster than $n^{-\xi}$.\smallskip
    
    Next we prove the lemma for Poisson models. Similarly to the lattice case, and using a union bound, we observe that 
      \begin{align*}
        \p{\cF_{n,C}^{\, c}} & = \p{\exists\, x \in \cP_n\;\;\text{such that}\;\; |\Tilde{D}_x - \lambda_n(W_x)|\I{W_x \geq n^a} >C \sqrt{\lambda_n(W_x)\log({\lambda}_n(W_x))} }\\
   %     & = \E{\Cprob{\exists\, x \in \cP_n\;\;\text{such that}\;\; |\Tilde{D}_x - \lambda_n(W_x)|\I{W_x \geq n^a} >C \sqrt{\lambda_n(W_x)\log({\lambda}_n(W_x)) }}{\cP_n}}.
 %   \end{align*}
 %   Using a union bound we can upper bound the above by 
  %  \begin{align*}
         & \leq \E{\, \sum_{x\in \cP_n}\Cprob{|\Tilde{D}_x - \lambda_n(W_x)|\I{W_x \geq n^a} >C \sqrt{\lambda_n(W_x)\log({\lambda}_n(W_x))}}{\cP_n^x}}\\
        & = \int_{\T_n^d}\E{\Cprob{|\Tilde{D}_x - \lambda_n(W_x)|\I{W_x \geq n^a} >C \sqrt{\lambda_n(W_x)\log({\lambda}_n(W_x))}}{\cP_n^x}}dx\\
        & = \int_{\T_n^d}\px{|\Tilde{D}_x - \lambda_n(W_x)|\I{W_x \geq n^a} >C \sqrt{\lambda_n(W_x)\log({\lambda}_n(W_x))}}dx,
    \end{align*}
    where the second equality holds by Mecke's formula \cite[Theorem~4.1]{Last_Penrose_LPP}. The remainder of the proof is identical to that of lattice models and is therefore omitted. \qedhere

    %\CK{add extra text and $D_x(\cP_n^x)$ and explain the difference between $D_x(\cP_n)$ and $D_x(\cP_n^x)$}
\end{proof}

% The proof of Proposition \ref{prop:limit_iid_sum} can be found in Section \ref{ssec:proof_prop_limit_iid_sum}. Let us state one more useful approximation before proving Proposition~\ref{prop:order_bound_ehigh}.
\subsection{The few-jumps phenomenon: Proof of Proposition \ref{prop:limit_iid_sum}}\label{sec:few_jumps_phen}

It remains to prove Proposition~\ref{prop:limit_iid_sum}. We first state and prove a 
result that shows that it is unlikely to have many small macroscopic contributions.
\begin{lem}\label{lem:magic_lemma}
    Let $k$ be a positive integer. For all $\delta>0$, there exists $\eps>0$ such that
    \begin{align}\label{eq:magic_lemma}
        \p{\sum_{x\in V_n} \lambda_n(W_x)\I{n^a\leq W_x \leq \eps n}>\delta n}=o(n^{-k(\beta-2)}).
    \end{align}
\end{lem}
Lemma~\ref{lem:magic_lemma} tells us that the sum $\sum_{x\in V_n} \lambda_n(W_x)\I{n^a\leq W_x \leq \eps n}$ cannot contribute significantly at a linear scale. Let us discuss the proof strategy. First, we use Lemma~\ref{lem:bounds_pi} to reduce the question to understanding the probability 
\begin{align*}
\p{\sum_{i=1}^{N}  W_i\I{n^a\leq W_i \leq \eps n}>K \delta n},  \end{align*}
for some fixed constant $K>0$. The sum inside the last probability cannot contribute at least $K \delta n$ with good probability when we choose $\eps$ arbitrarily small. To show this, we use a Chernoff inequality to reduce the problem to understanding exponential moments of the form $\mathbb E [\exp{(s_n \hat{W}^{\sss (n)})}]$, where $\hat{W}^{\sss (n)}$ is the truncated random variable $W \I{n^a<W<\eps n}$ and where $W$ has density \eqref{eq:Pareto_density}, and we take $s_n=b\log{n}/n$ for a suitable $b>0$. This we achieve by a second order Taylor expansion and choosing the scale $b$ appropriately. %Here we remark that if the random variables $W_i$ were i.i.d.\ copies of a random variable $W$ with finite exponential moments in a  neighbourhood of $0$, then Lemma \ref{lem:magic_lemma} is very easy to show. Indeed, since the average $X_n=(\sum_{i=1}^{N}W_i)/N$ satisfies a large deviation principle, the probability on the right-hand side of \eqref{eq:magic_lemma} decays exponentially fast in $n$. It is because we are dealing with random variables with power-law tails having density \eqref{eq:Pareto_density}, that we need to separately prove Lemma~\ref{lem:magic_lemma}. Let us go into the proof now.

%\RvdH{I would drop the discussion on $W$'s with thin tails. Shall I make a proposal to trim the above text a little?}
%\PM{Yes please, I also think this is a bit too technical for a strategic discussion.}

\begin{proof}
Let $(W_i)_{i\geq 1}$ be an i.i.d.\ collection with law \eqref{eq:Pareto_density}. For both lattice and Poisson models, we write
\begin{align*}
    \pbigg{\sum_{x\in V_n} \lambda_n(W_x)\I{n^a\leq W_x \leq \eps n}>\delta n} & \leq \pbigg{\sum_{i=1}^{N}  W_i\I{n^a\leq W_i \leq \eps n}>\delta n/\sK},
\end{align*}
where the inequality holds by Lemma~\ref{lem:bounds_pi}, and we use that the $W_x$ are independent of the location of $x$. {Recall that for lattice models, $N$ equals the number $n$ of points on the integer lattice of  $\T^d_n$ and for Poisson models, $N$ is a Poi($n$)-distributed random variable.} Recalling $I_M(n)$ from \eqref{def:interval_n_sqrt} and using~\eqref{eq:conc_bound_nb_ppp} we deduce that, for Poisson models,
\begin{align*}
    & \p{\sum_{i=1}^{N}   W_i\I{n^a\leq W_i \leq \eps n}>\delta n/\sK}\\
    & = \p{\sum_{i=1}^{N} W_i\I{n^a\leq W_i \leq \eps n}>\delta n/\sK, \, N \in I_M(n)} + o(n^{-k(\beta-2)})\phantom{wefillthisspacetoalign}
\end{align*}
\begin{align*}    
    & = \E{ \I{N \in I_M(n)}\Cprob{\sum_{i=1}^{N} W_i\I{n^a\leq W_i \leq \eps n}>\delta n/\sK}{N}}  + o(n^{-k(\beta-2)}).
\end{align*}
From this it is clear that to prove the lemma for both lattice and Poisson models, it suffices to show that, for all $\delta>0$,
\begin{align}\label{eq:weight_magic}
    \p{\sum_{i=1}^{m_n}  W_i\I{n^a\leq W_i \leq \eps n}>\delta  n}=o(n^{-k(\beta-2)}),
\end{align}
where $(m_n)_n$ is any deterministic sequence satisfying $m_n = (1+o(1))n$. To prove \eqref{eq:weight_magic}, define the truncated random variable $\hat{W}^{\sss (n)} := W\I{n^a\leq W< \eps n}$. The Chernoff bound gives us that, for all $s_n>0$,
\begin{align*}
     \p{\sum_{i=1}^{m_n}\hat{W}^{\sss (n)}_i > m} \leq {\rm e}^{-s_nm_n}\E{\exp(s_n\hat{W}^{\sss (n)})}^{m_n}.
\end{align*}
 Let $s_n := b\log n/n$, for a constant $b>0$ to be chosen later. By the Taylor expansion of~${\rm e}^x$,  
    \begin{align}\label{eq:taylor_expansion}
        \E{\exp(s_n\hat{W}^{\sss (n)})} & \leq 1+ s_n\hat{\mu}_n + \E{\sum_{k\geq 2} \frac{(s_n\hat{W}^{\sss (n)})^k}{k!}},
    \end{align}
    where $\hat{\mu}_n := \mathbb E[\hat{W}^{\sss (n)}]$ converges to zero as $n$ tends to infinity. We start by investigating the third term in \eqref{eq:taylor_expansion}. We choose constants $c,C>0$ such that ${\rm e}^x \leq 1 + x + Cx^2$ for all $x\in (0,c)$. Thus, for sufficiently large~$n$,
    \begin{align*}
        \E{\sum_{k\geq 2} \frac{(s_n\hat{W}^{\sss (n)})^k}{k!}} 
        % & = \E{\exp{(s_n\hat{W}^{\sss (n)})}} - 1 - s_n\E{\hat{W}^{\sss (n)}}\\
        & \leq Cs_n^2 \E{(\hat{W}^{\sss (n)})^2\I{\hat{W}^{\sss (n)} < c/s_n}}  + \E{\exp{(s_n\hat{W}^{\sss (n)})}\I{\hat{W}^{\sss (n)} \geq c/s_n}}\\
        & \leq  Cs_n^2 \E{W^2\I{W<c/s_n}} + \E{\exp{(s_nW)} \I{W \in [c/s_n, \eps n)}}\numberthis \label{eq:trunc_sum_bound},
    \end{align*}
    where the second inequality holds since $$(\hat{W}^{\sss (n)})^2\I{\hat{W}^{\sss (n)} < c/s_n} = W^2\I{n^a < W < \eps n}\I{W< c/s_n}$$ and \smash{$\I{\hat{W}^{\sss (n)} > c/s_n} = \I{n^a < W < \eps n}\I{W> c/s_n}$}. Note that $W$ is $L^{2-\ell}$-bounded for $3-\beta <\ell <2$. Thus we can choose $\ell>0$ such that $1<2-\ell< \beta -1$ and
    \begin{align*}
        s_n^2 \E{W^2\I{W<c/s_n}} & = s_n^2 \E{W^{\ell}W^{2-\ell}\I{W<c/s_n}}
        \leq c^{\ell} s_n^{2-\ell} \E{W^{2-\ell}} = o(s_n), 
    \end{align*}
    where the last step holds since $2-\ell>1$ and $s_n$ converges to $0$ as $n$ increases, so $s_n^{2-\ell} = o(s_n)$. Next, note that 
    \begin{align*}
        \E{\exp{(s_nW}) \I{W \in [c/s_n, \eps n)}} & \leq \e^{s_n\eps n}\p{W \in [c/s_n, \eps n)} \leq \e^{s_n\eps n}\p{W \geq c/s_n} \\
        & = \e^{s_n\eps n} (s_n/c)^{\beta -1} = s_n c^{1-\beta}\e^{s_n\eps n} s_n^{\beta - 2}.
    \end{align*}
    Plugging in $s_n = b\log n/n$, we get that, for some constant $C'$, 
    \begin{align*}
         \E{\exp{(s_nW^{\sss (n)}}) \I{W \in [c/s_n, \eps n)}} & \leq s_n C'n^{b\eps + 2- \beta}(\log n)^{\beta -2}.
    \end{align*}
    For $b < (\beta -2)\eps^{-1}$, we get  
    $n^{b\eps + 2 - \beta}(\log n)^{\beta -2} = o(1)$ and so, by \eqref{eq:trunc_sum_bound},
    \begin{align}\label{eq:exp_sum_order}
        \E{\sum_{k\geq 2} \frac{(s_n\hat{W}^{\sss (n)})^k}{k!}} & = o(s_n),
    \end{align}
    as $n\rightarrow \infty$. From \eqref{eq:exp_sum_order} it follows that 
    \begin{align*}
        \E{\exp{(s_n\hat{W}^{\sss (n)})}} & \leq 1 + s_n\hat{\mu}_n + o(s_n) = 1 + o(s_n)\leq \exp{(o(s_n))}.
    \end{align*}
    Putting everything together, it follows that 
    \begin{align*}
         \p{\sum_{i\leq m_n}W_i\I{n^a\leq W_i < \eps n} > \delta n} & \leq \exp{(-s_n\delta n + m_n o(s_n))} = \exp(s_n n(-\delta +o(1)))\\
         & = o(n^{1 - b\delta}),
    \end{align*}
    using that $s_n = b\log n/n$, $m_n =(1+o(1))n$ and $n^{o(1)} = o(n)$. Choosing $b>0$ and $\eps>0$ such that $(1-k(\beta -2))/\delta < b < (\beta - 2)\eps^{-1}$ is satisfied, we obtain that $1 - b\delta < k(\beta -2)$ which concludes the proof of \eqref{eq:weight_magic}.
\end{proof}
\noindent Next, we investigate the asymptotic behaviour of $\lambda_n(bn)$.
\begin{lem}\label{lem:asymp_behavior_tilde_lambda}
    As $n\to \infty$, for $w\in (0,\infty)$,
    \begin{align*}
        \frac{1}{n}{\lambda}_n(wn)\longrightarrow {\Lambda}(w),
    \end{align*}
    where we recall that ${\Lambda}(w) = \int_{[-\frac{1}{2}, \frac{1}{2}]^d}\E{\varphi\left(\frac{\|z\|^d}{\cK(w,W)}\right)}dz$ in \eqref{def:lambda}.
\end{lem}
\begin{proof}
    We prove the lemma first for lattice models. {Recall that $\Q_n \subset \R^d$ is the cube of volume~$n$ centred at $0$. } Embedding the torus around $x$ into the integer lattice and recalling the definition of $\Tilde{D}_x$, we write 
    \begin{align}\label{eq:rewrite_scaled_lambda}
         \frac{{\lambda}_n^x(wn)}{n} & = \frac{1}{n}\int_1^{\infty} f_W(t) \sum_{y\in ( \Q_n \cap \Z^d)\setminus\{0\}} \varphi\left(\frac{\|y\|^d}{\kappa(wn,t)}\right)dt\notag\\
         & \phantom{X} - \frac{1}{2n}\int_{n^a}^{\infty} f_W(t) \sum_{y\in (\Q_n \cap \Z^d)\setminus\{0\}} \varphi\left(\frac{\|y\|^d}{\kappa(wn,t)}\right)dt.
    \end{align}
    By \eqref{asspt:limiting_kernel}, for every $w \in (0,\infty)$,
    \begin{align*}
        \frac{1}{n}  \sum_{y\in (\Q_n \cap \Z^d)
    \setminus\{0\}} \varphi\left(\frac{\|y\|^d}{\kappa(wn,t)}\right) & = \frac{1}{n}  \sum_{y\in (\Q_n \cap \Z^d)\setminus\{0\}} \varphi\left(\frac{\|y\|^d}{(1+o(1))n\, \cK(w,t)}\right),
    \end{align*}
    where the $o(1)$ is uniform in $y\in\Z^d$. This term therefore
    converges to $$\int_{[-\frac{1}{2}, \frac{1}{2}]^d} \varphi\left(\frac{\|z\|^d}{\cK(w,t)}\right) dz.$$ 
    The fact that $|\Q_n \cap \Z^d| = n$ 
together with the dominated convergence theorem 
    implies that, for every $w \in (0,\infty)$,
    \begin{align*}
         \frac{1}{n}\int_1^{\infty} f_W(t)\sum_{y\in (\Q_n \cap \Z^d) \setminus \{0\}} \varphi\left(\frac{\|y\|^d}{\kappa(wn,t)}\right)dt \to \int_1^{\infty}\int_{[-\tfrac{1}{2},\tfrac{1}{2}]^d}f_W(t)\varphi\left(\tfrac{\|z\|^d}{\cK(w,t)}\right)\, dz dt,
    \end{align*}
    as $n \to \infty$. From this it also follows that the second summand in \eqref{eq:rewrite_scaled_lambda} converges to zero, concluding the proof for lattice models. The proof for Poisson models follows similarly and is therefore omitted.
\end{proof}

In the next lemma, we investigate the contribution of the macroscopic weights.

\begin{lem}\label{lem:upper_tail_for_linear_weights}
  For all non-integers $\rho>0$, with $k =  \lceil \rho \rceil$ and for all $b>0$,
\begin{align*}
    &\p{\sum_{x\in V_n} {\lambda}_n(W_x)\I{W_x>bn}>\rho n}=(F_b(\rho)+o(1))n^{-k(\beta - 2)}.
\end{align*}
\end{lem}

\begin{proof}
For Poisson models, we observe that 
\begin{align}\label{eq:ineq_linear_weights}
    & \p{\sum_{x\in V_n}\lambda_n(W_x)\I{W_x > bn}> \rho n,\, N\in I_M(n)}\notag \\
    &\leq \p{\sum_{x\in V_n}\lambda_n(W_x)\I{W_x > bn}> \rho n}\notag\\
    &\leq \p{\sum_{x\in V_n}\lambda_n(W_x)\I{W_x > bn}> \rho n,\, N \in I_M(n)} + o(n^{-k(\beta-2)}),
\end{align}
recalling $I_M(n)$ from \eqref{def:interval_n_sqrt} and using the bound \eqref{eq:conc_bound_nb_ppp}. To prove the lemma for Poisson models it therefore suffices to show that for $b>0$ and sufficiently large $M>0$, 
\begin{align}\label{eq:cond_upper_tail}
    & \E{\I{N\in I_M(n)}\Cprob{\sum_{x\in \cP_n}\lambda_n(W_x)\I{W_x>bn}>\rho n}{N}} \notag = (1+o(1))n^{-k(\beta - 2)}F_b(\rho).
\end{align}
Note that, for $x\in V_n$, $\lambda_n(W_x)$ is independent of the location of $x$ for lattice and Poisson models. Therefore in order to prove Lemma~\ref{lem:upper_tail_for_linear_weights} for both models, it suffices to show that
\begin{align}
    \p{\sum_{i=1}^{m_n}\lambda_n(W_i)\I{W_i>bn}>\rho n} = (1+o(1))n^{-k(\beta-{2})}F_b(\rho),
\end{align}
for an i.i.d.\ collection $(W_i)_{i=1}^{m_n}$ having law \eqref{eq:Pareto_density}, where $m_n$ is any deterministic sequence satisfying $m_n=(1+o(1))n$. Thus for the rest of the proof, we work with such a sequence.\smallskip

   Let $b>0$ and define the events $\cA_m = \cA_m(b)$ for $1\leq m\leq m_n$ as
    \begin{align}\label{def:event_amb}
        \cA_m(b) :=\{\text{there exist exactly}\;m\;\text{random variables with}\;W_i \geq bn\}.
    \end{align}
    We can write
    \[\p{\sum_{i = 1}^{m_n} {\lambda}_n(W_i)\I{W_i>bn}>\rho n}=\sum_{m=1}^{m_n}\p{\sum_{i = 1}^{m_n}{\lambda}_n(W_i)\I{W_i>bn}>\rho n, \, \cA_m}.\]
    Let  $m<k$. Then, using the fact that $(W_i)_i$ are i.i.d.,\,we remark that
    \begin{align*}
    &\p{\sum_{i = 1}^{m_n}{\lambda}_n(W_i)\I{W_i>bn}>\rho n, \cA_m}\\&=\binom{m_n}{m}\p{\sum_{j=1}^{m} {\lambda}_n(W_{i_j})>\rho n, W_{i_1},\ldots,W_{i_m}>bn, W_{i_{m+1}},\ldots,W_{i_{m_n}}\leq bn}=0,        
    \end{align*}
for sufficiently large $n$, because ${\lambda}_n(W)\leq N$ deterministically and so $\sum_{j=1}^m {\lambda}_n(W_{i_j})\leq m N \leq (1+o(1))(k-1)n < \rho n$. {Next, we note that, by a union bound of the {at least} $k+1$ vertices with weight larger than $bn$,   
\begin{align*}
    \sum_{m=k+1}^{m_n}\p{\sum_{i = 1}^{m_n}{\lambda}_n(W_i)\I{W_i>bn}>\rho n, \, \cA_m} & \leq \sum_{m=k+1}^{m_n}\p{\cA_m} \\
    & \leq (1 + o(1))n^{k+1}\p{ W_{1},\ldots,W_{{k+1}}>bn}.
\end{align*}
Since  $(W_i)_i$ have law \eqref{eq:Pareto_density}, this equals $ o(n^{k(2-\beta)})$.} Therefore,
\begin{align}
\label{eq:linear_weights_sum_ak}
    \mathbb P\Big(\sum_{i = 1}^{m_n} {\lambda}_n(W_i)\I{W_i>bn}>\rho n\Big)  \leq  o(n^{k(2-\beta)}) +  \mathbb P\Big(\sum_{i = 1}^{m_n}{\lambda}_n(W_i)\I{W_i>bn}>\rho n,\cA_{k}\Big).
\end{align}
It remains to study
\begin{align*}
    &\mathbb P\Big(\sum_{i = 1}^{m_n}{\lambda}_n(W_i)\I{W_i>bn}>\rho n,\cA_{k}\Big)\\&=\binom{m_n}{k}\mathbb P\Big(\sum_{i = 1}^{m_n}{\lambda}_n(W_i)\I{W_i>bn}>\rho n,W_{1},\dots,W_{k}>bn,W_{{k+1}},\dots,W_{{m_n}}\leq bn\Big)\\&=\binom{m_n}{k}(1-(\beta-1)(bn)^{1-\beta})^{m_n-k}(\beta-1)^{k}\int_{bn}^{\infty}\cdots\int_{bn}^{\infty}\I{\frac{1}{n}({\lambda}_n(t_i)+\ldots+{\lambda}_n(t_{k}))>\rho}\prod_{i=1}^{k}t_i^{-\beta}dt_i.
\end{align*}
Observe that {$(1-(\beta-1)(bn)^{1-\beta})^{m_n-k}=1+o(1)$ since $n^{-(\beta-1)}=o(n^{-1})$,} using $(m_n-k)/n \to 1$. By a change of variables $y_i=t_i/n$, we can write the last expression~as
\begin{align*}
    (1+o(1))\binom{n}{k}n^{-k(\beta-1)}(\beta-1)^{k}\int_{b}^{\infty}\cdots\int_{b}^{\infty}\I{\frac{1}{n}({\lambda}_n(ny_1)+\cdots+{\lambda}_n(ny_{k}))>\rho}\prod_{i=1}^{k}y_i^{-\beta}dy_i,
\end{align*}
since $m_n = (1+o(1))n$. By Lemma \ref{lem:asymp_behavior_tilde_lambda} and dominated convergence, the last expression {is asymptotically equivalent to $n^{-k(\beta - 2)} F_b(\rho)$.}
\end{proof}

\noindent Now we are ready to prove Proposition \ref{prop:limit_iid_sum}.
\begin{proof}[Proof of Proposition \ref{prop:limit_iid_sum}]
The following proof holds for both lattice and Poisson models. Recall the definition of $F_{b}(\rho)$ given in Lemma~\ref{lem:upper_tail_for_linear_weights} and recall that we write $F(\rho)$ for $F_0(\rho)$. Let $k= \lceil \rho \rceil$. Fix $0<\delta<1$ and $\eps > 0$ such that $(\rho-\delta,\rho+\delta)\subset (k-1,k)$, and such that
\begin{align*}
\mathbb P\Big(\sum_{x\in V_n} {\lambda}_n(W_x)\I{n^a\leq W_x \leq \eps n}>\delta n\Big)=o(n^{-k(\beta -2)}). \numberthis \label{eq:using_magic_lem}    
\end{align*}
By Lemma \ref{lem:magic_lemma} such choices are possible. A lower bound then easily follows using Lemma~\ref{lem:upper_tail_for_linear_weights}, {by observing}
\begin{align*}
    \mathbb P\Big(\sum_{x\in V_n} {\lambda}_n(W_x)\I{W_x\geq n^a}>\rho n\Big)&\geq \mathbb P\Big(\sum_{x\in V_n} {\lambda}_n(W_x)\I{W_x\geq \eps n}>\rho n\Big)\numberthis\label{eq:uppertail_linear_weights_lwbd}\\&=(F_{\eps}(\rho)+o(1))n^{-k(\beta-2)},
\end{align*}
since $n^a = o(n)$. Thus, 
\begin{align*}
    n^{k(\beta -2)} \, \p{\sum_{x\in V_n} {\lambda}_n(W_x)\I{W_x\geq n^a}>\rho n} \geq F_{\eps}(\rho)+o(1).
\end{align*}
The desired lower bound in Proposition~\ref{prop:limit_iid_sum} follows by letting $\eps\searrow 0$ and using monotone convergence. To establish an upper bound, we write 
 \begin{align*}
     &\mathbb P\Big(\sum_{x\in V_n} 
 {\lambda}_n(W_x)\I{W_x\geq n^a}>\rho n\Big)\\&\leq \mathbb P\Big(\sum_{x\in V_n}  {\lambda}_n(W_x)\I{n^a\leq W_x \leq \eps n}>\delta n\Big)+\mathbb P\Big(\sum_{x\in V_n}  {\lambda}_n(W_x)\I{W_x \geq \eps n}>(\rho-\delta) n\Big)\numberthis\label{eq:uppertail_linear_weights_uppbd}\\&=o(n^{-k(\beta -2)})+(F_{\eps}(\rho-\delta)+o(1))n^{-k(\beta -2)}
\leq (F(\rho-\delta)+o(1))n^{-k(\beta -2)},
 \end{align*}
 where the equality follows from \eqref{eq:using_magic_lem} and Lemma \ref{lem:upper_tail_for_linear_weights} and the last inequality holds since $F_{\eps}(\rho - \delta) \leq F(\rho -\delta)$. Therefore,
 \begin{align*}
     n^{k(\beta-2)} \mathbb P\Big(\sum_{x\in V_n}{\lambda}_n(W_x)\I{W_x\geq n^a}>\rho n\Big) & \leq F(\rho-\delta)+o(1).
 \end{align*}
 The desired upper bound follows by first letting $n \to \infty$, and then letting $\delta\searrow 0$ and using the continuity of $F$ on the interval $(k-1,k)$.\qedhere
\end{proof}
\pagebreak[3]

\section{Bulk in edge-length distribution: Proof of 
Proposition~\ref{prop:lln_local_functionals}}
\label{sec:bulk-shape}
\begin{proof}[Proof of 
Proposition~\ref{prop:lln_local_functionals}]
We need to show that
\[
	\mathbb P\bigg(\Big|\frac{1}{n} \sum_{\{x,y\} \in E_n} f(d_n(x,y)) - \frac{1}{2}\E{\lambda_f(W)}\Big| > \delta, |E_n| > n(\mu + \rho)\bigg) = o(n^{-k(\beta-2)}).
\]

Let $R > 0$ be such that $\text{supp}(f) \subseteq (0,R)$. Then
\begin{align*}
	\frac{1}{n} \sum_{\{x,y\} \in E_n} f(d_n(x,y)) 
	&= \frac{1}{n} \sum_{\{x,y\} \in E_n} f(d_n(x,y)) \I{d_n(x,y) \le R}\I{W_x \vee W_y \le n^a}\\
	&\hspace{10pt}+ \frac{1}{n} \sum_{\{x,y\} \in E_n} f(d_n(x,y)) \I{d_n(x,y) \le R}\I{W_x \vee W_y > n^a}.
\end{align*}
Since $R < \eps_n n^{1/d}$ for $n$ sufficiently large, Proposition~\ref{prop:general_concentration_edge_function} implies that
\begin{align*}
	\mathbb P\bigg(\Big|\frac{1}{n} \sum_{\{x,y\} \in E_n} f(d_n(x,y)) \I{d_n(x,y) \le R}\I{W_x \vee W_y \le n^a} - \frac{1}{2}\E{\lambda_f(W)}\Big| > \delta\bigg) = o(n^{-k(\beta-2)}).
\end{align*}
Therefore it suffices to show that
\[
	\mathbb P\bigg(\frac{1}{n} \sum_{\{x,y\} \in E_n} f(d_n(x,y)) \I{d_n(x,y) \le R}\I{W_x \vee W_y > n^a} > \delta\bigg) = o(n^{-k(\beta-2)}).
\]
We bound $|f(d_n(x,y))|\leq {\sf F}$, and
\[
	\frac{1}{n} \sum_{\{x,y\} \in E_n} \I{d_n(x,y) \le R}\I{W_x \vee W_y > n^a}
	\le \frac{2}{n} \sum_{x,y \in V_n}  \I{d_n(x,y) \le R}\I{W_x> n^a},
\]
so that we are left to show that
\[
	\mathbb P\bigg(\sum_{x,y \in V_n}  \I{d_n(x,y) \le R}\I{W_x >n^a} > \frac{\delta}{2 {\sf F}} n\bigg) = o(n^{-k(\beta-2)}).
\]
We first condition on the number of nodes $N$ and then consider the event $\cE_{\eps_n, M}$, which indicates that the number of points in each each cube of length $\eps_n n^{1/d}$ is inside the interval $[\eps_n^dn-M\sqrt{\eps_n^dn \log (\eps_n^dn)}, \eps_n^dn+M\sqrt{\eps_n^dn \log (\eps_n^dn)}]$. Recall that for $M$ sufficiently large, $\p{\cE_{\eps_n, M}} = 1 - o(n^{-k(\beta -2)})$ (see~\eqref{eq:chernoff_bound_cubes}). Hence, it is enough to study the sum under the event $\cE_{\eps_n, M}$. Recall that on this event, $$\sum_{j = 1}^N \I{d_n(X_i,X_j) < R} < 2^{d+1} R^d n^{-1}N\; \text{  for each $1 \le i \le N$, }$$
see~\eqref{eq:number_points_ball_on_good_event}. Hence, we can bound
\[
	\sum_{i,j = 1}^N  \I{d_n(X_i,X_j) \le R}\I{W_i > n^a}
	\le 2^{d+1} R^d n^{-1}N \sum_{i = 1}^N \I{W_i > n^a}.
\]	
We now consider
\[
	\mathbb P_N\Big(\sum_{i = 1}^N \I{W_i > n^a} > \frac{\delta}{2^{2+d} R^d {\sf F}} n^2 N^{-1} \Big).
\]
Conditioned on $N$, the sum is just a $\mathrm{Bin}(N,\p{W > n^a})$ random variable. Now we condition on the event that $N \in I_M(n)$, which happens with probability $1 - o(n^{-k(\beta - 2)})$ for $M$ large enough by \eqref{eq:conc_bound_nb_ppp}. Then, on this event, $ \frac{\delta}{2^{2+d} R^d {\sf F}}n^2 N^{-1} > N \p{W > n^a}$ for $n$ sufficiently large. Hence, an application of Hoeffding's inequality yields that (writing $Q := \frac{\delta}{2^{2+d} R^d {\sf F}}$)
\begin{align*}
	\mathbb P_N\Big(\sum_{i = 1}^N \I{W_i > n^a} > \frac{\delta}{2^{2+d} R^d {\sf F}} n^2 N^{-1} \Big)
	&= \pN{\mathrm{Bin}(N,\p{W > n^a}) > Q n^2 N^{-1}}\\
	&\le 2\exp\left(-2Q^2 n^4 N^{-3}\right).
\end{align*}
Since the last term is bounded by $2\exp\left(-2^{-2}Q^2 n\right)$ on the event $N \in I_M(n)$, we conclude that, for $M$ sufficiently large,
\begin{align*}
	&\mathbb P\Big( \sum_{x,y \in V_n} \I{\{x,y\} \in E_n} \I{d_n(x,y) \le R}\I{W_x n^a} > \frac{\delta}{2 {\sf F}} n\Big)\\
	&\le \E{\I{N \in I_M(n)} \pN{\mathrm{Bin}(N,\p{W > n^a}) > Q n^2 N^{-1}}} + o(n^{-k(\beta-2)})\\
	&\le 2\exp\left(-2^{-2}Q^2 n\right) + o(n^{-k(\beta-2)}) = o(n^{-k(\beta-2)}),
\end{align*}
which is what we needed to show.
\end{proof}
%\color{red} We first give the argument for the lattice case. We bound $|f(d_n(x,y))|\leq C,$ and 
%    \eqn{
%    \frac{1}{n} \sum_{\{x,y\} \in E_n} \I{d_n(x,y) \le R}\I{W_x \vee W_y > n^a}
%    \leq 2(2R+1)^d \frac{1}{n}\sum_{x \in V_n} \I{W_x> n^a}.
%    }
%The sum is a Bin$(n,\prob(W>n^a))$ random variable, which is greater than $n\delta/[2C(2R+1)^d]$ with smaller than exponential probability. For the Poisson case, we can use a similar bound, with $(2R+1)^d$ replaced by the maximal number of points in a ball of radius $R,$ which is with very high probability bounded by $\log{n}$.
%\color{black}

\section{Travelling wave in edge-length distribution: Proof of 
Proposition~\ref{prop:conv_distr_scaled_high_weight_edges}}
\label{sec:lenghts}
Our proof strategy is to show that the edges of macroscopic length are those in $E_{\rm high}$ and the form of the condensate emerges from the fact that $E_{\rm high}$ originates from $k$ uniform vertices that all have weight of order $n$. The total mass, that is at least $\rho n$, of these corresponds to edges with length of order $n^{1/d}$.  For the proof, we rely on the proofs of the previous sections. 
\smallskip

To make the intuition outlined above precise, define the function $g: (0,1)^2 \rightarrow [0,1)$ by 
\begin{align}\label{def:vol_cube_ball}
    g(x,y) := \text{Vol}\left([-\tfrac{1}{2},\tfrac{1}{2}]^d \cap \left(\cB(0,y) \setminus \cB(0,x)\right)\right).
\end{align}
Fix $0\leq b_1<b_2\leq \sqrt{d}/2$. Define the vertex set $\cU_n(x) = \cU_n^{(b_1,b_2)}(x)$ as the intersection of $V_n$ and the annulus $\cB(x, b_2n^{1/d})\setminus \cB(x, b_1n^{1/d})$ centred at $x\in V_n$. 
Write $\cU_n$ for $\cU_n^{(b_1,b_2)}({0})$. Remark that for all $x\in V_n$, for the lattice case, we have that $|\cU_n(x)| = (1+o(1))g(b_1,b_2)n$, while for the Poisson case, $|\cU_n(x)|$ is Poisson distributed with mean $g(b_1,b_2)n$. Recall that $(W_{\sss(i)})_{i \in [N]}$ is the order statistic of the weights $(W_x)_{x\in V_n}$.
For $i\in [N]$, recall that $X_{\sss(i)} \in V_n$ denotes the random vertex with the $i^{\text{th}}$ largest weight, i.e., \smash{$W_{X_{\sss(i)}} = W_{\sss(i)}$}. Define $\mathbf{X}_{\sss(i)} := \{X_{\sss(1)}, \ldots, X_{\sss(i)}\}$, to be the set of locations of the first $i$ largest weight vertices.
\smallskip
Remark that to prove 
Proposition~\ref{prop:conv_distr_scaled_high_weight_edges}, it suffices to prove the statement for $f = \Itwo{[b_1,b_2)}$, for any $b_1,b_2$ satisfying $0\leq b_1< b_2\leq \sqrt{d}/2$. For $w\in (0,\infty)$, define $\Lambda_{[b_1,b_2)}(w)$ as 
\begin{align}\label{def:lambda_interval}
    \Lambda_{[b_1,b_2)}(w) := \int_{[-\frac{1}{2}, \frac{1}{2}]^d \cap \left(\cB(0,b_2)\setminus \cB(0,b_1)\right)} {{\Lambda}}(w,z) \, dz,
\end{align}
where ${{\Lambda}}(w,z)=\E{\varphi\left(\frac{\|z\|^d}{\cK(w,W)}\right)}$, defined in (\ref{def:lambda_refined}).\pagebreak[2]

Recall that $\mathbb{P}^{x}$ denotes the probability $\mathbb{P}$ conditioned on having a point at $x\in \T_n^d$. 

\begin{lem}\label{lem:scaled_high_weight_conc}
    Let $w\in (0,\infty)$ and $x \in \T_n^d$. For both lattice and Poisson models, 
    \begin{align}\label{eq:ber_conc_simpler}
          \frac{1}{n}\sum_{z\in \cU_n(x)}A_{x, z}  \to \Lambda_{[b_1,b_2)}(w),
    \end{align}
    where the {convergence is in probability with respect to $\mathbb {P}^{\,x}$ conditionally on $W_{x} = wn$.}
\end{lem}

\begin{proof}[Proof of Lemma~\ref{lem:scaled_high_weight_conc}]
    Fix $w\in (0,\infty)$ throughout the proof. We begin by proving the lemma for the lattice case.
    Representing the neighbourhood $\cU_n(x)$ of $x$ by $\Z^d \cap \Q_n$  with $x$ mapped to the origin, we get $\cU_n = \Z^d \cap \Q_n \cap \cB({0}, b_2 n^{1/d})\setminus \cB({0}, b_1 n^{1/d})$. Conditionally on $W_{{0}}=w n$, we note that $(A_{0, y})_{y\in \cU_n}$ are independent Bernoulli-$\E{p_{{0},y}(w n, W_y)}$ distributed random variables. 
    Using standard concentration arguments (see e.g.~\cite[Theorem~2.8.]{JanLucRuc00}), it suffices to show that $\frac{1}{n}\mathbb E^{ 0}[\sum_{y\in \cU_n}A_{{0}, y}\mid {W_{0}=w n}] \to \Lambda_{[b_1,b_2)}(w)$ as $n\to \infty$. We write 
      \begin{align}\label{eq:scaled_high_weight_conc_lattice_mean}
        \frac{1}{n}\Cexpzero{\sum_{y \in U_n} A_{{0}, y}}{W_{{0}} =wn} & = \frac{|\cU_n|}{n} \int_{1}^{\infty}f_W(t) \frac{1}{|\cU_n|}\sum_{y \in \cU_n} \varphi\left(\frac{\|y\|^d}{\kappa(wn,t)}\right)dt.
        %  & = \frac{ 1}{n} \sum_{y \in \cU_n} \int_{1}^{\infty}f_W(t)\left( 1 - (1-\varphi\left(\frac{\|y\|^d}{n\kappa(w,t)}\right))\right)dt\\
        % & = \frac{ 1}{n}\int_{1}^{\infty} f_W(t)\left(| \cU_n| - \sum_{y \in \cU_n}  (1-\varphi\left(\frac{\|y\|^d}{n\kappa(w,t)}\right))\right)dt\\
        % & = \frac{|\cU_n|}{n}\left( 1 -  \int_{1}^{\infty}f_W(t)\frac{1}{|\cU_n|}\sum_{y \in \cU_n} (1-\varphi\left(\frac{\|y\|^d}{\kappa(wn,t)}\right))dt\right).
    \end{align} 
    By a slight modification of Lemma~\ref{lem:asymp_behavior_tilde_lambda}, working with $\cU_n$ instead of $V_n$, we get that
    \begin{align*}
        \lim_{n\to \infty}\frac{1}{|\cU_n|}\sum_{y \in \cU_n} \varphi\left(\frac{\|y\|^d}{\kappa(wn,t)}\right) = \frac{1}{g(b_1,b_2)}\int_{[-\frac{1}{2}, \frac{1}{2}]^d \cap \left(\cB(0,b_2)\setminus \cB(0,b_1)\right)}\varphi\left(\tfrac{\|y\|^d}{\cK(w,t)}\right)dy,
    \end{align*}
    since $|\cU_n| = (1+o(1))g(b_1,b_2)n$. We can further conclude by the dominated convergence theorem that
    \begin{align*}
         \frac{1}{n}\Cexpzero{\sum_{y \in \cU_n} A_{{0}, y}}{W_{{0}} =wn} \to \int_{1}^{\infty}\int_{[-\frac{1}{2}, \frac{1}{2}]^d \cap \left(\cB(0,b_2)\setminus \cB(0,b_1)\right)}f_W(t)\,\varphi\left(\tfrac{\|z\|^d}{\cK(w,t)}\right)dzdt.
    \end{align*}

    The proof for the Poisson case is of a similar flavour.  First note that given ${0}\in V_n$ and the weight $W_{{0}}=wn$, the locations and the weights of the vertices contributing to $\sum_{y\in \cU_n}A_{{0}, y}$ are distributed as a Poisson point process on $\R^d \times \R_+$ with intensity
    \begin{align}\label{eq:scaled_high_weight_mean}
        h(y,t) = f_W(t)\,\varphi\left(\frac{\|y\|^d}{\kappa(w n, t)}\right)\I{t\geq 1}\I{y\in \T_n^d\cap (\cB(0,b_2n^{1/d})\setminus \cB(0,b_1n^{1/d}))}.
    \end{align}
    This implies that given ${0}\in V_n$ and the weight $W_{{0}} =wn$, the sum  $\sum_{y\in \cU_n}A_{{0}, y}$ is a Poisson random variable with parameter 
    \[\int_1^{\infty}\int_{\T_n^d \cap (\cB(0,b_2n^{1/d})\setminus \cB(0,b_1n^{1/d}))}h(y,t)dy dt.\]
    This is an analogue of Lemma~\ref{lem:dx_poisson} and we omit the proof as it is a simple adaptation of the proof of \cite[Lemma 2.1]{scaling_clust_23}. Using Poisson concentration arguments \cite[Remark 2.6]{JanLucRuc00}, it suffices to prove that the mean of the Poisson random variable \eqref{eq:scaled_high_weight_mean} converges to the right-hand side of \eqref{eq:ber_conc_simpler}. Just as in the lattice case, this follows from a modified version of Lemma~\ref{lem:asymp_behavior_tilde_lambda}, using that $\E{|\cU_n|} = g(b_1,b_2)n$. We omit the details.
\end{proof}

\begin{lem}\label{lem:lim_scaled_high_weight_edg_contr}
    For lattice and Poisson models and $w_1 \geq \dots \geq w_k >0$, conditionally on $(W_{\sss(1)}, \dots, W_{\sss(k)}) = (w_1 n, \dots, w_k n)$,
    \begin{align*}
        \frac{1}{n}\sum_{i=1}^k\sum_{y\in \cU_n({X_{\sss(i)}})}A_{X_{\sss(i)}, y} \convp \sum_{i=1}^k\Lambda_{[b_1,b_2)}(w_i),
    \end{align*}
    where $\Lambda_{[b_1,b_2)}(w)$ is given by \eqref{def:lambda_interval}. % for $w\in(0,\infty)$.
\end{lem} 
%\RvdH{I think that the condition $w_1<1$ should not be there.}

\begin{proof}
The following proof holds for both lattice and Poisson models. Let $\eps>0$. A union bound gives that 
\begin{align*}
      & \CprobBigg{\Big|\frac{1}{n}\sum_{i=1}^k\sum_{y\in\,  \cU_n({X_{\sss(i)}})}A_{X_{\sss(i)}, y} - \sum_{i=1}^k\Lambda_{[b_1,b_2)}(w_i)\Big|> \eps}{(W_{(j)})_{j\in [k]} = (w_j n)_{j\in [k]}}\\
        & \leq \sum_{i=1}^k\CprobBigg{\Big|\frac{1}{n}\sum_{y\in \, \cU_n({X_{\sss(i)}})}A_{X_{\sss(i)}, y} - \Lambda_{[b_1,b_2)}(w_i)\Big|> \eps/k}{(W_{(j)})_{j\in [k]} = (w_j n)_{j\in [k]}}. \numberthis\label{eq:lim_scaled_high_weight_sum}
\end{align*}
We will show that, for any $i \in [k]$, the summand of the right-hand side above is $o(1)$. To be able to apply Lemma~\ref{lem:scaled_high_weight_conc} requires still some work since the summand above conditions on the $k$ highest weights instead of on $W_{X_{\sss(i)}}$. Fix $i \in [k]$ and recall $\mathbf{X}_{\sss(i)} := \{X_{\sss(1)}, \ldots, X_{\sss(i)}\}$.  Using the triangle inequality we can upper bound the summand by
\begin{align}
     & \CprobBig{\Big|\frac{1}{n}\sum_{\substack{y\in \cU_n({X_{\sss(i)}})}\setminus \mathbf{X}_{\sss(k)}
     }A_{X_{\sss(i)}, y} - \Lambda_{[b_1,b_2)}(w_i)\Big|> \eps/2k}{(W_{\sss(j)})_{j\in [k]} = (w_j n)_{j\in [k]}} \label{eq:lim_scaled_high_weight_edg_contr_first}\\
     & + \CprobBig{\Big|\frac{1}{n}\sum_{j= 1}^k A_{X_{\sss(i)}, X_{\sss(j)}}\Big|> \eps/2k}{(W_{\sss(j)})_{j\in [k]} = (w_j n)_{j\in [k]}}.
\end{align}
Note that the second summand equals zero for sufficiently large $n$ since we have the deterministic bound $\frac{1}{n}\sum_{j=1}^kA_{X_{\sss(i)}, X_{\sss(j)}} \leq k/n < \eps/2k$. Therefore it remains to show that the first summand \eqref{eq:lim_scaled_high_weight_edg_contr_first} tends to zero as $n$ tends to infinity. The fact that the weights $(W_x)_x$ are independent and that the random variable $A_{x,y}$ only depends on the locations $x,y$ and weights $W_x,W_y$, gives that, for $\ell \geq 0$,
\begin{align*}
    & \CprobBig{\sum_{{y\in \cU_n({X_{\sss(i)}})\setminus \mathbf{X}_{\sss(k)}}}A_{X_{\sss(i)}, y}\geq \ell}{(W_{(j)})_{j\in [k]} = (w_j n)_{j\in [k]}} \numberthis \label{eq:scaled_high_weight_edg_contr}\\
    & = \CprobBig{\sum_{\substack{y\in \cU_n({X_{\sss(i)}})\setminus \mathbf{X}_{\sss(k)}}}A_{X_{\sss(i)}, y} \geq \ell}{W_{X_{\sss(i)}} = w_i n,\, W_y \leq w_k n \, \forall y\in \cU_n({X_{\sss(i)}})\setminus \mathbf{X}_{\sss(k)}}\\
    & = \CprobBig{\sum_{\substack{y\in \cU_n({X_{\sss(i)}})\setminus \mathbf{X}_{\sss(k)}}}A_{X_{\sss(i)}, y}\geq \ell}{W_{X_{\sss(i)}} = w_i n,\, W_y \leq w_k n \, \forall y\in \cU_n({X_{\sss(i)}})}\\
    & \leq \CprobBig{\sum_{\substack{y\in \cU_n({X_{\sss(i)}})}}A_{X_{\sss(i)}, y} \geq \ell}{W_{X_{\sss(i)}} = w_i n,\, W_y \leq w_k n \, \forall y\in \cU_n({X_{\sss(i)}})}.\numberthis \label{eq:scaled_high_weight_edg_contr_upper}
\end{align*}
Similarly, we can lower bound \eqref{eq:scaled_high_weight_edg_contr} by
\begin{align}\label{eq:scaled_high_weight_edg_contr_lower}
    \CprobBig{ \sum_{\substack{y\in \cU_n({X_{\sss(i)}})}}A_{X_{\sss(i)}, y} -k \geq \ell}{W_{X_{\sss(i)}} = w_i n,\, W_y \leq w_k n \, \forall y\in \cU_n({X_{\sss(i)}})}.
\end{align}
{Thus, for sufficiently large $n$,  \eqref{eq:lim_scaled_high_weight_edg_contr_first} can be bounded from above by 
\begin{align*}
    \CprobBig{\Big|\frac{1}{n}\sum_{y\in \cU_n({X_{\sss(i)}})} A_{X_{\sss(i)}, y} -\Lambda_{[b_1,b_2)}(w_i)\Big|> \eps'}{W_{X_{\sss(i)}} = w_i n,\, W_y \leq w_k n \, \forall y\in \cU_n({X_{\sss(i)}})},
\end{align*}
for some $\eps'>0$ such that $\eps' < \eps/2k - k/n$. Observing that the location of the vertex $X_{\sss(i)}$ is just a uniformly chosen (lattice) point in $\T_n^d$, the above equals 
\begin{align*}
    & \CprobzeroBig{\Big|\frac{1}{n}\sum_{y\in \cU_n} A_{{0}, y} -\Lambda_{[b_1,b_2)}(w_i)\Big|> \eps'}{W_{{0}} = w_i n,\, W_y \leq w_k n \, \forall y\in \cU_n} \\
    & \leq \frac{\CprobzeroBig{\Big|\frac{1}{n}\sum_{y\in \cU_n} A_{{0}, y} -\Lambda_{[b_1,b_2)}(w_i)\Big|> \eps'}{W_{{0}} = w_i n}}{\CprobzeroBig{ W_y \leq w_k n \, \forall y\in \cU_n}{W_{{0}} = w_i n}}
    .\numberthis \label{eq:scaled_high_weight_upper_bound}
\end{align*}
Note that 
\begin{align*}\CprobzeroBig{W_y \leq w_k n \, \forall y\in \cU_n}{W_{{0}} = w_i n} & = \pzeroBig{ W_y \leq w_k n \, \forall y\in \cU_n}  = \Ezero{\p{W\leq w_k n}^{|\cU_n|}},
\end{align*}
and recall that $|\cU_n| = (1+o(1))g(b_1,b_2)n$ for lattice models and $|\cU_n|$ is ${\rm Poi}(g(b_1,b_2)n)$-distributed for Poisson models. Using the distribution of $W$ given by \eqref{eq:Pareto_density}, we have \smash{$\mathbb E^{0}[\p{W\leq w_k n}^{|\cU_n|}]\rightarrow 0$} for $n\rightarrow \infty$. We can conclude that \eqref{eq:lim_scaled_high_weight_edg_contr_first} tends to zero by applying Lemma~\ref{lem:scaled_high_weight_conc} to the first factor of \eqref{eq:scaled_high_weight_upper_bound}.}
\end{proof}
\pagebreak[3]

\begin{corollary}\label{cor:conv_distr_scaled_high_weight_edges}
       Conditionally on the event $|E_n| > n(\rho + \mu)$, 
        \begin{align*}
            \frac{1}{n}\sum_{i=1}^k\sum_{y\in \,\cU_n({X_{\sss(i)}})}A_{X_{\sss(i)}, y} - \sum_{i=1}^k\Lambda_{[b_1,b_2)}\left(W_{\sss(i)}/n\right) \convp 0,
        \end{align*}
        where 
        %we recall that $k=\lceil \rho \rceil$ and where 
        $\Lambda_{[b_1,b_2)}(w)$ is given by \eqref{def:lambda_interval} for $w\in (0,\infty)$.
\end{corollary}
\begin{proof}
    By Theorem~\ref{thm:main_uldp}, it suffices to show that, for all $\eps>0$
    $$\mathbb P\bigg(\Big|\frac{1}{n}\sum_{i=1}^k\sum_{y\in \cU_n({X_{\sss(i)}})}A_{X_{\sss(i)}
    , y} - \sum_{i=1}^k\Lambda_{[b_1,b_2)}\left(W_{\sss(i)}/n\right)\Big|> \eps,\, |E_n| > n(\rho + \mu)\bigg) = o(n^{-k(\beta -2)}).$$
    Let $b>0$ and recall $\cA_k(b)$ from \eqref{def:event_amb}. We begin by noting that 
    \begin{align*}
        & \mathbb P \bigg( \Big|\frac{1}{n}\sum_{i=1}^k\sum_{y\in \cU_n({X_{\sss(i)}})}A_{X_{\sss(i)}, y} - \sum_{i=1}^k\Lambda_{[b_1,b_2)}\left(W_{\sss(i)}/n\right)\Big|> \eps,\, |E_n| > n(\rho + \mu), \cA_k(b)\bigg)\\
        & \leq \mathbb P \bigg( \Big|\frac{1}{n}\sum_{i=1}^k\sum_{y\in \cU_n({X_{\sss(i)}})}A_{X_{\sss(i)}, y} - \sum_{i=1}^k\Lambda_{[b_1,b_2)}\left(W_{\sss(i)}/n\right)\Big|> \eps,\, \cA_k(b)\bigg)\\
        & \leq \mathbb P \bigg( \Big|\frac{1}{n}\sum_{i=1}^k\sum_{y\in \cU_n({X_{\sss(i)}})}A_{X_{\sss(i)}, y} - \sum_{i=1}^k\Lambda_{[b_1,b_2)}\left(W_{\sss(i)}/n\right)\Big|> \eps,\, W_{\sss(i)}> bn\,\forall i\in [k]\bigg).
    \end{align*}
    %\nmar{is the second inequality just an equality?}
By Lemma~\ref{lem:lim_scaled_high_weight_edg_contr}, the last line is $o(\p{W_{\sss(k)}> bn})$. Using the density of $W$ and the fact that $(W_x)_{x\in V_n}$ are i.i.d., we get that $\p{W_{\sss(k)}> bn} = O(n^{-k(\beta -2)})$.
Next, remark that 
\begin{align*}
    & \mathbb P\bigg(\Big|\frac{1}{n}\sum_{i=1}^k\sum_{y\in \cU_n({X_{\sss(i)}})}A_{X_{\sss(i)}, y} - \sum_{i=1}^k\Lambda_{[b_1,b_2)}\Big(W_{\sss(i)}/n\Big)\Big|> \eps,\, |E_n| > n(\rho + \mu), \cA_k(b)^c\bigg)\\
    & \leq \p{|E_n| > n(\rho + \mu), \cA_k(b)^c}.
\end{align*}
The last line equals $o(n^{-k(\beta -2)})$ by the bound \eqref{eq:intermediate_main_thm_UB} combined with Propositions~\ref{prop:order_bound_emain}, \ref{prop:order_bound_elong} and inequalities \eqref{eq:e_high_ineq_series} and \eqref{eq:linear_weights_sum_ak}.
% a slight variation of the proofs of Propositions~\ref{prop:order_bound_ehigh} and Lemma~\ref{lem:upper_tail_for_linear_weights}. \qedhere
\end{proof}

Finally, we have all the tools to prove Proposition~\ref{prop:conv_distr_scaled_high_weight_edges}.
\begin{proof}[Proof of Proposition~\ref{prop:conv_distr_scaled_high_weight_edges}]
    It is enough to prove the statement for $f = \Itwo{(b_1,b_2)}$, where $0<b_1< b_2<\sqrt{d}/2$.
    % \pmar{The move to Euclidean happens naturally when we fix a center of the torus and thereby represent the torus as a cube.}
    For this choice of $f$, it holds that 
    \begin{align*}
        \int f\Big( \frac{x}{n^{1/d}}\Big) \, d\mu_n = \frac{1}{n}\sum_{_{\{x,y\}\in E_n}} \I{d_n(x,y)/n^{1/d}\in (b_1,b_2)}.
    \end{align*}
    % and
    % \begin{align*}
    %      \int f \, d\mu_{Y_1,\ldots, Y_k} = \sum_{i=1}^k \Lambda_{[b_1,b_2)}(Y_i).
    % \end{align*} 
    By Corollary~\ref{cor:conv_distr_scaled_high_weight_edges} it therefore suffices to prove that, conditionally on $|E_n| \geq n(\rho + \mu)$,
\begin{align}\label{eq:diff_long_edges_vs_high_weight_long_edges}
        \frac{1}{n}\bigg| \sum_{_{\{x,y\}\in E_n} }\I{d_n(x,y)/n^{1/d}\in (b_1,b_2)} - \sum_{i=1}^k \sum_{y\in \cU_n({X_{\sss(i)}})}A_{X_{\sss(i)}, y}\bigg| \convp 0.
    \end{align}
    Remark that
    \begin{align*}
         \sum_{_{\{x,y\}\in E_n}} \I{d_n(x,y)/n^{1/d}\in (b_1,b_2)}& = \frac{1}{2}\sum_{x\in V_n}\sum_{y\in \cU_n(x)}A_{x,y}.
         % & = \frac{1}{2n}\sum_{i=1}^k \sum_{y\in \cU_n^{X_{\sss(i)}}}A_{X_{\sss(i)}, y} + \frac{1}{2n}\sum_{\substack{x\in V_n\\ x\neq X_1, \dots, X_k}}\sum_{y \in \cU_n^{x}}A_{x,y}.
    \end{align*}
    We can rewrite the double sum as  
    \begin{align}
    \label{eq:high_weight_edge_count_rewrite}
        & \sum_{x\in V_n}\sum_{y\in \cU_n(x)}A_{x,y}  = \sum_{i=1}^k \sum_{y\in \cU_n({X_{\sss(i)}})}A_{X_{\sss(i)}, y} +  \sum_{{x\in V_n}\atop{x\not\in\mathbf{X}_{\sss(k)}}}\sum_{y \in \cU_n(x)}A_{x,y}\phantom{thisistobetteralign}\notag 
        \end{align}
      
        \begin{align}
        & \phantom{X} = \sum_{i=1}^k \sum_{y\in \cU_n({X_{\sss(i)}})}A_{X_{\sss(i)}, y} + \sum_{i=1}^k\sum_{{x\in \cU_n({X_{\sss(i)}})}\atop{x\not\in\mathbf{X}_{\sss(k)}}}A_{X_{\sss(i)},x} + \sum_{{x\in V_n}\atop{x\not\in\mathbf{X}_{\sss(k)}}}\sum_{{y \in \cU_n(x)}\atop{y\not\in\mathbf{X}_{\sss(k)}}}A_{x,y}\\
        & \phantom{X} = 2\sum_{i=1}^k \sum_{y\in \cU_n({X_{\sss(i)}})}A_{X_{\sss(i)}, y} - \sum_{i=1}^k\sum_{j=1}^kA_{X_{\sss(i)}, X_{(j)}}\I{X_{(j)} \in \cU_n({X_{\sss(i)}})}+ \sum_{{x\in V_n}\atop{x\not\in\mathbf{X}_{\sss(k)}}}\sum_{{y \in \cU_n(x)}\atop{y\not\in\mathbf{X}_{\sss(k)}}}A_{x,y}.\notag
    \end{align}
    Further, bounding $A_{x,y}$ from above by one, we get that \smash{$\frac{1}{n} \sum_{i=1}^k\sum_{j=1}^kA_{X_{\sss(i)}, X_{\sss(j)}} \leq k^2/n$}, which vanishes as $n$ increases.
    Putting everything together, in order to prove \eqref{eq:diff_long_edges_vs_high_weight_long_edges} we must show that, conditionally on $|E_n| \geq n(\rho + \mu)$,
    \begin{align}\label{eq:sum_long_edges_in_annulus}
         \frac{1}{n} \sum_{\substack{x\in V_n\setminus\,\mathbf{X}_{\sss(k)}}}\sum_{\substack{y \in \cU_n(x)\setminus\,\mathbf{X}_{\sss(k)}}}A_{x,y}\convp 0.
    \end{align}
   First, remark that $\sum_{\substack{x\in V_n\setminus\,\mathbf{X}_{\sss(k)}}}\sum_{\substack{y \in \cU_n(x)\setminus\,\mathbf{X}_{\sss(k)}}}A_{x,y}\I{W_x \vee W_y \leq n^a} \leq |E_{\rm long}|$, since $b_1>\eps_n$ for sufficiently large $n$. Thus, Proposition~\ref{prop:order_bound_elong} and Theorem~\ref{thm:main_uldp} give that, for all $\delta>0$,
   \[\CprobBig{\sum_{\substack{x\in V_n\setminus\,\mathbf{X}_{\sss(k)}}}\sum_{\substack{y \in \cU_n(x)\setminus\,\mathbf{X}_{\sss(k)}}}A_{x,y}\I{W_x \vee W_y \leq n^a} >\delta n}{|E_n| \geq n(\rho + \mu)} = o(1).\] 
   Further, by a slightly modified version of the upper bound \eqref{eq:e_high_ineq_series} together with Lemma~\ref{lem:magic_lemma}, 
    \[\CprobBig{\sum_{\substack{x\in V_n\setminus\,\mathbf{X}_{\sss(k)}}}\sum_{\substack{y \in \cU_n(x)\setminus\,\mathbf{X}_{\sss(k)}}}A_{x,y}\I{n^a \leq W_x \leq \eps n} >\delta n}{|E_n| \geq n(\rho + \mu)} = o(1).\] 
    Lastly, 
    \begin{align*}
        & \CprobBig{\sum_{\substack{x\in V_n\setminus\,\mathbf{X}_{\sss(k)}}}\sum_{\substack{y \in \cU_n(x)\setminus\,\mathbf{X}_{\sss(k)}}}A_{x,y}\I{ W_x \geq \eps n} >\delta n}{|E_n| > n(\rho + \mu)} \\
        & \leq \Cprob{W_{\sss(k+1)} \geq \eps n}{|E_n| \geq n(\rho + \mu)} = o(1),
    \end{align*}
    where the equality follows from the proof of  Lemma~\ref{lem:upper_tail_for_linear_weights}. By \eqref{eq:high_weight_edge_count_rewrite}, combining the above three bounds concludes the proof of \eqref{eq:sum_long_edges_in_annulus}.
\end{proof}

\section{Limit for the weight order statistics: Proof of Proposition~\ref{prop:heavy_wt_density}}
\label{sec:heavy_wt_density}

% Let us state and prove a Corollary that will be useful for proving Theorem \ref{thm:distances} at this point.

\begin{proof}[Proof of Proposition~\ref{prop:heavy_wt_density}]
     Fix positive reals $a_1 \geq \dots \geq a_{k}>0$. 
     Similar to the definition \eqref{def:event_amb}, define the event  
     \begin{align}
         \cH_k = \cH_{k}(a_1,\ldots,a_k) := \{W_{\sss(i)}> a_i n\;\;\text{for all}\;\;1\leq i \leq k\}.
     \end{align}
   To prove Proposition~\ref{prop:heavy_wt_density} for lattice and Poisson models, we will show that
        \begin{align*}
    & \p{\cH_k \cap \{|E_n|\geq (\mu+\rho)n\}} \\
    & =(1+o(1))n^{-k(\beta - 2)}(\beta -1)^k\int_{a_1}^{\infty}\dots \int_{a_{k}}^{\infty}\I{y_1>\dots>y_{k}}\I{\sum_{i=1}^{k}{\Lambda}(y_i)>\rho}\prod_{i=1}^{k}y_i^{-\beta}dy_i. \numberthis \label{eq:to_show_high_wt_density}
    \end{align*}
  %  Recall $E_{\rm high}$ from (\ref{def:ehigh}).  
  By the proof of Theorem~\ref{thm:main_uldp}, it suffices to show that $\p{\cH_k \cap  \{|E_{\rm high}|> n\rho\}}$ equals the right-hand side of \eqref{eq:to_show_high_wt_density} in order to conclude  \eqref{eq:to_show_high_wt_density}. Following the proof of \eqref{eq:e_high_ineq_series} and using Lemma~\ref{lem:conc_tilde_d}, we can deduce that, for small $\delta>0$ and sufficiently large $n$,
    \begin{align*}
        \mathbb P(\cH_k\, \cap\, & \{|E_{\rm high}|>\rho n\}) \\ &\leq  \p{\cH_k \cap \left\{ \sum_{x\in V_n} \lambda_n(W_x)\I{W_x\geq n^a} >(\rho/(1+\delta))n\right\}} + o(n^{-k(\beta -2)}),
    \end{align*}
    and 
    \begin{align*}
        \mathbb P(\cH_k \, \cap \,& \{|E_{\rm high}|>\rho n\}) \\ & \geq  \p{\cH_k \cap \left\{ \sum_{x\in V_n} \lambda_n(W_x)\I{W_x\geq n^a} >(\rho/(1-\delta))n\right\}} + o(n^{-k(\beta -2)}).
    \end{align*}
    Combining \eqref{eq:uppertail_linear_weights_lwbd} and \eqref{eq:uppertail_linear_weights_uppbd} together with slight variants of Lemma~\ref{lem:magic_lemma} and the proof of Lemma~\ref{lem:upper_tail_for_linear_weights}, we can see that \eqref{eq:to_show_high_wt_density} follows if we can prove that \smash{$\mathbb P(\cH_k \cap \{ \sum_{i = 1}^{k} \lambda_n(W_{\sss(i)}) >\rho n\})$} equals the right-hand side of \eqref{eq:to_show_high_wt_density}. We begin to prove this for lattice models. Note that
    \begin{align*}
        & \p{\cH_k \cap \Big\{ \sum_{i = 1}^{k} \lambda_n(W_{\sss(i)}) >\rho n\Big\}}\\
        & = (N)_k (\beta-1)^{k}\int_{a_1 n}^{\infty}\dots \int_{a_k n}^{\infty}\I{\sum_{i=1}^{k}\lambda_n(x_i)>\rho n}\I{x_1>\dots>x_{k}}\prod_{i=1}^{k}x_i^{-\beta}dx_i,
    \end{align*}
   % Remark that $N = (1+o(1))n$. 
  where $(N)_k:=N(N-1)\dots(N-k+1)$ is the falling factorial. By the change of variables $x_i/n = y_i$ for all $i\in [k]$, the above equals up to a $(1+o(1))$ factor,
    \begin{align*}
        n^{-k (\beta -2)}(\beta-1)^{k}\int_{a_1}^{\infty}\dots\int_{a_{k}}^{\infty}\I{\sum_{i=1}^{k}\frac{\lambda_n(ny_i)}{n}>\rho}\I{y_1>\dots>y_{k}}\prod_{i=1}^{k}y_i^{-\beta}dy_i,
    \end{align*}
    since $N = n$ in the lattice case. Using Lemma~\ref{lem:asymp_behavior_tilde_lambda}, we conclude that 
    \begin{align*}
        & \p{\cH_k \cap \Big\{ \sum_{i = 1}^{k} \lambda_n(W_{\sss(i)})> \rho n\Big\}} \\
        & = (1+o(1))n^{-k(\beta - 2)}(\beta -1)^k\int_{a_1}^{\infty}\dots \int_{a_{k}}^{\infty}\I{y_1>\dots>y_{k}}\I{\sum_{i=1}^{k}{\Lambda}(y_i)>\rho}\prod_{i=1}^{k}y_i^{-\beta}dy_i,
    \end{align*}
    and so Proposition~\ref{prop:heavy_wt_density} follows. The proof for Poisson models follow by conditioning on the number of points {$N$ and using that $N$ is a $\text{Poi}(n)$-distributed random variable.} More precisely, recall $I_M(n)$ given by \eqref{def:interval_n_sqrt} for $M>0$, and note that for a given inverse power of $n$, we can choose a sufficiently large $M$ such that $\p{{N} \notin I_{M}(n)}$ tends to zero faster than this inverse power as $n \to \infty$. %since $|\cP_n|$ is a Poi($n$)-distributed random variable. 
    Thus,
    \begin{align*}
         & \p{\cH_k \cap \big\{ \sum_{i = 1}^{k} \lambda_n(W_{\sss(i)}) >\rho n\big\}}\\
         & = \E{\I{{N}\in I_M(n)}\Cprob{\cH_k \cap  \Big\{ \sum_{i = 1}^{k} \lambda_n(W_{\sss(i)}) >\rho n\Big\}}{{N}}} + o(n^{-k(\beta -2)}).
    \end{align*}
    The remainder of the proof follows straightforwardly and is omitted. \qedhere
\end{proof}

\section{The empirical degree distribution: Proof of Proposition~\ref{prop:degree_distr}}\label{sec:degree_distribution}

We begin the proof of 
Proposition~\ref{prop:degree_distr}\ref{eq:degree_distr_large_degrees} %\eqref{eq:k_high_deg_convp} 
by following that of Proposition \ref{prop:conv_distr_scaled_high_weight_edges} in Section~\ref{sec:lenghts}. For lattice and Poisson models, recall that Lemma~\ref{lem:scaled_high_weight_conc} holds when we take the sum over all of $V_n$, that is for $b_1 = 0$ and {\smash{$b_2 =\sqrt{d}/2$}}. 
%\phmar{Does the result of this lemma not use that $0 < b_1 < b_2 < \sqrt{d}/2$?} \nmar{It doesn't, I have slightly changed the statement to allow for $0\leq b_1<b_2\leq \frac{\sqrt{d}}{2}$. The proof is the same.} 
This gives, for $w\geq 0$, $x\in \T_n^d$ and conditionally on $W_x = wn$, that 
\begin{align}
    \frac{1}{n}\sum_{y \in V_n}A_{x,y} \to \Lambda(w),
\end{align}
where the convergence is in probability with respect to $\mathbb{P}^x$. 
%{and we recall that $\mathbb{P}^x$ denotes $\mathbb{P}$ conditioned on having a point at $x\in \T_n^d$.}{}%\pmar{Is $\mathbb{P}^x$ globally defined?}
By adapting the proof of Lemma~\ref{lem:lim_scaled_high_weight_edg_contr}, starting from \eqref{eq:lim_scaled_high_weight_sum} and summing again over all of $V_n$, we can show that, conditionally on the event $(W_{\sss(1)}, \ldots, W_{\sss(k)}) = (y_1n,\ldots, y_k n)$ for $y_1 \geq \ldots \geq y_k>0$, 
    \begin{align}\label{eq:deg_convp_cond_weights}
        \frac{1}{n}\bigg(\sum_{y\in V_n}A_{X_{\sss(1)}, y},\ldots, \sum_{y\in V_n}A_{X_{\sss(k)}, y}\bigg) \convp \left(\Lambda(y_1),\ldots,\Lambda(y_k)\right),
    \end{align}
    for lattice and Poisson models, proving 
    Proposition~\ref{prop:degree_distr}\ref{eq:degree_distr_large_degrees}. %\eqref{eq:k_high_deg_convp}.
    %\RvdH{Previously, we used $a_i$ for the highest weights?}
    \smallskip
    
    Next, to show Proposition~\ref{prop:degree_distr}\ref{eq:degree_distr_lim_distr}, 
    %\eqref{eq:pin_convp} 
    we use a second-moment method. Let $O_1, O_2$ be two independent uniformly chosen random vertices in $V_n$.
    For
%    $y_1 \geq \cdots \geq y_k>0$ and  
%$u_{i} \in \mathbb T_1^d$,
{$u_1, \dots, u_k \in \mathbb{T}_1^d$} 
    we define
    $$C_n := \{ (W_{\sss(i)})_{i\in [k]} = (y_{i}n)_{i\in [k]},\, (X_{\sss(i)})_{i\in [k]} =  (n^{1/d}u_{i})_{i\in [k]}\},$$
  where, in the lattice case, we project $n^{1/d}u_i$ to the nearest lattice point. {Recall~\eqref{eq:def_two_degrees}}. We can write
    \begin{align*}
         & \mathbb E[ (\pi_{a,b}^{\sss(n)})^2 \, | \,C_n]
         = \mathbb P \big(D_{O_1}^{\bb} =a,\,D_{O_1}^{\bh} =b,\,D_{O_2}^{\bb} =a,\, D_{O_2}^{\bh} =b \, \big| \, C_n \big) \\
           & = \mathbb P(D_{O_1}^{\bb} =a,\,D_{O_1}^{\bh} =b,\,D_{O_2}^{\bb} =a,\, D_{O_2}^{\bh} =b; O_1,O_2 \notin \{X_{\sss(1)},\dots,X_{\sss(k)}\} \, | \, C_n)+o(1).\numberthis\label{eq:square_degree_distr}
    \end{align*}
We now {first}{} focus on the lattice case. For distinct points $o_1, o_2 \in V_n$, we define 
%the event $C_n(o_1, o_2)  = C_n \cap \{O_1=o_1, O_2=o_2\}$ and
$$f(o_1, o_2):= \mathbb P \big(D_{o_1}^{\bb} =a, \,D_{o_1}^{\bh} =b,\,D_{o_2}^{\bb} =a,\, D_{o_2}^{\bh} =b, \,  \{o_1,o_2\} \notin E_n \,\big| \, C_n \big).$$

Then \eqref{eq:square_degree_distr} can be written as
\begin{align*}
    \mathbb E\Big[ f(O_1, O_2) \I{O_1,O_2 \notin \{X_{\sss(1)},\dots,X_{\sss(k)}\}}\Big| \, C_n \Big]+o(1),
\end{align*}
since on the event that $O_1$ and $O_2$ are not high-weight vertices, with high probability they do not share an edge. To see this, note that the weights of the vertices in $V_n\setminus \mathbf{X}_{(k)}$ are independent copies of a variable %$\hat{W}$ 
{$\hat{W}_n$}
 having the conditional law of the weights
$W$ in \eqref{eq:Pareto_density} 
{conditionally}{} on $\{W<y_k n\}$, and their locations are distributed as the locations of $V_n$ with $k$ uniformly random vertices removed. In particular, the probability that $O_1$ and $O_2$ share an edge on the event that they are not in $\mathbf{X}_{(k)}$, equals up to an asymptotically negligible error, the probability that two random vertices share an edge in the graph $G_n$ where each vertex in $V_n$ now has an independent copy of the weight $\hat{W}_n$, instead of~$W$. The latter probability goes to zero using Assumption~
\hyperref[assumption_a]{A} %\eqref{asspt:kernel_bounds_var}, \eqref{asspt:limiting_kernel} and \eqref{asspt:profile_bounds} 
and a straightforward calculation. 
\smallskip

%\phmar{Do we have a result for this claim? We use it again later.} \nmar{Hopefully this is okay? Didn't want to make a separate lemma for this.}
Next, for convenience, let us introduce the notation $D_{x \rightarrow A} = D_{x \rightarrow A}(G_n)$ for $x \in V_n$ and $A \subseteq V_n$ to denote the number of edges connecting the vertex $x$ to the set $A$ in the graph $G_n$. Recall that ${\bX_{\sss(k)}}$ stands for the set $\{X_{\sss(1)},\dots, X_{\sss(k)}\}$. Then 
\begin{align}\label{eq:degree_distr_two_indep}
 f(o_1, o_2)=   &\mathbb{P}\big(D_{o_1 \rightarrow V_n \setminus (\bX_{\sss(k)} \cup \{o_2\})} =a,\,D_{o_1 \rightarrow \bX_{\sss(k)}} =b,\,D_{o_2\rightarrow V_n \setminus (\bX_{\sss(k)} \cup \{o_1\})} =a,\notag\\& \hspace{10 pt}\, D_{o_2 \rightarrow \bX_{\sss(k)}} =b,  \{o_1,o_2\} \notin E_n \mid C_n \big).
\end{align}
By conditioning on the weights $W_{o_1}$ and $W_{o_2}$, and using the conditional independence of the events, %\eqref{eq:degree_distr_two_indep} equals 
\begin{align*}
f(o_1,o_2) & = \mathbb{E}\Big[
    \Cprob{D_{o_1 \rightarrow V_n \setminus (\bX_{\sss(k)} \cup \{o_2\})} =a}{W_{o_1}, \,C_n} \Cprob{D_{o_1 \rightarrow \bX_{\sss(k)}} =b}{W_{o_1}, \,C_n} \\&\qquad \cdot \Cprob{D_{o_2 \rightarrow V_n \setminus (\bX_{\sss(k)} \cup \{o_1\})} =a}{W_{o_2}, \,C_n} \Cprob{D_{o_2 \rightarrow \bX_{\sss(k)}} =b}{W_{o_2}, \,C_n} \\&\qquad \cdot \Cprob{\{o_1,o_2\} \notin E_n }{W_{o_1}, W_{o_2}, \, C_n} \Big| \, C_n \Big].\numberthis\label{eq:cond_degree_distr_two_indep}
\end{align*}
Recall that $D_{\sss\infty}(w)$ denotes the degree of the root of the local limit of the unconditional graph given {that}{} the weight of the root is $w$. We claim that, for fixed $w_1\geq 1$ and
{with}
    $$
  h_{w_1}(o_1, o_2):=  \Cprob{D_{o_1 \rightarrow V_n \setminus (\bX_{\sss(k)} \cup \{o_2\})} =a}{W_{o_1}=w_1, C_n},
$$
we have
\begin{align}\label{eq:conv_degree_o1}
  h_{w_1}(O_1, O_2) \convp \p{D_{\sss \infty}(w_1)=a}.
\end{align}
{To prove \eqref{eq:conv_degree_o1},}{} as observed earlier, we recall that the weights of vertices in $V_n \setminus \bX_{\sss(k)}$ are independent copies of a variable %$\hat{W}$ 
{$\hat{W}_n$} 
having the conditional law of the weights
$W$ in \eqref{eq:Pareto_density} 
{conditionally}{} on $\{W<y_k n\}$. We now show how to use this observation to couple the graph $G_n$, conditionally on $C_n$, with an unconditional geometric random graph $G_n'$ 
%\RvdH{Is the notation $\G$ used consistently?}
%\PM{I think this is the only place where we couple and therefore two graphs live on the same probability space. Everywhere else we only have one random graph per probability space. Maybe $G_n'$ instead of $\G_n$ would be better. }
where the weights are i.i.d.\ copies of~$\hat{W}_n$. Sample independent copies of $\hat{W}$ for every vertex in $V_n$
and, given the weights, edges between pairs of distinct vertices using the connection function $\varphi$. The graph $G_n'$ is formed using these edges.  For the conditional graph $G_n$, we modify the weights of the vertices
$X_{\sss(i)} = n^{1/d}u_{i}$ to be $y_{i}n$, for all $i\in [k]$, and resample (only) the edges adjacent to 
$\mathbf X_{\sss(k)}$ using these weights. \pagebreak[3]
\smallskip

Note that the graph $G_n'$ has the same local limit as the original {unconditioned} graph (with weights of law $W$) since $\hat{W} \convdist W$ as $n\to \infty$, and using the fact that the local limit is preserved under removal of a finite set of vertices. Therefore the degree of $o_1$ in~$G_n'$, given $W_{o_1} = w_1$, converges in distribution to $D_{\sss \infty}(w_1)$. We also observe that with high probability the vertex $O_1$ has distance of order $n^{1/d}$ from the locations of the vertices in $\bX_{\sss(k)} \cup \{O_2\}$. 
Hence, with high probability, $O_1$ does not share an edge with the vertices at $\mathbf X_{\sss(k)} \cup \{O_2\}$ and the degree of $O_1$ in $G_n'$ is the same as its degree in the subgraph of $G_n'$, induced by the vertices $V_n \setminus (\bX_{\sss(k)} \cup \{ O_2\})$. {Since} in our coupling, {these} induced subgraphs of $G_n'$ and $G_n$ match, we obtain that in fact $h_{w_1}(O_1, O_2)\convp \p{D_{\sss \infty}(w_1)=a}$. {This proves~\eqref{eq:conv_degree_o1}.}{} 
\smallskip

%\cel{We also observe that with high probability $o_1$ is not connected to any vertex in $\bX$ in the graph $\G_n$ as, for some $\sK>0$, \begin{align*}    & \Cprob{D_{o_1 \to \bX}(\G_n)\geq 1}{W_{o_1} = w_1, C_n(o_1, o_2)}\\     & \leq \frac{k}{n}\int_{Q_n}\int_1^{ny_k}(\beta - 1)t^{-\beta}\varphi\left(\frac{\|x\|^d}{\kappa(w_1,t)}\right)dxdt \leq \sK w_1/n, \end{align*} where 
%$Q_n$ is the $d$-dimensional cube of volume $n$ centered at the origin and 
%\pmar{$Q_n$ should be globally defined.}
%the second inequality holds by \eqref{eq:bounds_pi_key} from Lemma~\ref{lem:bounds_pi_conditional}. In particular, $D_{o_1 \to V_n}(\G_n) = D_{o_1 \to V_n\setminus \bX}(\G_n)$ with high probability. By the coupling, the subgraphs induced by the vertices $V_n \setminus \bX$, of $\G_n$ and $G_n$ match and thus  
%$$D_{o_1 \to V_n\setminus (\bX \cup \{o_2\})}(\G_n) = D_{o_1 \to V_n\setminus (\bX \cup \{o_2\})}(G_n).$$ We obtain that in fact $\Cprob{D_{o_1 \rightarrow V_n \setminus \{\bX \cup \{o_2\}\}}(G_n) =a}{W_{o_1}=w_1, C}\to \p{D_{\sss \infty}(w_1)=a}$, as claimed in \eqref{eq:conv_degree_o1}.}
Define
    \smash{$g_{w_1}(o_1, o_2):=\mathbb P(D_{o_1 \rightarrow \bX_{\sss(k)}}=b \mid W_{o_1}=w_1, C_n).$}
We analyse the asymptotic behaviour of $g_{w_1}(O_1, O_2)$. In the {conditioned} graph $G_n$, we can write \smash{$D_{o_1 \rightarrow \bX_{\sss(k)}}=\sum_{i=1}^k {\tilde B^{\sss(i)}_{o_1,w_1}},$}
%\end{align*}
where the random variables \smash{$\tilde B^{\sss(i)}_{o_1,w_1}$} are independent and Bernoulli distributed with parameter \smash{$\varphi({d_n(o_1,  n^{1/d}u_{i})^d}/{\kappa(w_1, ny_{i})})$}. Using that $\varphi$ is continuous almost everywhere, we get%
\begin{align}\label{parametersconverge}
\varphi\bigg(\frac{d_n(O_1, n^{1/d}u_{i})^d}{\kappa(w_1, ny_{(i)})}\bigg) \convdist \varphi \left(\frac{d_1(U_1,u_i)^d}{\cK(y_{i},w_1)} \right),\end{align}
where $U_1$ is uniformly distributed on $\mathbb T_1^d$.
The random variables \smash{$\tilde B^{\sss(i)}_{O_1,w_1}$} therefore converge in distribution to conditionally independent Bernoulli random variables $B^{\sss(i)}_{U_1,w_1}$ with parameter \smash{$\varphi ({d_1(U_1,u_i)^d}/{\cK(y_{i},w_1)})$} given $U_1$. Therefore, 
$$g_{w_1}(O_1, O_2) \convdist \mathbb P\Big(\sum_{i=1}^k B^{(i)}_{U_1,w_1}=b~\Big|~ U_1 \Big).$$
Note that $(W_{O_1}, O_1)$ and $(W_{O_2}, O_2)$ converge in distribution to two independent copies $(W_1, U_1)$ and $(W_2, U_2)$ of the pair $(W, U)$ with a Pareto-distributed random variable $W$ and an independent uniform variable $U\in\mathbb T_1^d$. As $g,h$ are continuous and bounded,  
\begin{align*}
\mathbb E\Big[ f(O_1, O_2) & \I{O_1,O_2 \notin \{X_{(1)},\dots,X_{(k)}}\} ~\Big|~ C_n \Big] \\
& =
\mathbb E\big[ h_{W_{O_1}}(O_1, O_2) g_{W_{O_1}}(O_1, O_2) h_{W_{O_2}}(O_2, O_1) g_{W_{O_2}}(O_2, O_1) \mid C_n\big]
\end{align*}
converges to 
\begin{align*}
& \mathbb E\bigg[\mathbb P (D_{\sss \infty}(W_1) = a | W_1) 
    \mathbb P(D_{\sss \infty}(W_2)=a \mid W_2) \\ & \phantom{xxxxxx} \cdot \mathbb P\Big(\sum_{i=1}^k B^{(i)}_{U_1,W_1}=b ~\Big|~ U_1, W_1\Big) \mathbb P\Big( \sum_{i=1}^k B^{(i)}_{U_2,W_2}=b~\Big|~ U_2, W_2\Big)\bigg]
\end{align*}
  \begin{align*}  
& =  \mathbb E \bigg[ \mathbb P(D_{\sss \infty}(W)=a \mid W) \mathbb P\Big( \sum_{i=1}^k B^{(i)}_{U,W}=b ~\Big|~ W\Big)  \bigg]^2.
\end{align*}
Combining the above with \eqref{eq:square_degree_distr} and \eqref{eq:cond_degree_distr_two_indep} gives
\begin{align*}
%\label{eq:exp_pi_ab_square_order}
\mathbb E[(\pi_{a,b}^{{\sss(n)}})^2 \mid C_n]=(1+o(1))(\pi_{a,b}((y_{i},u_{i})_{i \in [k]}))^2.
\end{align*}
The above computations also imply that
\begin{align*}
    \mathbb E[\pi_{a,b}^{_{(n)}} \mid C_n] = (1+o(1)) \pi_{a,b}((y_{i},u_{i})_{i \in [k]}).
\end{align*}
Hence the variance of 
$\pi_{a,b}^{_{(n)}}$
vanishes asymptotically and an application of Chebyshev's inequality
%yields \begin{align*}
%\mathbb P\big( & \big|\pi_{a,b}^{\sss(n)}- \pi_{a,b}( (y_{i},u_{i})_{i\in[k]})\big|>\eps \mid C\big) \\
%& \leq \eps^{-2} \left( \Cexp{(\pi_{a,b}^{\sss(n)})^2}{C} -  (\pi_{a,b}((y_{i},u_{i})_{i \in [k]}))^2\right) + o(1),
%\end{align*}
finishes the proof of Proposition~\ref{prop:degree_distr}\ref{eq:degree_distr_lim_distr} %\eqref{eq:pin_convp} 
for the lattice case. The Poisson case is analogous using that removing finitely many randomly chosen vertices from the Poisson point process leaves its law intact.\pagebreak[3]\medskip

We next prove 
Proposition~\ref{prop:degree_distr}\ref{eq:degree_distr_small_degrees}, %\eqref{eq:cond_scaled_low_weight_vertices}, 
that is, $\mathbb E[D_{X_{\sss(j)}}(G_n)|C_n] = o(n)$ for each $j>k$. By the coupling of the conditional graph $G_n$ and the unconditional graph $G_n'$,  for $j>k$, 
 $$D_{X_{\sss(j)}}(G_n) \leq D_{X_{\sss(j)} \to V_n \setminus \bX_{\sss(k)}}(G_n) + k = D_{X_{\sss(j)} \to V_n \setminus \bX_{\sss(k)}}(G_n') + k \leq D_{X_{\sss(j)}}(G_n') + k.$$
Let $Y$ be the maximum of $n$ i.i.d.\,random variables with law \eqref{eq:Pareto_density}. Then, for $\eps>0$ there is  $c>0$ such that $\p{Y > cn^{1/(\beta -1)}} < \eps$. In $G_n'$ we have $W_{\sss X_{\sss(j)}} \preceq_{st} Y$ and thus 
\begin{align*}
    \frac{1}{n}  \E{D_{X_{\sss(j)}}(G_n')} & \leq \E{\varphi\left(\frac{d_n(n^{1/d}u_{j}, U)^d}{\kappa(Y, W)}\right)} \leq \E{\varphi\left(\frac{\|U\|^d}{\kappa(cn^{1/(\beta - 1)}, W)}\right)} + \eps,
\end{align*}
where $W$ has law \eqref{eq:Pareto_density} and $U$ is a random uniform point in $V_n$, and where the inequality holds since $\kappa$ is increasing and $\varphi$ decreasing.By Lemma~\ref{lem:calc} we can bound the above by a constant multiple of $n^{-1}n^{1/(\beta-1)}$. This vanishes as $n$ tends to infinity, concluding the proof of 
Proposition~\ref{prop:degree_distr}\ref{eq:degree_distr_small_degrees}.

%\PvdH{I do not see how Lemma~\ref{lem:bounds_pi_conditional} is used. The expression in the statement does not resemble the expression being considered here.}
%\PM{The lemma is not used. Only the integral that needs to be bounded here is similarly bounded in the proof of Lemma~\ref{lem:bounds_pi_conditional}.}

\subsection*{Concluding remarks on local limits.}
   Note that the proof of Proposition~\ref{prop:degree_distr} shows that the subgraph of the conditional graph $G_n$ given $|E_n|\geq n(\mu+\rho)$ induced by removing the vertices with the $k$ highest weights has a local limit which {equals}{} the unconditional local limit $G_{\infty}$ of $G_n$. The local limit of the conditional graph $G_n$ given $|E_n|\geq n(\mu+\rho)$ itself, however, does not exist. This is because with asymptotically positive probability a randomly sampled vertex connects to one of the $k$ high-weight vertices, and as these vertices have infinite degree in the limit, the two-neighbourhood of a randomly sampled vertex becomes infinite with positive probability. However, {one can still give}{} a local description of the conditional graph, which we now describe informally.
  \smallskip

   Sample a random vertex $o_n$ in the conditional graph. With high probability, it is not one of the $k$ high-weight vertices. Now start exploring the neighbourhood of this vertex in some order, e.g., the depth-first order. At every step of the exploration, when we discover an incidence to one of the $k$ high-weight vertices, we 
   stop exploring from that vertex, and go back to the last visited vertex {to restart}{} exploring from there. In particular, this method of exploration avoids discovering the `forward' neighbours of a high-weight vertex, of which there are linearly %(in the graph size) 
   many, as we have seen. This `restricted' exploration can be formalized easily, {which we omit to}{} keep the discussion informal. 
   Now, say we explore the conditional graph in this fashion, and look at the set of explored vertices that are at graph distance $r$ from $o_n$, and call this explored graph $B_n(r)$. Then, $B_n(r)$, rooted at $o_n$, for any fixed $r$, converges in the local (Benjamini-Schramm) topology (see \cite[Chapter 2]{RGCN_2}) to the graph~$B_r$, whose construction we describe next.
\smallskip

   We briefly describe the local limit $G_{\infty}$ of the unconditional graph $G_n$. The vertex set of $G_{\infty}$ is the corresponding infinite vertex set - the infinite integer lattice $\mathbb{Z}^d$ in the lattice case and the Palm version $\Gamma \cup \{0\}$ of a standard Poisson process $\Gamma$ on $\R^d$ in the Poisson case. Then the weight for each vertex is sampled as an independent random variable with density \eqref{eq:Pareto_density}, and conditionally on the weights and locations, edges are placed independently between vertex pairs with probability \eqref{eq:connection_func}. The graph $G_{\infty}$ is rooted at $0$.
   \smallskip
   
   Now we describe the graph $B_r$. Let $B_{\infty}(r)=(V_{\infty}(r), E_{\infty}(r))$ be the graph neighbourhood of radius $r$ about $0$ in  $G_{\infty}$. Let the weights of the vertices in $V_{\infty}(r)$ be $(W(v))_{v \in V_{\infty}(r)}$. Note that the $W(v)$ are not i.i.d., since the {weights are tilted due to the fact that they are}{} in the $r$-neighbourhood of $0$, {and the precise law of $(W(v))_{v \in V_{\infty}(r)}$ is rather involved.}{} The construction of the graph $B_r$ is as follows.
   \begin{itemize}
       \item[$\rhd$] Append $k$ extra vertices, say $v_{1},\dots, v_{k}$, to $V_{\infty}(r)$ and call the new vertex set~$V_r$.
       \item[$\rhd$] {The vertices $(v_{i})_{i\in[k]}$ have marks $(U_{i}, Y_{i})$, where the first coordinates are i.i.d.\ uniform points of $[-\frac12,\frac12]^d$ independent of $(Y_i)_{i \in [k]}$, and the second coordinates have joint distribution \eqref{eq:prob_y}.}{}
       \item[$\rhd$] %Conditionally on the collections $(W_v)_{v \in V_{\infty}(r)}$ and $((U_{(i)})_{i \in [k]},(Y_{(i)})_{i \in [k]})$, 
       Given the variables sampled above, we place an edge between each vertex $v$ of $V_{\infty}(r)$ and each $v_i,$ $i\in[k],$ independently, with probability $\varphi (\|U_i\|^d/\mathcal{K}(Y_i,W(v))).$ Denote the new edge set by $E_r$.
       \item[$\rhd$] The graph $B_r$ has vertex set $V_r$ and edge set $E_r$.
   \end{itemize}
   This description gives a way to compute the typical distance (graph distance between two randomly sampled vertices) in the {conditional graph $G_n$}: Independently sample two random vertices $o_1$ and $o_2$ {in $G_n$}{}.  
   Start exploring in the fashion described above from both $o_1$ and $o_2$, and let $d_1$ (resp., $d_2$) be the distance from $o_1$ (resp., $o_2$) up to which one has to explore to see a high-weight vertex. Denote by $H_i$ the set of high-weight vertices exactly at graph distance $d_i$ from $o_i$, for $i=1,2$. Recalling from Remark \ref{rem:hub_clique} that the high-weight vertices form a clique, we note that the graph distance from $o_1$ to $o_2$ in the conditional graph is $d_1+d_2+\I{H_1 \cap H_2 = \emptyset}$ with high probability.
   %\RvdH{This is not quite correct. At distance $d_i$, there could be {\em several} hubs, and when one of these agree, then the distance is $d_1+d_2$. Discuss.}
   This quantity as a sequence in $n$ is in fact a tight sequence of random variables - starting from a random vertex in the conditional graph, one finds a high-weight vertex within tight graph distance from it. 
   \pagebreak[3]

\addtocontents{toc}{\protect\setcounter{tocdepth}{-1}}

\medskip

\addtocontents{toc}{\SkipTocEntry} %TO SUPPRESS ACKNOWLEDGEMENTS FROM TABLE OF CONTENTS

\medskip

\paragraph{\bf Acknowledgements}
CK and PM are supported by DFG project 444092244 ``Condensation in random geometric graphs" within the priority programme SPP~2265. The work of RvdH is supported in part by the Netherlands Organisation for Scientific Research (NWO) through Gravitation-grant {\sf NETWORKS}-024.002.003.

\addtocontents{toc}{\protect\setcounter{tocdepth}{2}}

%%%%%%%%%%%%%%%%%%%%%%%%%%%%%%%%%%%%%%%%%%%%%%%%%%%%%
%\small
%\addtocontents{toc}{\SkipTocEntry} %TO SUPPRESS LIST OF NOTATION FROM TABLE OF CONTENTS
%\printnomenclature[3.1cm]
%\normalsize
%%%%%%%%%%%%%%%%%%%%%%%%%%%%%%%%%%%%%%%%%%%%%%%%%%%%

\footnotesize

%Use these commands if you're using a bib file called citation.bib 
\bibliographystyle{plainnat}
\bibliography{citation}

\begin{thebibliography}{28}
\providecommand{\natexlab}[1]{#1}
\providecommand{\url}[1]{\texttt{#1}}
\expandafter\ifx\csname urlstyle\endcsname\relax
  \providecommand{\doi}[1]{doi: #1}\else
  \providecommand{\doi}{doi: \begingroup \urlstyle{rm}\Url}\fi

\bibitem[Adams and Dickson(2021)]{AD21}
S.~Adams and M.~Dickson.
\newblock {An explicit large deviation analysis of the spatial cycle
  Huang–Yang–Luttinger model}.
\newblock \emph{Annales Henri Poincar\'e}, 22\penalty0 (5):\penalty0
  1535--1560, 2021.

\bibitem[Berestycki and Yadin(2019)]{BY19}
N.~Berestycki and A.~Yadin.
\newblock {Condensation of a self-attracting random walk}.
\newblock \emph{Ann. Inst. Henri Poincar{\'e}, Probab. Stat.}, 55\penalty0
  (2):\penalty0 835 -- 861, 2019.

\bibitem[Betz and Ueltschi(2009)]{BU09}
V.~Betz and D.~Ueltschi.
\newblock Spatial random permutations and infinite cycles.
\newblock \emph{Comm. Math. Phys.}, 285\penalty0 (2):\penalty0 469--501, 2009.

\bibitem[Chatterjee and Harel(2020)]{CH20}
S.~Chatterjee and M.~Harel.
\newblock {Localization in random geometric graphs with too many edges}.
\newblock \emph{Ann. Probab.}, 48\penalty0 (2):\penalty0 574 -- 621, 2020.

\bibitem[Combes(2015)]{combes2015extension}
R.~Combes.
\newblock An extension of {McDiarmid's} inequality.
\newblock \emph{Preprint arXiv:1511.05240}, 2015.

\bibitem[Deijfen et~al.(2013)Deijfen, van~der Hofstad, and Hooghiemstra]{DHH13}
M.~Deijfen, R.~van~der Hofstad, and G.~Hooghiemstra.
\newblock {Scale-free percolation}.
\newblock \emph{Ann. Inst. Henri Poincar{\'e}, Probab. Stat.}, 49:\penalty0 817
  -- 838, 2013.

\bibitem[Deprez and W{\"u}thrich(2019)]{DW19}
P.~Deprez and M.~V. W{\"u}thrich.
\newblock Scale-free percolation in continuum space.
\newblock \emph{Communications in Mathematics and Statistics}, 7\penalty0
  (3):\penalty0 269--308, 2019.

\bibitem[Dickson and Vogel(2021)]{DV22}
M.~Dickson and Q.~Vogel.
\newblock Formation of infinite loops for an interacting bosonic loop soup.
\newblock \emph{arXiv preprint arXiv:2109.01409}, 2021.
\newblock \doi{10.48550/arXiv.2109.01409}.

\bibitem[Evans and Majumdar(2008)]{EM08}
M.~R. Evans and S.~N. Majumdar.
\newblock Condensation and extreme value statistics.
\newblock \emph{J. Stat. Mech.}, 2008\penalty0 (05):\penalty0 P05004, 2008.

\bibitem[Gracar et~al.(2019)Gracar, Grauer, L{\"u}chtrath, and
  M{\"o}rters]{GGLM}
P.~Gracar, A.~Grauer, L.~L{\"u}chtrath, and P.~M{\"o}rters.
\newblock The age-dependent random connection model.
\newblock \emph{Queueing Systems}, 93\penalty0 (3):\penalty0 309--331, 2019.

\bibitem[Gracar et~al.(2021)Gracar, Lüchtrath, and Mörters]{GLM}
P~Gracar, L~Lüchtrath, and P~Mörters.
\newblock Percolation phase transition in weight-dependent random connection
  models.
\newblock \emph{Adv. Appl. Probab.}, 53\penalty0 (4):\penalty0 1090–1114,
  2021.

\bibitem[Gracar et~al.(2022)Gracar, Heydenreich, M{\"o}nch, and
  M{\"o}rters]{GHMM}
P.~Gracar, M.~Heydenreich, C.~M{\"o}nch, and P.~M{\"o}rters.
\newblock {Recurrence versus transience for weight-dependent random connection
  models}.
\newblock \emph{Elect. J. Probab.}, 27:\penalty0 1--31, 2022.

\bibitem[{\swap{Hofstad}{~van~der~}}(2024)]{RGCN_2}
R~{\swap{Hofstad}{~van~der~}}.
\newblock \emph{Random Graphs and Complex Networks, Volume~2}.
\newblock Cambridge University Press, 2024.

\bibitem[{\swap{Hofstad}{~van~der~}}
  et~al.(2023{\natexlab{a}}){\swap{Hofstad}{~van~der~}}, van~der Hoorn, and
  Maitra]{LWC_SIRGs_2020}
R~{\swap{Hofstad}{~van~der~}}, P~van~der Hoorn, and N~Maitra.
\newblock Local limits of spatial inhomogeneous random graphs.
\newblock \emph{Adv. Appl. Probab.}, page 1–48, 2023{\natexlab{a}}.

\bibitem[{\swap{Hofstad}{~van~der~}}
  et~al.(2023{\natexlab{b}}){\swap{Hofstad}{~van~der~}}, van~der Hoorn, and
  Maitra]{scaling_clust_23}
R~{\swap{Hofstad}{~van~der~}}, P~van~der Hoorn, and N~Maitra.
\newblock Scaling of the clustering function in spatial inhomogeneous random
  graphs.
\newblock \emph{J. Stat. Phys.}, 190\penalty0 (6):\penalty0 110,
  2023{\natexlab{b}}.

\bibitem[Janson(2012)]{J12}
S.~Janson.
\newblock {Simply generated trees, conditioned Galton–Watson trees, random
  allocations and condensation}.
\newblock \emph{Probability Surveys}, 9:\penalty0 103 -- 252, 2012.

\bibitem[Janson et~al.(2000)Janson, {\L}uczak, and Rucinski]{JanLucRuc00}
S.~Janson, T.~{\L}uczak, and A.~Rucinski.
\newblock \emph{Random graphs}.
\newblock Wiley-Interscience Series in Discrete Mathematics and Optimization.
  Wiley-Interscience, New York, 2000.

\bibitem[Jorritsma et~al.(2023)Jorritsma, Komjáthy, and Mitsche]{JKM}
J.~Jorritsma, J.~Komjáthy, and D.~Mitsche.
\newblock Cluster-size decay in supercritical kernel-based spatial random
  graphs, 2023.

\bibitem[Kerriou and Mörters(2022)]{KM23}
C.~Kerriou and P.~Mörters.
\newblock The fewest-big-jumps principle and an application to random graphs.
\newblock \emph{Preprint arXiv:2206.14627}, 2022.

\bibitem[Last and Penrose(2018)]{Last_Penrose_LPP}
G.~Last and M.~Penrose.
\newblock \emph{Lectures on the {P}oisson process}.
\newblock Cambridge University Press, 2018.

\bibitem[Malyshev and Yakovlev(1996)]{MY96}
V.~A. Malyshev and A.~V. Yakovlev.
\newblock {Condensation in large closed Jackson networks}.
\newblock \emph{Ann. Appl. Probab.}, 6\penalty0 (1):\penalty0 92 -- 115, 1996.

\bibitem[McDiarmid(1989)]{mcdiarmid_1989}
C.~McDiarmid.
\newblock \emph{On the method of bounded differences}, page 148–188.
\newblock London Mathematical Society Lecture Note Series. Cambridge University
  Press, 1989.
\newblock \doi{10.1017/CBO9781107359949.008}.

\bibitem[Quitmann and Taggi(2023)]{QT23}
A.~Quitmann and L.~Taggi.
\newblock Macroscopic loops in the {B}ose gas, spin {O(N)} and related models.
\newblock \emph{Comm. Math. Phys.}, 400\penalty0 (3):\penalty0 2081--2136,
  2023.

\bibitem[Rafferty et~al.(2018)Rafferty, Chleboun, and Grosskinsky]{RCG18}
T.~Rafferty, P.~Chleboun, and S.~Grosskinsky.
\newblock Monotonicity and condensation in homogeneous stochastic particle
  systems.
\newblock \emph{Ann. Inst. Henri Poincar{\'e}, Probab. Stat.}, 54\penalty0
  (2):\penalty0 790--818, 2018.

\bibitem[Smith and Majumdar(2022)]{SM22}
N.~R. Smith and S.~N. Majumdar.
\newblock Condensation transition in large deviations of self-similar gaussian
  processes with stochastic resetting.
\newblock \emph{J. Stat. Mech.}, 2022\penalty0 (5):\penalty0 053212, 2022.

\bibitem[Stegehuis and Zwart(2023)]{SZ23}
C.~Stegehuis and B.~Zwart.
\newblock {Scale-free graphs with many edges}.
\newblock \emph{Elect. Comm. Probab.}, 28:\penalty0 1 -- 11, 2023.
\newblock \doi{10.1214/23-ECP567}.
\newblock URL \url{https://doi.org/10.1214/23-ECP567}.

\bibitem[Stoyan et~al.(2013)Stoyan, Kendall, and Mecke]{Sto13}
D.~Stoyan, W.~Kendall, and J.~Mecke.
\newblock \emph{Stochastic Geometry and its Applications}.
\newblock Wiley, 2013.

\bibitem[Ueltschi(2006)]{Ue06}
D.~Ueltschi.
\newblock {Feynman cycles in the Bose gas}.
\newblock \emph{J. Math. Phys.}, 47\penalty0 (12):\penalty0 123303, 2006.

\end{thebibliography}

%Use this for bibliography within the main .tex file
%\begin{thebibliography}{34}
%\providecommand{\natexlab}[1]{#1}
%\providecommand{\url}[1]{\texttt{#1}}
%\expandafter\ifx\csname urlstyle\endcsname\relax
%  \providecommand{\doi}[1]{doi: #1}\else
%  \providecommand{\doi}{doi: \begingroup \urlstyle{rm}\Url}\fi
%  
%\bibitem[Schaeffer(1998)]{Schaeffer1998Conjugaison}
%Gilles Schaeffer.
%\newblock \emph{Conjugaison d'arbres et cartes combinatoires al{\'e}atoires}.
%\newblock PhD thesis, Universit{\'e} Bordeaux~I, 1998.
%
%\end{thebibliography}
%\normalsize

\appendix

\end{document}